\documentclass[a4paper,12pt
]{article}  %
\usepackage{hyperref} %
\usepackage{fancyhdr,subfigure, pstricks} %
\usepackage{amssymb,amsmath,amsthm} %
\usepackage[dvips]{epsfig}
\usepackage[dvips]{graphicx}
\usepackage{pst-all}
\usepackage{multido}
\usepackage{oldgerm}
\usepackage[all]{xy}
\let\SAVEDRightarrow=\Rightarrow
\usepackage{marvosym}
\let\Rightarrow=\SAVEDRightarrow
\usepackage{pifont}
\usepackage{makeidx}
\usepackage{setspace} %
\makeatletter

\newgray{gray}{0.50}
\newrgbcolor{royalblue}{0.25 0.41 0.88}
\newrgbcolor{green}{0.1 0.8 0.2}
\newrgbcolor{lightgreen}{0 1 0}
\newrgbcolor{ygreen}{0.7 1 0.7}
\newrgbcolor{lblue}{0.8 0.83 1}
\theoremstyle{plain}
\newtheorem{satz}{Satz}[section]%
\newtheorem{lemma}[satz]{Lemma}

\newtheorem{proposition}[satz]{Proposition}
\newtheorem{theorem}[satz]{Theorem}
\newtheorem{corollary}[satz]{Corollary}
\newtheorem{conj}[satz]{Conjecture}

\theoremstyle{definition}
\newtheorem{defi}[satz]{Definition}
\newtheorem{remark}[satz]{Remark}
\newtheorem{ex}[satz]{Example}

\theoremstyle{remark}

\DeclareMathOperator{\id}{id}

\DeclareMathOperator{\im}{Im}

\DeclareMathOperator{\Hom}{Hom}

\DeclareMathOperator{\pr}{pr}
\DeclareMathOperator{\supp}{Supp}
\DeclareMathOperator{\Ind}{Ind}
\DeclareMathOperator{\Coind}{Coind}

\DeclareMathOperator{\cusp}{cusp}
\DeclareMathOperator{\Eis}{Eis}

\DeclareMathOperator{\Gal}{Gal}
\DeclareMathOperator{\into}{\hookrightarrow}
\newcommand{\onto}{\twoheadrightarrow}

\DeclareMathOperator{\res}{res}
\DeclareMathOperator{\coinf}{coinf}
\DeclareMathOperator{\coker}{coker}

\DeclareMathOperator{\ord}{ord}

\DeclareMathOperator{\GL}{GL}
\DeclareMathOperator{\SU}{SU}
\DeclareMathOperator{\SL}{SL}
\DeclareMathOperator{\SO}{SO}

\DeclareMathOperator{\Dist}{Dist}
\DeclareMathOperator{\Div}{Div}
\DeclareMathOperator{\harm}{harm}

\renewcommand{\bar}{\overline}

\newcommand{\Ophi}{\phi}
\renewcommand{\.}{\cdot}

\renewcommand{\epsilon}{\varepsilon}
\renewcommand{\rho}{\varrho}
\renewcommand{\emptyset}{\varnothing}

\newcommand{\NN}{\mathbb{N}}
\newcommand{\ZZ}{\mathbb{Z}}
\newcommand{\QQ}{\mathbb{Q}}
\newcommand{\RR}{\mathbb{R}}
\newcommand{\CC}{\mathbb{C}}
\newcommand{\KK}{k}%
\newcommand{\PP}{\mathbb{P}}

\renewcommand{\AA}{\mathbb{A}}
\newcommand{\II}{\mathbb{I}}
\newcommand{\dd}{\partial} %
\newcommand{\minus}{-}%

\newcommand{\p}{\mathfrak{p}}
\newcommand{\f}{\mathfrak{f}}
\newcommand{\Co}{\mathfrak{Co}}
\newcommand{\cc}{\mathfrak{c}}
\newcommand{\m}{\mathfrak{m}}
\newcommand{\M}{\mathcal{M}}
\newcommand{\N}{\mathcal{N}}
\newcommand{\F}{\mathcal{F}}

\newcommand{\Fsemilocal}{{\underline{F}}}

\newcommand{\A}{\mathcal{A}}
\renewcommand{\H}{\mathcal{H}}
\newcommand{\GG}{\mathcal{G}}
\newcommand{\WW}{\mathcal{W}}
\newcommand{\OO}{\mathcal{O}}
\newcommand{\C}{\mathcal{C}}
\newcommand{\D}{\mathcal{D}}
\newcommand{\LL}{\mathcal{L}}
\newcommand{\XX}{\mathcal{X}}

\newcommand{\T}{\mathcal{T}}
\newcommand{\V}{\mathcal{V}}
\newcommand{\E}{\mathcal{E}}
\newcommand{\TT}{\tilde{\mathcal{T}}}
\newcommand{\VV}{\tilde{\mathcal{V}}}
\newcommand{\EE}{\tilde{\mathcal{E}}}
\newcommand{\B}{\mathcal{B}}

\newcommand{\R}{{\mathcal{R}}}%
\newcommand{\muu}{\nu}
\newcommand{\aaa}{a'}%

\renewcommand{\Re}{\operatorname{Re}}  %
\renewcommand{\Im}{\operatorname{Im}}
\newcommand{\rr}{n}
\newcommand{\nn}{k}

\newcommand{\CCoo}{\mathfrak{C}\mathfrak{o}}

\newcommand{\Qp}{{\mathbb{Q}_p}}

\newcommand{\subeq}{\subseteq}
\newcommand{\supeq}{\supseteq}
\newcommand{\iso}{\cong}
\newcommand{\thus}{\Rightarrow}

\newcommand{\einhalb}{{\tfrac{1}{2}}}  %
\newcommand{\einh}{{\genfrac{}{}{}{2}{1}{2}}} %

\newcommand{\tensor}{\otimes}
\newcommand{\omeg}{\varpi} %

\newcommand{\alphaeinszwei}{{\underline{\alpha_1},\underline{\alpha_2}}}
\newcommand{\ZC}{\chi_Z}     %
\newcommand{\ul}[1]{{\underline{#1}}}
\newcommand{\bee}{\flat}
\newcommand{\elll}{{\tilde{\ell}}}

\newcommand{\Whochi}{{\langle W \rangle}}

\newcommand{\<}{\langle}
\renewcommand{\>}{\rangle}
\newcommand{\bilinear}{\mbox{$\<\mbox{-}\,,\mbox{-}\>$}}
\numberwithin{equation}{section}
\DeclareFontFamily{U}{schwell}{}
\DeclareFontShape{U}{schwell}{m}{n}{
   <8> <9> <10> <10.95> <12> <14.4> <17.28>  <20.74> <24.88> schwell}{}
\DeclareMathAlphabet{\schwell}{U}{schwell}{m}{n}
\newcommand\textschwell{\usefont{U}{schwell}{m}{n}}
\DeclareTextFontCommand{\schwell}{\textschwell}
\DeclareFontFamily{U}{suet}{}
\DeclareFontShape{U}{suet}{m}{n}{
   <8> <9> <10> <10.95> <12> <14.4> <17.28>  <20.74> <24.88> suet14}{}
\DeclareMathAlphabet{\suet}{U}{suet}{m}{n}
\newcommand\textsuet{\usefont{U}{suet}{m}{n}}
\DeclareTextFontCommand{\suet}{\textsuet}
\DeclareMathAlphabet{\dis}{T1}{cmss}{bx}{sl}
\DeclareMathAlphabet{\sfb}{OT1}{cmss}{bx}{n}
\input{cyracc.def}
\newfont{\cyrfnt}{wncyr10}

\newfont{\cybfnt}{wncyb10}

\newfont{\cyifnt}{wncyi10}

\newfont{\cyscfnt}{wncysc10}

\newfont{\cyssfnt}{wncyss10}

\makeindex

\title{\bf $\boldsymbol{p}$-adic L-functions of automorphic forms and exceptional zeros}    %
\author{Holger Deppe}%
\date{ }

\begin{document}
\maketitle

\begin{abstract}
{Let $F$ be a number field, $p$ a prime number. We construct the (multi-variable) $p$-adic L-function of an automorphic representation of $GL_2(\AA_F)$ (with certain conditions at places above $p$ and $\infty$), which interpolates the complex (Jacquet-Langlands) L-function at the central critical point. We use this construction to prove that the $p$-adic L-function of a modular elliptic curve $E$ over $F$ has vanishing order greater or equal to the number of primes above $p$ at which $E$ has split multiplicative reduction, as predicted by the exceptional zero conjecture.

This is a generalization of analogous results by \href{http://arxiv.org/abs/1207.2289}{Spie{\ss}} over totally real fields.}

\end{abstract}

%

\tableofcontents	%
\pagestyle{plain}

%
%

%
%

%
%
%

%

%

%

%
\section*{Introduction}
\addcontentsline{toc}{section}{Introduction}

Let $F$ be a number field (with adele ring $\AA_F$), and $p$ a prime number. Let $\pi=\bigotimes_{v}\pi_v$ be an automorphic representation of $\GL_2(\AA_F)$. %
Attached to $\pi$ is the automorphic L-function $L(s,\pi)$, for $s\in\CC$, of Jacquet-Langlands \cite{JL}. %
Under certain conditions on $\pi$, we can also define a $p$-adic L-function $L_p(s,\pi)$ of $\pi$, with $s\in \ZZ_p$.
It is related to $L(s,\pi)$ by the {\it interpolation property}: For every character $\chi:\GG_p\to\CC^*$ of finite order we have
\[L_p(0,\pi\otimes \chi)=%
\tau(\chi)\prod_{\p|p} e(\pi_\p,\chi_\p)\cdot  L(\einhalb,\pi\tensor\chi)  ,\]
where $e(\pi_\p,\chi_\p )$ is a certain Euler factor (see theorem \ref{interpolation} for its definition) and $\tau(\chi)$ is the Gauss sum of $\chi$.

$L_p(s,\pi)$ was defined by Haran \cite{Haran} in the case where $\pi$ has trivial central character and $\pi_\p$ is an ordinary spherical principal series representation for all $\p|p$. For a totally real field $F$, Spie{\ss} \cite{Sp} has given a new construction of $L_p(s,\pi)$ that also allows for $\pi_\p$ to be a special (Steinberg) representation for some $\p|p$.\\ %

Here, we generalize Spie{\ss}' construction of $L_p(s,\pi)$ to automorphic representations $\pi$ of $\GL_2$ over any number field, with arbitrary central character, and use it to prove a part of the exceptional zero conjecture on $p$-adic L-functions of elliptic curves (see below%
). For $F$ not totally real, $L_p$ can naturally be defined as a multi-variable function due to the existence of several $\ZZ_p$-extensions. %

As in \cite{Sp}, we assume that $\pi$ is ordinary at all primes $\p|p$ (cf. definition \ref{ordinary}), that $\pi_v$ is discrete of weight 2 at all real infinite places $v$, and a similar condition at the complex places.

Throughout most of this paper, we follow \cite{Sp}; for section \ref{upper}, we follow Bygott \cite{By}, Ch. 4.2, who in turn follows Weil \cite{Weil}.\\

We define the $p$-adic L-function of $\pi$ as an integral, with respect to a certain measure $\mu_\pi$, on the Galois group $\GG_p$ of the maximal abelian extension that is unramified outside $p$ and $\infty$,
specifically
\[L_p(\ul{s},\pi):= L_p(s_1,\ldots,s_t,\kappa_{\pi}):= \int_{\GG_p}\prod_{i=1}^t\exp_p(s_i\ell_i(\gamma)) \mu_\pi(d\gamma)\]
(for $s_1,\ldots ,s_t\in\ZZ_p$), where $\kappa_\pi$ is a %
cohomology class attached to $\pi$ and the $\ell_i$ are $\ZZ_p$-valued homomorphisms corresponding to the $t$ independent $\ZZ_p$-extensions of $F$ %
(cf. section \ref{result} for their definition). \\

Heuristically, $\mu_\pi$ is the image of $\mu_{\pi_\p}\times W^p\left(\begin{smallmatrix}
                                                                                                                          x^p & 0\\ 0& 1
                                                                                                                         \end{smallmatrix}\right) d^\times x^p$ under the reciprocity map $\II_F=F_p^*\times \II^p\to \GG_p$ of global class field theory. 
Here $\mu_{\pi_p}=\prod_{\p|p}\mu_{\pi_\p}$ is the product of certain local distributions $\mu_{\pi_\p}$ on $F_\p^*$ attached to $\pi_\p$%
, $d^\times x^p$ is the Haar measure on the group $ \II^p=\prod'_{v\nmid p} F_v^*$ of $p$-ideles, and $W^p=\prod_{v\nmid p}W_v$ is a specific Whittaker function of $\pi^p:=\otimes_{v\nmid p} \pi_v$.\\

The structure of this work is the following:
In chapter \ref{local}, we describe the local distributions $\mu_{\pi_\p}$ on $F_\p^*$; they are the image of a Whittaker functional under a map $\delta$ on the dual of $\pi_\p$. For constructing $\delta$, we describe %
$\pi_\p$ in terms of what we call the ``Bruhat-Tits graph'' of $F_\p^2$: the directed graph whose  vertices (resp. edges) are the lattices of $F_\p^2$ (resp. inclusions between lattices). Roughly speaking, it is a covering of the (directed) Bruhat-Tits tree of $\GL_2(F_\p)$ with fibres $\iso\ZZ$. When $\pi_\p$ is the Steinberg representation, $\mu_\p$ can actually be extended to all of $F_\p$.%

In chapter \ref{cohom}, we attach a $p$-adic distribution $\mu_\phi$ to any map $\phi(U,x^p)$ of an open compact subset $U\subeq F_p^*:=\prod_{\p|p} F_\p^*$ and an %
idele  $x^p\in \II^p$ (satisfying certain conditions). Integrating $\phi$ over all the infinite places, we get a cohomology class $\kappa_\phi\in H^d({F^*}', \D_f(\CC))$ (where $d=r+s-1$ is the rank of the group of units of $F$, ${F^*}'\iso F^*/\mu_F$ is a maximal torsion-free subgroup of $F^*$, and $\D_f(\CC)$ is a space of distributions on the finite ideles of $F$). We show that $\mu_\phi$ can be described solely in terms of $\kappa_\phi$, and $\mu_\phi$ is a (vector-valued) $p$-adic measure if $\kappa_\phi$ is ``integral'', i.e. if it lies in the image of $H^d({F^*}', \D_f(R))$, for a Dedekind ring $R$ consisting of ``$p$-adic integers''. %

In chapter \ref{L-functions}, we define a map $\Ophi_\pi$ %
by
\[\Ophi_\pi(U,x^p):=\sum_{\zeta\in F^*}\mu_{\pi_p}(\zeta U)W^p\begin{pmatrix}\zeta x^p&0\\0&1\end{pmatrix} \]
($U\subeq F_p^*$ compact open, $x^p\in \II^p$).  $\Ophi_\pi$ satisfies the conditions of chapter \ref{cohom}, and we show that $\kappa_\pi:=\kappa_{\Ophi_\pi}$ is integral by ``lifting'' the map $\Ophi_\pi\mapsto \kappa_{\pi}$ to a function mapping an automorphic form to a cohomology class in  $H^d(\GL_2(F)^+, \A_f)$, for a certain space of functions $\A_f$. (Here $\GL_2(F)^+$ is the subgroup of $M\in\GL_2(F)$ with totally positive determinant.) For this, we associate to each automorphic form $\varphi$ a harmonic form $\omega_\varphi$ on a generalized upper-half space $\H_\infty$, which we can integrate between any two cusps in $\PP^1(F)$.

Then we can define the $p$-adic L-function $L_p(\ul{s},\pi):=L_p(\ul{s},\kappa_{\pi})$ %
as above, with $\mu_\pi:=\mu_{\phi_\pi}$. By a result of Harder \cite{Ha}, $H^d(\GL_2(F)^+,\A_f)_\pi$ is one-dimensional, which implies that $L_p(\ul{s},\pi)$ has values in a one-dimensional $\CC_p$-vector space.\\

We use our construction to prove the following result on the vanishing order of 
$p$-adic L-functions of elliptic curves:

If $E$ is a modular elliptic curve over $F$ corresponding to $\pi$ (i.e. %
the local L-factors of the Hasse-Weil L-function $L(E,s)$ and of the automorphic L-function $L(s-\einhalb,\pi)$ coincide at %
all places $v$ of $F$), we define the (multi-variable) $p$-adic L-function of $E$ as $L_p(E,\ul{s}):=L_p(\ul{s},\pi)$. %
The condition that $\pi$ be ordinary at all $\p|p$ means that $E$ must have good ordinary or multiplicative reduction at all places $\p|p$ of $F$. %

The {\it exceptional zero conjecture} (formulated by Mazur, Tate and Teitelbaum \cite{MTT} for $F=\QQ$, and by Hida \cite{Hida} for totally real $F$) states that 
\begin{equation}\label{1}
 \ord_{s=0}L_p(E,s)\ge n, 
\end{equation}
where $n$ is the number of $\p|p$ at which $E$ has split multiplicative reduction, and gives an explicit formula for the value of the $n$-th derivative $L_p^{(n)}(E,0)$ %
as a multiple of certain L-invariants times $L(E,1)$.
The conjecture was proved in the case $F=\QQ$ by Greenberg and Stevens \cite{GS} and independently by Kato, Kurihara and Tsuji, and for totally real fields $F$ by Spie{\ss} \cite{Sp}. 

In section \ref{result}, we formulate the exceptional zero conjecture and prove \eqref{1} for all number fields $F$.\\[+3ex] %

{\it Acknowledgements.} This paper is based on my Ph.D. thesis ``${p}$-adic L-functions of automorphic forms'' \cite{De},    %
submitted at Bielefeld University in August 2013.

I would like to thank Michael Spie{\ss} for suggesting and advising the thesis%
, and for many helpful discussions. 
I am also thankful to Werner Hoffmann for a useful discussion, and to the CRC 701, `Spectral Structures and Topological Methods in Mathematics', %
for providing financial support during most of my studies. %

\clearpage
\section{Preliminaries}
Let $\XX$ be a totally disconnected locally compact topological space, $R$ a topological Hausdorff ring. We denote by $C(\XX,R)$ the ring of continuous maps $\XX\to R$, and  let $C_c(\XX,R)\subeq C(\XX,R)$ be the subring of compactly supported maps.
When $R$ has the discrete topology, we also write $C^0(\XX,R):=C(\XX,R)$, $C_c^0(\XX,R):=C_c(\XX,R)$.\\ %

We denote by $\CCoo(\XX)$ the set of all compact open subsets of $\XX$, and for an $R$-module $M$ we denote by $\Dist(\XX,M)$ the $R$-module of $M$-valued distributions on $\XX$, i.e. the set of %
maps $\mu:\CCoo(\XX)\to M$ such that $\mu(\bigcup_{i=1}^n U_i)=\sum_{i=1}^n \mu(U_i)$ for any pairwise disjoint sets $U_i\in\CCoo(\XX)$.\\

For an open set $H\subeq \XX$, we denote by $1_H\in C(\XX,R)$ the $R$-valued indicator function of $H$ on $\XX$.\\

Throughout this paper, we fix a prime $p$ and embeddings $\iota_\infty:\bar{\QQ}\into\CC$, $\iota_p:\bar{\QQ}\into\CC_p$. Let $\bar{\OO}$ denote the valuation ring of $\bar{\QQ}$ with respect to the $p$-adic valuation induced by $\iota_p$. \\

We write  $G:=\GL_2$ throughout the thesis, and let $B$ denote the Borel subgroup of upper triangular matrices, $T$ the maximal torus (consisting of all diagonal matrices), and $Z$ the center of $G$.\\ %

For a number field $F$, we let $G(F)^+\subeq G(F)$ and $B(F)^+\subeq B(F)$ denote the corresponding subgroups of matrices with totally positive determinant, i.e. $\sigma(\det(g))$ is positive for each real embedding $\sigma:F\into \RR$. (If $F$ is totally complex, this is an empty condition, so we have $G(F)^+=G(F)$, $B(F)^+=B(F)$ in this case.) Similarly, we define $G(\RR)^+$ and $G(\CC)^+=G(\CC)$. \\

\subsection{$p$-adic measures}
\begin{defi}
 Let $\XX$ be a compact totally disconnected topological space. For a distribution $\mu:\CCoo(\XX)\to\CC$, consider the extension of $\mu$ to the $\CC_p$-linear map $C^0(\XX,\CC_p)\to \CC_p\otimes_{\bar{\QQ}}\CC$, $f\mapsto\int f d\mu$. If its image is a finitely-generated $\CC_p$-vector space, $\mu$ is called a {\it $p$-adic  measure}. 
\end{defi}
We denote the space of $p$-adic measures on $\XX$ by $\Dist^b(\XX,\CC)\subeq \Dist(\XX,\CC)$.
It is easily seen that $\mu$ is a $p$-adic measure if and only if the image of $\mu$, considered as a map $C^0(\XX,\ZZ)\to\CC$, is contained in a finitely generated $\bar{\OO}$-module. %
A $p$-adic measure can be integrated against any continuous function $f\in C(\XX,\CC_p)$.

\clearpage

\section{Local results}
\label{local}

For this chapter, let $F$ be a finite extension of $\Qp$, $\OO_F$ its ring of integers,  $\omeg$ its uniformizer and $\p=(\varpi)$ the maximal ideal. Let $q$ be the cardinality of $\OO_F/\p$, and set $U:=U^{(0)}:=\OO_F^\times$, $U^{(n)}:=1+\p^n\subeq U$ for $n\ge 1$.

We fix an additive character $\psi:F\to\bar{\QQ}^*$ with $\ker\psi\supeq\OO_F$ and $\p^{-1}\not\subeq \ker\psi$.\footnote{So $\psi(\p^{-1})$ is the set of all ${p^e}$-th roots of unity in $\bar{\QQ}^*$, where $e$ is the ramification index of $F|\QQ_p$. There is in general no $\psi$ such that $\ker(\psi)=\OO_F$, since $\p^{-1}/\OO_F$ has more than $p$ points of order $p$ if $F|\QQ_p$ has inertia index $>1$.} %
We let  $|\cdot|$ be the absolute value on $F^*$ (normalized by $|\varpi|=q^{-1}$), $\ord=\ord_\omeg$ the additive valuation, and $dx$ the Haar measure on $F$ normalized by $\int_{\OO_F} dx=1$. We define a (Haar) measure on $F^*$ by  $d^\times x := \frac{q}{q-1}\frac{dx}{|x|}$ (so $\int_{\OO_F^\times}d^\times x=1$).\\

\subsection{Gauss sums}
Recall that the {\it conductor} of a character $\chi:F^*\to\CC^*$ is by definition the largest ideal $\p^n$, $n\ge 0$, such that $\ker \chi \supeq U^{(n)}$, and that $\chi$ is {\it unramified} if its conductor is $\p^0=\OO_F$.
%
%
%

\begin{defi}
Let $\chi:F^*\to\CC^*$ be a quasi-character with conductor $\p^f$. The {\it Gauss sum} of $\chi$ (with respect to $\psi$) is  defined by 
\[\tau(\chi):=[U:U^{(f)}]\int_{\varpi^{-f} U}\psi(x)\chi(x)d^\times x.\]
\end{defi}

For a locally constant function $g:F^*\to\CC$, we define \[\int_{F^*}g(x)dx:=\lim_{n\to\infty}\int_{x\in F^{*}, -n\le \ord(x)\le n} g(x) dx,\] whenever that limit exists.
Then we have the following lemma of \cite{Sp}%
:

\begin{lemma}\it
\label{Lemma2.4}
Let $\chi:F^*\to\CC^*$ be a quasi-character with conductor $\p^f$. For $f=0$, assume $|\chi(\varpi)|<q$. Then we have
\begin{eqnarray*}\int_{F^*}\chi(x)\psi(x)dx=\begin{cases} \frac{1-\chi(\varpi)^{-1}}{1-\chi(\varpi)q^{-1}}  &\mbox{  if } f=0\\
                                            \tau(\chi)   &\mbox{  if } f>0 .
					      \end{cases}    \end{eqnarray*}
\end{lemma}
(Cf. \cite{Sp}, lemma 3.4.) 
%
%
%
%
%
%
%

%
\subsection{Tamely ramified representations of $\GL_2(F)$}\label{tamely}
For an ideal $\mathfrak{a}\subset \OO_F$, let $K_0(\mathfrak{a})\subeq G(\OO_F)$ be the subgroup of matrices congruent to an upper triangular matrix modulo $\mathfrak{a}$.\\

Let $\pi:\GL_2(F)\to\GL(V)$ be an irreducible admissible infinite-dimensional representation on a $\CC$-vector space $V$, with %
central quasicharacter $\chi$.
It is well-known (e.g \cite{Gelbart}, Thm. 4.24) that there exists a maximal ideal $\mathfrak{c}(\pi)=\mathfrak{c}\subset \OO_F$, the {\it conductor} of $\pi$, such that the space $V^{K_0(\mathfrak{c}),\chi}=\{v\in V|\pi(g) v=\chi(a)v\;\forall g=\bigl( \begin{smallmatrix} a&b\\c&d \end{smallmatrix} \bigr)\in K_0(\mathfrak{c})\}$ is non-zero (and in fact one-dimensional). A representation $\pi$ is called {\it tamely ramified} if its conductor divides $\p$.

If $\pi$ is tamely ramified, then $\pi$ is the spherical resp. special representation $\pi(\chi_1,\chi_2)$ (in the notation of \cite{Gelbart} or \cite{Sp}):  %

If the conductor is $\OO_F$, $\pi$ is (by definition) spherical and thus a principal series representation $\pi(\chi_1,\chi_2)$ for two unramified quasi-characters $\chi_1$ and $\chi_2$ with $\chi_1\chi_2^{-1}\ne |\cdot |^{\pm 1}$ (\cite{Bump}, Thm. 4.6.4).

If the conductor is $\p$, then $\pi=\pi(\chi_1,\chi_2)$  with $\chi_1\chi_2^{-1}=|\cdot |^{\pm 1}$.\\

For  $\alpha\in\CC^*$, we define a character $\chi_\alpha:F^*\to\CC^*$ by $\chi_\alpha(x):=\alpha^{\ord(x)}$. \\

So let now $\pi=\pi(\chi_1,\chi_2)$ be a %
tamely ramified irreducible admissible infinite-dimensional representation of $\GL_2(F)$; in the special case, we assume $\chi_1$ and $\chi_2$ to be ordered such that $\chi_1=| \cdot |\chi_2$.\\
Set $\alpha_i:=\nolinebreak\chi_i(\varpi)\sqrt{q}\in\CC^*$ for $i=1,2$. (We also write $\pi=\pi_{\alpha_1,\alpha_2}$ sometimes.)
Set $a:=\alpha_1+\alpha_2$, $\muu:=\alpha_1\alpha_2/q$.  %
Define a distribution $\mu_{\alpha_1,\muu}:=\mu_{\alpha_1/\muu}:=\psi(x)\chi_{\alpha_1/\muu}(x) dx$ on $F^*$.\\  %

For later use, we will need the following condition on the $\alpha_i$:
\begin{defi}\label{ordinary}
 Let $\pi=\pi_{\alpha_1,\alpha_2}$ be tamely ramified. $\pi$ is called {\it ordinary} if $a$ and $\muu$ both lie in $\bar{\OO}^*$	(i.e. they are $p$-adic units in $\bar{\QQ})$. Equivalently, this means that either $\alpha_1\in\bar{\OO}^*$ and $\alpha_2\in q\bar{\OO}^*$, or vice versa.\\
\end{defi}

\begin{proposition}\it
\label{Prop2.7} 
Let $\chi:F^*\to\CC^*$ be a quasi-character with conductor $\p^f$; for $f=0$, assume $|\chi(\varpi)|<|\alpha_2|$. Then the integral $\int_{F^*}\chi(x)\mu_{\alpha_1/\muu}(dx)$ converges and we have
\[ \int_{F^*}\chi(x)\mu_{\alpha_1/\muu}(dx)=e(\alpha_1,\alpha_2,\chi)\tau(\chi)L(\einhalb,\pi\otimes\chi),\]
where
\[\small e(\alpha_1,\alpha_2,\chi)= \begin{cases}
  \dfrac{(1-\alpha_1\chi(\varpi)q^{-1})(1-\alpha_2\chi(\varpi)^{-1}q^{-1})(1-\alpha_2 \chi(\omeg) q^{-1})}{(1-\chi(\omeg)\alpha_2^{-1})},  & f=0\text{ and }\pi\text{ spherical,}\phantom{\Biggl(\Biggr)}\\   %

  \dfrac{(1-\alpha_1\chi(\omeg)q^{-1})(1-\alpha_2\chi(\varpi)^{-1}q^{-1})}%
{(1-\chi(\omeg)\alpha_2^{-1})},  & f=0\text{ and }\pi\text{ special,}\phantom{\Biggl(\Biggr)}\\

  (\alpha_1/\muu)^{-f}=(\alpha_2/q)^f,                      &f>0,
\end{cases} \]
and where we assume the right-hand side to be continuously extended to the potential removable singularities at $\chi(\omeg)= q/\alpha_1$ or $=q/\alpha_2$.

\end{proposition}
\begin{proof}

{\it Case 1:} $f=0$, $\pi$ spherical \\
We have
\[L(s, \pi\otimes\chi)= \frac{1}{\big(1-\alpha_1\chi(\varpi) q^{-\left(s+\einh\right)}\big) \big(1-\alpha_2\chi(\varpi) q^{-\left(s+\einh\right)}\big)}   ,\]
so
\begin{eqnarray*}
L(\einhalb, \pi\otimes\chi)\cdot\tau(\chi)\cdot e(\alpha_1,\alpha_2,\chi)&=&\frac{1-\alpha_2 q^{-1}\chi(\varpi)^{-1}}{1-\chi(\varpi)\alpha_2^{-1}}  \\
								      &=&\frac{1-\muu\alpha_1^{-1}\chi(\varpi)^{-1}}{1-\alpha_1\chi(\varpi)\muu^{-1} q^{-1}}  \\
								      &=&\int_{F^*}\chi(x)\chi_{\alpha_1/\muu}(x)\psi(x) dx\\
								      &=&\int_{F^*}\chi(x)\mu_{\alpha_1/\muu}(dx)\\
				      \end{eqnarray*}
by lemma \ref{Lemma2.4}.\\

{\it Case 2:} $f=0$, $\pi$ special\\
Assuming $\chi_1=|\cdot| \chi_2$, we have
\[L(s,\pi\otimes\chi)=\frac{1}{1-\alpha_1\chi(\varpi) q^{-\left(s+\einh\right)}}\]
and thus 
\begin{eqnarray*}
L(\einhalb, \pi\otimes\chi)\cdot\tau(\chi)\cdot e(\alpha_1,\alpha_2,\chi)&=&\frac{1-\muu\alpha_1^{-1}\chi(\varpi)^{-1}}{1-\alpha_1\muu^{-1}\chi(\varpi)q^{-1}}\\
			 &=&\int_{F^*}\chi(x)\chi_{\alpha_1/\muu}(x)\psi(x) dx\mbox{     }\\
			 &=&\int_{F^*}\chi(x)\mu_{\alpha_1/\muu}(dx)   .\end{eqnarray*}
 by lemma \ref{Lemma2.4}.\\

{\it Case 3:} $f>0$\\
In this case, $L(s,\pi\otimes\chi)=1$ for $s>0$  and %
\begin{eqnarray*}
\int_{F^*}\chi(x)\mu_{\alpha_1/\muu}(dx)&=&\tau(\chi\cdot\chi_{\alpha_1/\muu})\\
				   &=&q^{f-1}(q-1)\int_{\varpi^{-f}U}\psi(x)\chi(x)\chi_{\alpha_1/\muu}(x)d^{\times}x\\
				   &=&(\alpha_1/\muu)^{-f}q^{f-1}(q-1)\int_{\varpi^{-f}U}\psi(x)\chi(x)d^{\times}x\\
				   &=&e(\alpha_1,\alpha_2,\chi)\cdot\tau(\chi)\cdot L(\einhalb, \pi\otimes\chi)   .\end{eqnarray*}
\end{proof}

\subsection{The Bruhat-Tits graph $\TT$}

Let  $\VV$ denote the set of lattices (i.e. submodules isomorphic to $\OO_F^2$) in $F^2$, and let $\EE$ be the set of %
all inclusion maps   %
between two lattices; for such a map $e:v_1\into v_2$ in $\EE$, we define $o(e):=v_1$, $t(e):=v_2$. Then the pair $\TT:=(\VV,\EE)$ is naturally a directed graph, connected, with no directed cycles (specifically, $\EE$ induces a partial ordering on $\VV$). For each $v\in\VV$, there are exactly $q+1$ edges beginning  (resp. ending) in $v$, each.\\

Recall that the {\it Bruhat-Tits tree} $\T=(\V,\vec{\E})$ of $G(F)$ is the directed graph whose vertices are homothety classes of lattices of $F^2$ (i.e. $\V=\VV/\sim$, where $v \sim \omeg^i v$ for all $i\in\ZZ$), 
and the directed edges $\bar{e}\in \vec{\E}$ are homothety classes of inclusions of lattices. We can define maps $o,t:\vec{\E}\to\V$ analogously. For each edge $\bar{e}\in\vec{\E}$, there is an opposite edge $\bar{e}'\in\vec{\E}$ with $o(\bar{e}')=t(\bar{e})$, $t(\bar{e}')=o(\bar{e})$; and the undirected graph underlying $\T$ is simply connected. We have a natural ``projection map" $\pi:\TT\to\T$, mapping each lattice and each homomorphism to its homothety class. Choosing a (set-theoretic) section $s:\V\to\VV$, we get a %
bijection $\V\times\ZZ\xrightarrow{\iso}\VV$ via $(v,i)\mapsto \omeg^i s(v)$.\\ %

The group $G(F)$ operates on $\VV$ via its standard action on $F^2$, i.e. $gv=\{ gx| x\in v\}$ for $g\in G(F)$, and on $\EE$ by mapping $e:v_1\to v_2$ to the inclusion map $ge:gv_1\to gv_2$. The stabilizer of the standard vertex $v_0:=\OO_F^2$ is $G(\OO_F)$.

For a directed edge $\bar{e}\in\vec{\E}$ of the Bruhat-Tits tree $\T$, we define $U(\bar{e})$ to be the set of ends of $\bar{e}$ (cf. \cite{SerreTrees}/\cite{Sp}); it is a compact open subset of $\PP^1(F)$, and we have $gU(\bar{e})=U(g\bar{e})$ for all $g\in G(F)$.\\

For $n\in\ZZ$, we set $v_n:=\OO_F\oplus \p^n \in\VV$, and denote by $e_n$ the edge from $v_{n+1}$ to $v_n$; %
the ``decreasing'' sequence $(\pi(e_{-n}))_{n\in\ZZ}$ is the geodesic from $\infty$ to $0$. (The geodesic from $0$ to $\infty$ traverses the $\pi(v_n)$ in the natural order of $n\in\ZZ$.) We have $U(\pi(e_n))=\p^{-n}$ for each $n$.\\   %

Now (following \cite{BL} and \cite{Sp}), we can define a "height" function $h:\V \to\ZZ$ as follows: The geodesic ray from $v\in\V$ to $\infty$ must contain some $\pi(v_n)$ ($n\in\ZZ$), since it has non-empty intersection with $A:=\{\pi(v_n)|n\in\ZZ\}$; we define $h(v):=n-d(v,\pi(v_n))$ for any such $v_n$; this is easily seen to be well-defined, and we have $h(\pi(v_n))=n$ for all $n\in\ZZ$. We have the following lemma of \cite{Sp}:\\

\begin{lemma}\label{2.6}
(a) For all $\bar{e}\in\E$, we have
\[ h(t(\bar{e})) = \begin{cases}
			h(o(\bar{e}))+ 1 & \text{if $\infty\in U(\bar{e})$,}\\
			h(o(\bar{e}))- 1 & \text{otherwise.} 
		   \end{cases}  \]   %

(b) For $a\in F^*$, $b\in F$, $\bar{v}\in\V$ we have 

\[h\left(\begin{pmatrix} a &b\\0&1\end{pmatrix} \bar{v}\right) = h(\bar{v}) -\ord_\omeg(a). \]
\end{lemma}
(Cf. \cite{Sp}, Lemma 3.6.)

%
%
%
%

%
\subsection{Hecke structure of $\TT$}

Let $R$ be a ring, $M$ an $R$-module. We let $C(\VV,M)$ be the $R$-module of maps $\phi:\VV\to M$, and $C(\EE,M)$ the $R$-module of maps $\EE\to M$. Both are $G(F)$-modules via $(g\phi)(v):=\phi(g^{-1}v)$, $(g c)(e):=c (g^{-1}e)$.

We let $\C_c(\VV,M)\subeq C(\VV,M)$ and $\C_c(\EE, M)\subeq C(\EE,M)$ be the ($G(F)$-stable) submodules of maps with compact support, i.e. maps that are zero outside a finite set. We get pairings

\begin{equation}\label{(1)}
\bilinear: C_c(\VV, R)\times C(\VV,M)\to M, \quad \<\phi_1, \phi_2\>:=\sum_{v\in\VV}\phi_1(v)\phi_2(v)
\end{equation}
and
\begin{equation}\label{(2)}
\bilinear: C_c(\EE, R)\times C(\EE,M)\to M, \quad \<c_1, c_2\>:=\sum_{e\in\EE}c_1(v)c_2(v) .\\
\end{equation}

We define Hecke operators $T, \R: \C(\VV,M)\to\C(\VV, M)$ by 
\[T\phi(v)=\sum_{t(e)=v}\phi(o(e))\quad\text{and}\quad \R\phi:=\omeg\phi \text{ (i.e. }\R\phi(v)=\phi(\omeg^{-1} v)\text{)} \]
for all $v\in\VV$. These restrict to operators on $C_c(\VV,R)$, which we sometimes denote by $T_c$ and $\R_c$ for emphasis. With respect to \eqref{(1)}, $T_c$ is adjoint to $T\R$, and $\R_c$ is adjoint to its inverse operator $\R^{-1}:\C_c(\VV,R)\to C_c(\VV,R)$. \\

$T$ and $\R$ obviously commute, and we have the following Hecke structure theorem on compactly supported functions on $\VV$ (an analogue of \cite{BL}, Thm. 10):
\begin{theorem}\label{freeHecke}
$C_c(\VV,R)$ is a free $R[T, \R^{\pm 1}]$-module (where $R[T,\R^{\pm1}]$ is the ring of Laurent series in $\R$ over the polynomial ring $R[T]$, with $\R$ and $T$ commuting).
\end{theorem}

\begin{proof}
Fix a vertex $v_0\in\VV$. For each $n\ge 0$, let $C_n$ be the set of vertices $v\in\VV$ such that there is a directed path of length $n$ from $v_0$ to $v$ in $\VV$, and such that $d(\pi(v_0),\pi(v))=n$ in the Bruhat-Tits tree $\T$. So $C_0=\{v_0\}$, and $C_n$ is a lift of the "circle of radius $n$ around $v_0$" in $\T$, in the parlance of \cite{BL}.%

One easily sees that $\bigcup_{n=0}^\infty C_n$ is a complete set of representatives for the projection map $\pi:\VV\to\V$; specifically, for $n>1$ and a given $v\in C_{n-1}$, $C_n$ contains exactly $q$ elements adjacent to $v$ in $\VV$; and we can write $\VV$ as a disjoint union $\bigcup_{j\in\ZZ}\bigcup_{n=0}^\infty \R^j (C_n)$.

We further define $V_0:=\{v_0\}$ and choose subsets $V_n\subeq C_n$ as follows: We let $V_1$ be any subset of cardinality $q$. For $n>1$, we choose $q-1$ out of the $q$ elements of $C_n$ adjacent to $v'$, for every $v'\in C_{n-1}$, and let $V_n$ be the union of these elements for all $v'\in C_{n-1}$. Finally, we set \[H_{n,j}:=\{\phi \in C_c(\VV,R)| \supp(\phi)\subeq \bigcup_{i=0}^n \R^j (C_i)\} \quad \text{for each } n\ge 0, j\in\ZZ,\]
$H_n:=\bigcup_{j\in\ZZ}H_{n,j}$, and $H_{-1}:= H_{-1,j}:=\{0\}$. (For ease of notation, we identify $v\in\VV$ with its indicator function $1_{\{v\}}\in C_c(\VV,R)$ in this proof.)

Define $T': C_c(\VV,R)\to C_c(\VV,R)$ by 
\[T'(\phi)(v):=\mspace{-18mu}\sum_{{t(e)=(v),}\atop{o(e)\in \R^j(C_n)}}\mspace{-15mu}\phi(o(e))\quad\text{  for each }v\in \R^j(C_{n-1}), j\in\ZZ;\]
 $T'$ can be seen as the "restriction to one layer" $\bigcup_{n=0}^\infty \R^j (C_n)$ of $T$. We have $T'(v)\equiv T(v) \mod H_{n-1}$ for each $v\in H_n$, since the "missing summand" of $T'$ lies in $H_{n-1}$.

We claim that for each $n\ge0$, the set $X_{n,j}:= \bigcup_{i=0}^n \R^j T^{n-i} (V_i)$ is an $R$-basis for $H_{n,j}/H_{n-1,j}$. 
By the above congruence, we can replace $T$ by $T'$ in the definition of $X_{n,j}$.

The claim is clear for $n=0$. 
So let $n\ge 1$, and assume the claim to be true for all $n'\le n$. 
For each $v\in C_{n-1}$, the $q$ points in $C_n$ adjacent to $v$ are generated by the $q-1$ of these points lying in $V_{n}$, plus %
$T'v$ %
(which just sums up these $q$ points). By induction hypothesis, $v$ is generated by $X_{n-1,0}$, and thus (taking the union over all $v$), $C_n$ is generated by $T'(X_{n-1,0})\cup V_n= X_{n,0}$. Since the cardinality of $X_{n,0}$ equals the $R$-rank of $H_{n,0}/H_{n-1,0}$ (both are equal to $(q+1)q^{n-1}$), $X_{n,0}$ is in fact an $R$-basis.

Analoguously, we see that $H_{n,j}/H_{n-1,j}$ %
has $\R^j(X_{n,0})=X_{n,j}$ as a basis, for each $j\in\ZZ$. 

From the claim, it follows that $\bigcup_{j\in\ZZ}X_{n,j}$ is an $R$-basis of $H_n/H_{n-1}$ for each $n$, and that $V:=\bigcup_{n=0}^\infty V_n$ is an $R[T, \R^{\pm 1}]$-basis of $C_c(\VV,R)$. 

\end{proof}

\phantom{X}

For $a\in R$ and $\muu\in R^*$ , we let $\tilde{\B}_{a,\muu}(F,R)$ be the "common cokernel" of $T-a$ and $\R-\muu$ in $C_c(\VV,R)$, namely $\tilde{\B}_{a,\muu}(F,R):=C_c(\VV, R)/(\Im(T-a)+\im(\R-\muu))$; dually, we define $\tilde{\B}^{a,\muu}(F,M):=\ker(T-a)\cap\ker(\R-\muu)\subeq C(\VV, M)$.\\

For a lattice $v\in\VV$, we define a valuation $\ord_v$ on $F$ as follows: For $w\in F^2$, the set $\{x\in F| xw\in v\}$ is some fractional ideal $\omeg^m\OO_F\subeq F$ ($m\in\ZZ$); we set $\ord_v(w):=m$. This map can also be given explicitly as follows:
Let $\lambda_1,\lambda_2$ be a basis of $v$. We can write any $w\in F^2$ as $w=x_1\lambda_1+x_2\lambda_2$; then we have $\ord_v(w)=\min\{\ord_\omeg(x_1),\ord_\omeg(x_2)\}$. This gives a "valuation" map on $F^2$, as one easily checks. We restrict it to $F\iso F\times\{0\}\into F^2$ to get a valuation $\ord_v$ on $F$, and consider especially the value at $e_1:=(1,0)$.\\

\begin{lemma}\label{rho}
Let $\alpha, \muu\in R^*$, and put $a:=\alpha+q\muu/\alpha$. Define a map $\rho=\rho_{\alpha,\muu}:\VV\to R$ by $\rho(v):=\alpha^{h(\pi(v))} \muu^{-\ord_v(e_1)}$.
Then $\rho\in \tilde{\B}^{a,\muu}(F,R)$. %
\end{lemma}

\begin{proof}
One easily sees that \mbox{$\left(v\mapsto\muu^{-\ord_v(e_1)}\right)\in\ker(\R-\muu)$}. It remains to show that $\rho\in\ker(T-a)$:

We have the Iwasawa decomposition $G(F)=B(F)G(\OO_F)= \lbrace \left(\begin{smallmatrix} *& * \\0&1 \end{smallmatrix}\right) \rbrace Z(F) G(\OO_F)$; thus every  vertex in $\VV$ can be written as $\omeg^i v$ with $v = \left(\begin{smallmatrix}a & b\\0 &1 \end{smallmatrix}\right) v_0$, with $i\in\ZZ$, $a\in F^*, b\in F$. 

Now the lattice $v= \left( \begin{smallmatrix}a & b\\0 &1 \end{smallmatrix} \right) v_0$ is generated by the vectors $\lambda_1=\left(\begin{smallmatrix} a\\0\end{smallmatrix}\right)$ and $\lambda_2=\left(\begin{smallmatrix}b\\1\end{smallmatrix}\right)\in\OO_F^2$, so $e_1 =a^{-1}\lambda_1$ and thus $\ord_v (e_1)=\ord_\omeg(a^{-1})=-\ord_\omeg(a)$. The $q+1$ neighbouring vertices $v'$ for which there exists an $e\in\EE$ with $o(e)=v', t(e)=v$ are given by $N_i v$, $i\in \{\infty\}\cup\OO_F/\p$, with $N_\infty:=\left( \begin{smallmatrix}1&0\\0&\omeg\end{smallmatrix} \right)$, and $N_i:=\left(\begin{smallmatrix}\omeg&i\\0&1\end{smallmatrix}\right)$ where $i\in\OO_F$ runs through a complete set of representatives $\mod\omeg$. By lemma \ref{2.6}, $h(\pi(N_\infty v))=h(\pi(v))+1$ and $h(\pi(N_i v))=h(\pi(v))-1$ for $i\ne\infty$. By considering the basis $\{N_i\lambda_1, N_i\lambda_2\}$ of $N_iv$ for each $N_i$, we see that $\ord_{N_\infty v}(e_1)=\ord_v(e_1)$ and $\ord_{N_i v}(e_1)=\ord_v(e_1)-1$ for $i\ne\infty$. 
Thus we have

\begin{eqnarray*}
(T\rho)(v)=\sum_{t(e)=v}\alpha^{h(\pi(o(e)))} \muu^{-\ord_{o(e)} (e_1)} &=& \alpha^{h(\pi(v))+1}\muu^{-\ord_v e_1} + q\cdot \alpha^{h(\pi(v))-1}\muu^{1-\ord_v(e_1)} \\
								&=& (\alpha+q\alpha^{-1}\muu) \alpha^{h(\pi(v))}\muu^{-\ord_ve_1} = a \rho(v),
\end{eqnarray*}
and also $(T\rho)(\omeg^i v)=(T\R^{-i}\rho)(v)=\R^{-i}(a\rho)(v) =a \rho(\omeg^i v)$ for a general $\omeg^i v\in\VV$, which shows that $\rho\in\ker(T-a)$. %
\end{proof}

\phantom{ X}

If $a^2 \ne \muu (q+1)^2$ (we will call this the ``spherical case''\footnote{We use this term since these pairs of $a,\muu$ will later be seen to correspond to a spherical representation of $\GL_2(F)$. The case $a^2=\muu (q+1)^2$ means that there exists an $\alpha\in R^*$ with $a=\alpha(q+1)$, $\muu=\alpha^2$, which will correspond to a special representation.}), we put $\B_{a,\muu}(F,R):=\tilde{\B}_{a,\muu}(F,R)$ and $\B^{a,\muu}(F,M):=\tilde{\B}^{a,\muu}(F,M)$.\\

In the ``special case'' $a^2=\muu (q+1)^2$, we need to assume that the polynomial ${X^2-a\muu X+q\muu^{-1}} \in R[X]$ has a zero $\alpha'\in R$. Then the map $\rho:=\rho_{\alpha',\muu}\in C(\VV,R)$ defined as above lies in 
$\tilde{\B}^{a\muu,\muu^{-1}}(F,R) = %
\ker(T\R-a)\cap\ker(\R^{-1}-\muu)$ by Lemma \ref{rho}, since $a\muu=\alpha'+q\muu^{-1}/\alpha'$. In other words, the kernel of the map ${\langle\cdot,\rho\rangle:C_c(\VV,R)\to R}$ contains $\Im(T-a)+\Im(\R-\muu)$; and we define \[\B_{a,\muu}(F,R):=\ker\left(\langle\cdot,\rho\rangle\right)/\left( \Im(T-a)+\Im(\R-\muu) \right)\]  to be the quotient; evidently, it is an $R$-submodule of codimension 1 of $\tilde{\B}_{a,\muu}(F,R)$. %
Dually, $T-a$ and $\R-\muu$ both map the submodule $\rho M=\{ \rho\cdot m, m\in M \}$ of $C(\VV,M)$ to zero and thus induce endomorphisms on $C(\VV,M)/\rho M$;  we define $\B^{a,\muu}(F,M)$ to be the intersection of their kernels. \\

In the special case, since $\muu=\alpha^2$, Lemma \ref{rho} states that $\rho(gv_0)=\chi_\alpha(ad)\rho(v_0)=\chi_{\alpha}(\det g)\rho(v_0)$ for all $g=\begin{pmatrix}a&b\\0&d\end{pmatrix}\in B(F)$, and thus for all $g\in G(F)$ by the Iwasawa decomposition, since $G(\OO_F)$ fixes $v_0$ and lies in the kernel of $\chi_\alpha\circ\det$. By the multiplicity of $\det$, we have $(g^{-1}\rho)(v)=\rho(gv)=\chi_{\alpha}(\det g)\rho(v)$ for all $g\in G(F)$, $v\in\VV$. So $\phi\in \ker\langle\cdot,\rho\rangle$ implies $\langle g\phi,\rho\rangle =\langle\phi,g^{-1}\rho\rangle=\chi_{\alpha}(\det g)\langle\phi,\rho\rangle=0$, i.e. $\ker\langle\cdot,\rho\rangle$ and thus $\B_{a,\muu}(F,R)$ are $G(F)$-modules.\\

By the adjointness properties of the Hecke operators $T$ and $\R$, %
we have pairings 
$\coker (T_c-a)\times \ker (T\R-a)\to M$
 and %
$\coker(\R_c-\muu) \times \ker (\R^{-1}-\muu)\to M$, which "combine" to give a pairing
\[\bilinear: \B_{a,\muu}(F,R)\times \B^{a\muu,\muu^{-1}}(F,M)\to M \]   %
(since $\ker(T\R-a)\cap\ker(\R^{-1}-\muu)=\ker(T-a\muu)\cap\ker(\R-\muu^{-1})$),  %
and a corresponding isomorphism $\B^{a\muu,\muu^{-1}}(F,M)\xrightarrow{\iso} \Hom(\B_{a,\muu}(F,R), M)$.\\  %

\begin{defi} 
 Let $G$ be a totally disconnected locally compact group, $H\subeq G$ an %
 open subgroup. For a smooth $R[H]$-module $M$, we define the  {\it (compactly) induced $G$-representation} of $M$, denoted $\Ind^G_H M$, to be the space of maps $f:G\to M$ such that $f(hg)=%
f(g)$ for all $g\in G, h\in H$, %
 and such that $f$ has compact support modulo $H$. We let $G$ act on $\Ind^G_H M$ via $g\cdot f(x):=f(xg)$.  %
(We can also write $\Ind^G_H M= R[G]\otimes_{R[H]}M$, cf. \cite{Brown}, III.5.)%

We further define $\Coind^G_H M:=\Hom_{R[H]}(R[G], M)$.   %
Finally, for an $R[G]$-module $N$, we write $\res^G_H N$ for its underlying $R[H]$-module (``restriction of scalars'').\\
\end{defi}

By Theorem \ref{freeHecke}, $T_c-a$ (as well as $\R_c-\muu$) is injective, and the induced map \[\R_c-\muu:\coker(T_c-a)=C_c(\VV,R)/\im(T_c-a)\to \coker(T_c-a)\] (of $R[T,\R^{\pm1}]/(T-a)= R[\R^{\pm1}]$-modules) is also injective. Now since $G(F)$ acts transitively on $\VV$, with the stabilizer of 
$v_0:=\OO_F^2$ being $K:=G(\OO_F)$, we have an isomorphism $C_c(\VV,R)\iso \Ind^{G(F)}_{K}R$. Thus we have exact sequences

\begin{equation}
 0\to \Ind^{G(F)}_{K}R\xrightarrow{T-a}\Ind^{G(F)}_{K}R\xrightarrow{}\coker(T_c-a)\xrightarrow{} 0 \end{equation}

and (for $a,\muu$ in the spherical case)

\begin{equation} 0\to \coker(T_c-a)\xrightarrow{\R-\muu} \coker(T_c-a) \to \B_{a,\muu}(F,R) \to 0, \end{equation}

with all entries being free $R$-modules. Applying $\Hom_R(\cdot,M)$ to them, we get:

\begin{lemma}\label{sphericalResolution}
We have exact sequences of $R$-modules
\[0\to\ker(T\R-a)\to\Coind^{G(F)}_{K}M \xrightarrow{T-a} %
 \Coind^{G(F)}_{K}M\to 0\]
and, if $\B_{a,\muu}(F,M)$ is spherical (i.e. $a^2 \ne \muu (q+1)^2$),
\[0\to \B^{a\muu,\muu^{-1}}(F,M)\to \ker(T\R-a)\xrightarrow{\R-\muu}%
\ker(T\R-a)\to 0. \]
\end{lemma}

\phantom{   }

For the special case%
, we have to work a bit more to get similar exact sequences:\\

 By \cite{Sp}, eq. (22), for the representation $St^-(F,R):=\B_{-(q+1),1}(F,R)$  (i.e. $\muu=1$, $\alpha=-1$) with trivial central character, we have an exact sequence of $G$-modules

\begin{equation}\label{Steinberg} 0\to\Ind^G_{KZ}R\to\Ind^G_{K'Z}R\to St^-(F,R)\to 0, \end{equation}

where $K'={\Whochi  K_0(\p)}$ is the subgroup of $KZ$ generated by %
$W:=\left(\begin{smallmatrix}0&1\\ \omeg & 0\end{smallmatrix}\right)$ and the subgroup $K_0(\p)\subeq K$ of matrices that are upper-triangular modulo $\p$. (Since $W^2\in Z$,  $K_0(\p)Z$ is a subgroup of $K'$ of order 2.) Now %
 $(\pi, V)$ can be written as $\pi=\chi\otimes St^-$ for some character $\chi=\chi_Z$ (cf. the proof of lemma \ref{representation} below), and we have an obvious $G$-isomorphism 
\[(\pi,V)\iso (\pi\otimes (\chi\circ\det), V\otimes_R R(\chi\circ\det)),\]
where $R(\chi\circ\det)$ is the ring $R$ with $G$-module structure given via $gr=\chi(\det(g))r$ for $g\in G, r\in R$. Tensoring \eqref{Steinberg} with $R(\chi\circ\det)$ over $R$ gives an exact sequence of $G$-modules

\begin{equation}\label{chi}
 0\to\Ind^G_{KZ}\chi\to\Ind^G_{K'Z}\chi\to V \to 0.
\end{equation}

It is easily seen that $R(\chi\circ\det)$ fits into another exact sequence %
of $G$-modules
\[ 0\to \Ind^G_H R\xrightarrow{\left(\begin{smallmatrix}\omeg&0\\0&1 \end{smallmatrix}\right)-\chi(\omeg)\id}\Ind^G_H R\xrightarrow{\psi} R(\chi\circ\det)\to 0,\]
where $H:=\{g\in G|\det g\in\OO_F^\times\}$ is a normal subgroup containing $K$, $\left(\begin{smallmatrix}\omeg&0\\0&1 \end{smallmatrix}\right)(f)(g):=f (\left(\begin{smallmatrix}\omeg&0\\0&1 \end{smallmatrix}\right)^{-1} g)$ for $f\in \Ind^G_H R=\{f:G\to R|f(Hg)=f(g)$ for all $g\in G\}$, $g\in G$, is the natural operation of $G$, and where $\psi$ is the $G$-equivariant map defined by $1_U\mapsto 1$.

Now %
since $H\subeq G$ is a  normal subgroup, we have $\Ind^G_H R\iso R[G/H]$            %
as $G$-modules (in fact $G/H\iso\ZZ$ as an abstract group). Let $X\subeq G$ be a subgroup such that the natural inclusion $X/(X\cap H)\into G/H$ has finite cokernel; let $g_iH$, $i=1,\ldots n$ be a set of representatives of that cokernel. Then we have a (non-canonical) $X$-isomorphism $\bigoplus_{i=0}^n \Ind^X_{X\cap H}\to \Ind^G_H R$ defined via $(1_{(X\cap H)x})_i\mapsto 1_{Hxg_i}$ for each $i=1,\ldots,n$ (cf. \cite{Brown}, III (5.4)). %

Using this isomorphism and the ``tensor identity'' $\Ind^G_H M\otimes N\iso \Ind^G_H(M\otimes \res^G_H N)$ for any groups $H\subeq G$,  $H$-module $M$ and  $G$-module $N$ (\cite{Brown} III.5, Ex. 2%
), we have 
\begin{align*} 
\Ind^G_{KZ}R\otimes_R \Ind^G_H R &\iso \Ind^G_{KZ} (\res^G_{KZ}(\Ind^G_H R))\\    &=\Ind^G_{KZ} ((\Ind^{KZ}_{KZ\cap H} R)^2) \\ & =(\Ind^G_{KZ} (\Ind^{KZ}_K R))^2 =(\Ind^G_K R)^2
\end{align*}
(since $KZ/KZ\cap H\into G/H$ has index $2$), and similarly
\begin{align*}
\Ind^G_{K'Z}R\otimes_R \Ind^G_H R &\iso \Ind^G_{K'Z}(\res^G_{K'Z}(\Ind^G_H R)) \\  	&\iso \Ind^G_{K'Z}((\Ind^{K'Z}_{K'Z\cap H} R)^2) \\   &  \iso (\Ind^G_{K'} R)^2 %
\end{align*}

and thus, we can resolve the first and second term of \eqref{chi} into exact sequences

\[ 0\to(\Ind^G_K R)^2 \to (\Ind^G_K R)^2 \to \Ind^G_{KZ}\chi \to 0, \]
\[ 0\to(\Ind^G_{K'} R)^2 \to (\Ind^G_{K'} R)^2 \to \Ind^G_{\Whochi K_0(\p)Z}\chi \to 0.\]

Dualizing \eqref{chi} and these by taking $\Hom(\cdot,M)$ for an $R$-module $M$, %
 we get a ``resolution'' of $\B^{a\muu,\muu^{-1}}(F,M)$ in terms of coinduced modules:

\begin{lemma}\label{specialResolution}
We have exact sequences
\begin{align*}
 0\to &\B^{a\muu,\muu^{-1}}(F,M)\to \Coind^G_{K'Z} M(\chi)\to \Coind^G_{KZ} M(\chi)\to 0,\\
 0\to &\Coind^G_{KZ} M(\chi) \to (\Coind^G_K R)^2\to (\Coind^G_K R)^2\to 0,\\
 0\to &\Coind^G_{K'Z} M(\chi) \to (\Coind^G_{K'} R)^2\to (\Coind^G_{K'} R)^2\to 0
\end{align*}
for all special %
$\B_{a,\muu}(F,R)$ (i.e. $a^2 = \muu (q+1)^2$), where $\chi=\chi_Z$ is the central character.%
\end{lemma}
 
It is easily seen that the above arguments could be modified to get a similar set of exact sequences in the spherical case 
as well (replacing $K'$ by $K$ everywhere), in addition to that given in lemma \ref{sphericalResolution}; but we will not need this.\\

\subsection{Distributions on $\TT$}%

For $\rho\in C(\VV,R)$ we define $R$-linear maps

\[\tilde{\delta}_\rho:C(\EE,M) \to C(\VV,M), \quad   \tilde{\delta}_\rho(c)(v):=\sum_{v=t(e)} \rho(o(e))c(e)-\sum_{v=o(e)} \rho(t(e))c(e),\]
\[\tilde{\delta}^\rho:C(\VV,M)\to C(\EE,M), \quad    \tilde{\delta}^\rho(\phi)(e):=\rho(o(e))\phi(t(e))-\rho(t(e))\phi(o(e)). \]

One easily checks %
that these are adjoint with respect to the pairings \eqref{(1)} and \eqref{(2)}, i.e. we have $\< \tilde{\delta}_\rho(c),\phi \> = \< c, \tilde{\delta}^\rho(\phi) \>$ for all $c\in C_c(\EE,R)$, $\phi\in C(\VV,M)$. We denote the maps corresponding to $\rho\equiv 1$ by $\delta:=\tilde{\delta}_1$, $\delta^*:=\tilde{\delta}^1$.\\

For each $\rho$, the map $\tilde{\delta}_\rho$ fits into an exact sequence
\[C_c(\EE,R)\xrightarrow{\tilde{\delta}_\rho} C_c(\VV,R)\xrightarrow{\langle\cdot,\rho\rangle} R \to 0\]  %

but it is not injective in general: e.g. for $\rho\equiv 1$, the map $\EE\to R$ symbolized by
\[\xymatrix{\cdot\ar[d]^{1}\ar[r]^{-1}& \cdot\ar[d]^{-1} \\
\cdot\ar[r]^{1} &\cdot}	\]
(and zero outside the square) lies in $\ker\delta$.\\

The restriction $\delta^*|_{C_c(\VV,R)}$ to compactly supported maps is injective since $\TT$ has no directed circles, and we have a surjective map 
\[\coker\big(\delta^*:C_c(\VV,R) \to C_c(\EE,R)\big) \to C^0(\PP^1(F),R)/R, \quad c\mapsto \sum_{e\in\EE} c(e) 1_{U(\pi(e))}\] 
(which corresponds to an isomorphism of the similar map on the Bruhat-Tits tree $\T$). Its kernel is generated by the functions $1_{\{e\}} -1_{\{ {e'}\}}$ for $e, e'\in\EE$ with $\pi(e)=\pi(e')$.\\

For $\rho_1,\rho_2\in C(\VV,R)$ and $\phi\in C(\VV,M)$ it is easily checked %
that 
\[\big(\tilde{\delta}_{\rho_1}\circ \tilde{\delta}^{\rho_2}\big) (\phi) = (T+T\R)(\rho_1\cdot\rho_2)\cdot\phi -\rho_2\cdot (T+T\R)(\rho_1\cdot\phi). \] 

For $a'\in R$ and $\rho\in\ker(T+T\R-a')$		%
, applying this equality for %
$\rho_1=\rho$ and $\rho_2=1$ %
shows that %
$\tilde{\delta}_\rho$ maps $\im \delta^*$ into $\im(T+T\R-a')$%
, so we get an $R$-linear map 	%
\[\tilde{\delta}_\rho: \coker\big(\delta^*:C_c(\VV,R) \to C_c(\EE,R)\big) \to \coker(T_c+T_c \R_c -a').\]  %

\phantom{ X}

Let now again $\alpha, \muu\in R^*$, and $a:=\alpha+q\muu/\alpha$. We let $\rho:=\rho_{\alpha,\muu}\in\tilde{\B}^{a,\muu}(F,R)$ as defined in lemma \ref{rho}, %
and write $\tilde{\delta}_{\alpha,\muu} :=\tilde{\delta}_{\rho}$%
. 
Since $\tilde{\delta}_{\alpha,\muu}$ maps $1_{\{e\}}-1_{\{\omeg e\} }$ into $\im(R-\muu)$, it induces a map 

\[\tilde{\delta}_{\alpha,\muu}:C^0(\PP^1(F),R)/R \to \B_{a,\muu}(F,R) \]
(same name by abuse of notation) via the commutative diagram 

\[ \xymatrix{ \coker \delta^{*}  \ar[r]^-{\tilde{\delta}_{\alpha,\muu}} \ar[d]^{} & \coker(T_c+T_c \R_c -a')\ar[d]^{\!\!\!\!\mod (\R-\muu)} \\
C^0(\PP^1(F),R)/R \ar[r]^-{\tilde{\delta}_{\alpha,\muu}} & \B_{a,\muu}(F,R) } \]

with $a':=a(1+\muu)$, since $\rho\in\B^{a,\muu}(F,R) \subeq\ker(T+T\R-a')$.\\

\begin{lemma}\label{delta}
We have $\rho\left(g v\right)=\chi_\alpha(d/\aaa)\chi_\muu(\aaa)\rho(v)$, and thus
\[\tilde{\delta}_{\alpha,\muu}(gf)=\chi_\alpha(d/\aaa)\chi_\muu(\aaa) g\tilde{\delta}_{\alpha,\muu}(f),\]
 for all $v\in\VV$, $f\in C^0(\PP^1(F),R)/R$ and $g=\begin{pmatrix}\aaa&b\\0&d\end{pmatrix}\in B(F)$.\\

\end{lemma}
\begin{proof}
(a) Using lemma \ref{2.6}(b) and the fact that $\ord_{gv}(e_1)=-\ord_\omeg(\aaa)+\ord_v(e_1)$, we have 
 \[\rho\left(\begin{pmatrix}\aaa&b\\0&d\end{pmatrix} v\right)=\alpha^{h(v)-\ord_\omeg(\aaa/d)}\muu^{\ord_\omeg(\aaa)-\ord_{v}(e_1)}=\chi_\alpha(d/\aaa)\chi_\muu(\aaa)\rho(v)\]
 for all $v\in\VV$. For $f$ and $g$ as in the assertion, we thus have 
 \begin{eqnarray*}
  \tilde{\delta}_{\alpha,\muu}(gf)(v) & =& \sum_{v=t(e)}\rho(o(e))f(g^{-1}e)-\sum_{v=o(e)}\rho(t(e))f(g^{-1}e) \\
				     & =& \sum_{g^{-1}v=t(e)}\rho(o(ge))f(e) - \sum_{g^{-1}v=o(e)}\rho(t(ge))f(e) \\
				     & =& \chi_\alpha(d/\aaa)\chi_\muu(\aaa)\rho(v)\left(\sum_{g^{-1}v=t(e)}\rho(o(e))f(e) - \sum_{g^{-1}v=o(e)}\rho(t(e))f(e)\right)\\
				     & =& \chi_\alpha(d/\aaa)\chi_\muu(\aaa) g\tilde{\delta}_{\alpha,\muu}(f)(v) .
 \end{eqnarray*}

\end{proof}

We define a function $\delta_{\alpha,\muu}: C_c(F^*,R)\to \B_{a,\muu}(F,R)$ as follows: For $f\in C_c(F^*,R)$, we let $\psi_0(f)\in C_c(\PP^1(F),R)$ be the extension of $x\mapsto \chi_\alpha(x)\chi_\muu(x)^{-1}f(x)$ by zero to $\PP^1(F)$. We set $\delta_{\alpha,\muu}:= \tilde{\delta}_{\alpha,\muu}\circ\psi_0$. If $\alpha=\muu%
$, we can define $\delta_{\alpha,\muu}$ on all functions in $C_c(F,R)$. 

We let $F^*$ operate on $C_c(F,R)$ by $(tf)(x):=f(t^{-1}x)$; this induces an action of the group $T^1(F):=\{ \left(\begin{smallmatrix} t&0\\0&1 \end{smallmatrix}\right)|t\in F^* \}$, which we identify with $F^*$ in the obvious way. With respect to it, we have 

\[ \psi_0(tf)(x)=\chi_{\alpha}(t)\chi_\muu(t)^{-1} t \psi_0(f)(x)\]
and
\[\tilde{\delta}_{\alpha,\muu}(tf) =\chi_\alpha^{-1}(t)\chi_\muu(t) t \tilde{\delta}_{\alpha,\muu}(f),\]

so ${\delta}_{\alpha,\muu}$ is $T^1(F)$-equivariant.\\

For an $R$-module $M$, we define an $F^*$-action on $\Dist(F^*,M)$ by $\int f d(t\mu) :=t \int (t^{-1}f) d\mu$.
Let $H\subeq G(F)$ be a subgroup, and $M$ an $R[H]$-module. We define an $H$-action on $\B^{a\muu,\muu^{-1}}(F,M)$ by requiring $\langle\phi,h\lambda\rangle=h\cdot \langle h^{-1}\phi,\lambda\rangle$ for all $\phi\in\B_{a,\muu}(F,M)$, $\lambda\in\B^{a\muu,\muu^{-1}}(F,M)$, $h\in H$. 
With respect to these two actions, we get a $T^1(F)\cap H$-equivariant mapping
\[\delta^{\alpha,\muu}: \B^{a\muu,\muu^{-1}}(F,M)\to \Dist (F^*,M), \quad \delta^{\alpha,\muu}(\lambda):=\langle\delta_{\alpha,\muu}(\cdot),\lambda\rangle\]
dual to $\delta_{\alpha,\muu}$. %

\subsection{Local distributions}
Now consider the case $R=\CC$. Let $\chi_1,\chi_2:F^*\to\CC^*$ be two unramified characters. We consider $(\chi_1,\chi_2)$ as a character on the torus $T(F)$ of $\GL_2(F)$, which induces %
a character $\chi$ on $B(F)$ by 
\[\chi\begin{pmatrix}
       t_1&u\\0&t_2
      \end{pmatrix}:=\chi_1(t_1)\chi_2(t_2).\]
Put $\alpha_i:=\chi_i(\omeg)\sqrt{q}\in\CC^*$ for $i=1,2$. Set $\muu:=\chi_1(\omeg)\chi_2(\omeg)=\alpha_1\alpha_2 q^{-1}\in \CC^*$, and $a:=\alpha_1+\alpha_2=\alpha_i+q\muu/\alpha_i$ for either $i$. 
When $a$ and $\muu$ are given by the $\alpha_i$ like this, we will often write $\B_{\alpha_1,\alpha_2}(F,R):=\B_{a,\muu}(F,R)$ and $\B^{\alpha_1,\alpha_2}(F,M):=\B^{a\muu,\muu^{-1}}(F,M)$ (!) for its dual.

In the special case $a^2=\muu(q+1)^2$, we assume the $\chi_i$ to be sorted such that $\chi_1= |\cdot |\chi_2$ (not vice versa).\\

Let $\B(\chi_1,\chi_2)$ denote the space of continuous maps $\phi:G(F)\to \CC$ such that
\begin{equation} %
\phi\left(\begin{pmatrix}t_1&u\\0&t_2\end{pmatrix} g\right)=\chi_{\alpha_1}(t_1)\chi_{\alpha_2}(t_2)|t_1|\phi(g)
\end{equation}
for all $t_1,t_2\in F^*$, $u\in F$, $g\in G(F)$. $G(F)$ operates canonically on $\B(\chi_1,\chi_2)$ by right translation (cf. \cite{Bump}, Ch. 4.5).  %
If $\chi_1\chi_2^{-1}\ne|\cdot|^{\pm 1}$, $\B(\chi_1,\chi_2)$ is a model of the spherical representation $\pi(\chi_1,\chi_2)$;
if $\chi_1\chi_2^{-1}=|\cdot|^{\pm 1}$, the special representation $\pi(\chi_1,\chi_2)$ can be given as an irreducible subquotient of codimension 1 of $\B(\chi_1,\chi_2)$.\footnote{Note that \cite{Bump} %
 denotes this special representation by  $\sigma(\chi_1,\chi_2)$, not by $\pi(\chi_1,\chi_2)$.} %

\begin{lemma}\label{representation}
 We have a  $G$-equivariant isomorphism $\tilde{\B}_{a,\muu}(F,\CC) %
 \iso \B(\chi_1,\chi_2)$. 
It induces an isomorphism $\B_{a,\muu}(F,\CC)\iso \pi(\chi_1,\chi_2)$ both for spherical and special representations.
\end{lemma}

\begin{proof}
We choose a ``central'' unramified character $\chi_Z: F^*\to\CC$ satisfying $\chi_Z^2(\omeg)=\muu$; then we have $\chi_1=\chi_Z {\chi_0}^{-1},\chi_2=\chi_Z\chi_0$ for some unramified character $\chi_0$. We set $a':=\sqrt{q}\left( \chi_0(\omeg)^{-1}+\chi_0(\omeg) \right)$, which satisfies $a=\chi_Z(\omeg)a'$.

For a representation $(\pi, V)$ of $G(F)$, by \cite{Bump}, Ex. 4.5.9, we can define another representation $\chi_Z\otimes\pi$ on $V$ via \[(g,v)\mapsto \chi_Z(\det(g))\pi(g)v\quad\text{ for all }g\in G(F),v\in V,\]
and with this definition we have $\B(\chi_1,\chi_2) \iso \chi_Z \otimes \B(\chi_0^{-1},\chi_0)$. Since $\B(\chi_0^{-1},\chi_0)$ has trivial central character, \cite{BL}, Thm. 20 (as quoted in \cite{Sp}) states that \linebreak$\B(\chi_0^{-1},\chi_0)\iso \B_{a',1}(F,\CC) \iso \Ind^{G(F)}_{KZ} R/\Im(T-a')$. 
 
Define a $G$-linear map $\phi: \Ind^G_K R \to \chi_Z\otimes \Ind^G_{KZ} R$ by $1_K\mapsto (\chi_Z\circ \det) \cdot 1_{KZ}$. Since $1_K$ (resp.  $(\chi_Z\circ \det) \cdot 1_{KZ}$) generates $\Ind^G_K R$ (resp. $\chi_Z\otimes \Ind^G_{KZ} R$) as a $\CC[G]$-module, $\phi$ is well-defined and surjective. 

$\phi$ maps $\R 1_K=\left( \begin{smallmatrix}\omeg&0\\0&\omeg\end{smallmatrix} \right) 1_K$ to 
\[\left( \begin{smallmatrix}\omeg&0\\0&\omeg\end{smallmatrix} \right) ((\chi_Z\circ \det) \cdot 1_{KZ})=\chi_Z(\omeg)^2\cdot ((\chi_Z\circ \det) \cdot 1_{KZ})=\muu \cdot \phi(1_K).\]
Thus $\Im(\R-\muu)\subeq\ker\phi$,
and in fact the two are equal%
, since the preimage of the space of functions of support in a coset $KZg$ ($g\in G(F)$) under  $\phi$ is exactly the space generated by the $1_{Kzg}$, $z\in Z(F)=Z(\OO_F) \{\left( \begin{smallmatrix}\omeg&0\\0&\omeg\end{smallmatrix} \right)\}^{\ZZ}$.

Furthermore, $\phi$ maps $T 1_K = \sum_{i \in \OO_F/(\omeg) \cup\{\infty\} } N_i 1_K$ (with the $N_i$ as in Lemma \ref{rho}) to 
\[\sum_i \chi_Z(\det(N_i))\cdot((\chi_Z\circ\det) \cdot N_i 1_{KZ}) =\chi_Z(\omeg)\cdot (\chi_Z\circ\det) T 1_{KZ}\]
(since $\det(N_i)=\omeg$ for all $i$),%
and thus $\Im(T-a)$ is mapped to $\Im\big(\chi_Z(\omeg)T-a\big)=\Im\big(\chi_Z(\omeg)(T-a')\big)=\Im(T-a')$.
 
Putting everything together, we thus have $G$-isomorphisms 
\begin{eqnarray*}%
C_c(\VV,\CC)/\big(\Im(T-a)+\Im(\R-\muu) \big) &\iso& \Ind^G_K R /\big(\Im(T-a)+\Im(\R-\muu) \big)\\
 &\iso& \chi_Z\otimes \big(\Ind^G_{KZ} R /\Im(T-a')\big)  \quad\text{(via $\phi$)}\\  
 &\iso & \chi_Z \otimes \B(\chi_0^{-1},\chi_0) \iso \B(\chi_1,\chi_2).
\end{eqnarray*}
 Thus, $\B_{a,\muu}(F,\CC)$ is isomorphic to the spherical principal series representation $\pi(\chi_1,\chi_2)$ for $a^2\ne \muu(q+1)^2$. 

In the special case, $\B_{a,\muu}(F,\CC)$ is a $G$-invariant subspace of $\tilde{\B}_{a,\muu}(F,\CC)$ of %
codimension 1, so it must be mapped under the isomorphism to the unique $G$-invariant subspace of $\B(\chi_1,\chi_2)$ of codimension 1 (in fact, the unique infinite-dimensional irreducible $G$-invariant subspace, by \cite{Bump}, Thm. 4.5.1), which is the special representation $\pi(\chi_1,\chi_2)$. 
\end{proof}

By \cite{Bump}, section 4.4, there exists thus for all pairs $a,\muu$  a {\it Whittaker functional} $\lambda$ on $\B_{a,\muu}(F,\CC)$, i.e. a nontrivial linear map $\lambda:\B_{a,\muu}(F,\CC)\to\CC$ such that $\lambda\left( \left(\begin{smallmatrix}1&x\\0&1\end{smallmatrix}\right)\phi\right)=\psi(x)\lambda(\phi)$. It is unique up to scalar multiples.

From it, we furthermore get a {\it Whittaker model} $\WW_{a,\muu}$ of $\B_{a,\muu}(F,\CC)$:
\[\WW_{a,\muu}:= \{W_\xi:GL_2(F)\to\CC\,|\,\xi\in \B_{a,\muu}(F,\CC)\},\]
where $W_\xi(g):=\lambda(g\cdot\xi)$ for all $g\in GL_2(F)$.
(see e.g. \cite{Bump}, Ch. 3, eq. (5.6).)\\

Now write $\alpha:=\alpha_1$ for short. Recall the distribution $\mu_{\alpha,\muu}=%
 \psi(x)\chi_{\alpha/\muu}(x) dx \in \Dist(F^*,\CC)$. 
For $\alpha=\muu%
$, it extends to a distribution on $F$.\\

\begin{proposition}
\label{Prop2.9}

(a) There exists a unique  Whittaker functional $\lambda=\lambda_{a,\muu}$ on $\B_{a,\muu}(F,\CC)$ such that $\delta^{\alpha,\muu}(\lambda) =\mu_{\alpha,\muu}$.

(b) %
For every $f\in C_c(F^*,\CC)$, there exists $W=W_f\in\WW_{a,\muu}$ such that
\[\int_{F^*}(af)(x)\mu_{\alpha,\muu}(dx)=W_f \begin{pmatrix}a&0\\0&1\end{pmatrix}.\]

If $\alpha=\muu%
$, then for every $f\in C_c(F,\CC)$, there exists $W_f\in\WW_{a,\muu}$ such that \[\int_{F}(af)(x)\mu_{\alpha,\muu}(dx)=W_f \begin{pmatrix}a&0\\0&1\end{pmatrix}.\]

(c) Let $H\subeq U= \OO_F^\times$ be an open subgroup, and write $W_H:=W_{1_H}$. For every $f\in C_c^0(F^*,\CC)^H$ we have

\[\int_{F^*}f(x)\mu_{\alpha,\muu}(dx)=[U:H]\int_{F^*}f(x)W_H\begin{pmatrix}x&0\\0&1\end{pmatrix}d^\times x .\]\\[-2ex]
\end{proposition}

\begin{proof}
(a) By \cite{Sp}, proof of prop. 3.10, we have a Whittaker functional of the Steinberg representation given by the composite
\begin{equation}\label{Whitt}
St(F,\CC):=C^0(\PP^1(F),\CC)/\CC\xrightarrow{\iso} C_c(F,\CC)\xrightarrow{\Lambda} \CC , \end{equation}
where the first map is the $F$-equivariant isomorphism
\[C^0(\PP^1(F),\CC)/\CC \to C_c(F,\CC),\quad \phi\mapsto f(x):= \phi(x)-\phi(\infty),\]
(with $F$ acting on $C_c(F,\CC)$ by $(x\cdot f) (y):=f(y-x)$, and on $C^0(\PP^1(F),\CC)/\CC$ by $x\phi:=\left(\begin{smallmatrix}1&x\\0&1\end{smallmatrix} \right)\phi$), and the second is
\[\Lambda:C_c(F,\CC)\to\CC,\quad  f\mapsto\int_F f(x)\psi(x)dx.\]

%

Let now $\lambda:\B_{a,\muu}(F,\CC)\to \CC$ be a Whittaker functional of $\B_{a,\muu}(F,\CC)$. By lemma \ref{delta}, for $u=\left(\begin{smallmatrix}1&x\\0&1\end{smallmatrix} \right)\in B(F)$,

\[(\lambda\circ\tilde{\delta}_{\alpha,\muu})(u\phi)=\lambda(u\tilde{\delta}_{\alpha,\muu}(\phi))=\psi(x)\lambda(\tilde{\delta}_{\alpha,\muu}(\phi)),\]

so $\lambda\circ\tilde{\delta}_{\alpha,\muu}$ is a Whittaker functional if it is not zero.

To describe the image of $\tilde{\delta}_{\alpha,\muu}$, consider the commutative diagram
\[ \xymatrix{ C_c(\EE,R)\ar[r]^{\tilde{\delta}_{\alpha,\muu}}\ar[d]^{\eqref{(25)}} & C_c(\VV,R)\ar[d]^{\phi\mapsto\phi\cdot\rho} &  \\
C_c(\EE,R)\ar[r]^\delta & C_c(\VV,R)\ar[r]^-{\langle\cdot, 1\rangle} & R\ar[r] & 0 }   \]
where the vertical maps are defined by
\begin{equation}\label{(25)}
 C_c(\EE,R)\to C_c(\EE,R),\quad c\mapsto \big( e\mapsto c(e)\rho(o(e))\rho(t(e))\big)
\end{equation}
resp. by mapping $\phi$ to $v\mapsto\phi(v)\rho(v)$; both are obviously isomorphisms. 

Since the lower row is exact, we have $\im\delta=\ker\langle\cdot,1\rangle=:C_c^0(\VV,R)$ and thus $\im{\tilde{\delta}_{\alpha,\muu}}=\rho^{-1}\cdot C_c^0(\VV,R)$.

Since $\lambda\ne0$ and $\B_{a,\muu}(F,\CC)$ is generated by (the equivalence classes of) the $1_{\{v\}}$, $v\in\VV$, there exists a $v\in\VV$ such that $\lambda(1_{\{v\}})\ne 0$.
Let $\phi$ be this $1_{\{v\} }$, and let $u=\left(\begin{smallmatrix} 1&x\\0&1 \end{smallmatrix}\right)\in B(F)$ such that $x\notin\ker\psi$. Then  %

\[\rho\cdot (u\phi-\phi) = \rho\cdot (1_{\{u^{-1}v \} } - 1_{\{v\} }) =\rho(v)(1_{\{u^{-1}v \} } - 1_{\{v\} })\in C_c^0(\VV,R)\]

by lemma \ref{delta}, so $0\ne u\phi-\phi\in\im \tilde{\delta}_{\alpha,\muu}$, but $\lambda(u\phi-\phi)=\psi(x)\lambda(\phi)-\lambda(\phi)\ne 0$.

So $\lambda\circ\tilde{\delta}_{\alpha,\muu}\ne0$ is indeed a Whittaker functional. By replacing $\lambda$ by  a scalar multiple, we can assume $\lambda\circ\tilde{\delta}_{\alpha,\muu}=\eqref{Whitt}$.\\

Considering $\lambda$ as an element of $\B^{a\muu,\muu^{-1}}(F,\CC)\iso \Hom(\B_{a,\muu}(F,\CC),\CC)$, we have
\begin{eqnarray*}
\delta^{\alpha,\muu}(\lambda)(f) &=& \langle\delta_{\alpha,\muu}(f),\lambda\rangle \\ %
				     &=& \Lambda%
						(\chi_{\alpha} \chi_{\muu}^{-1} f) \\
				     &=& \int_{F^*}\chi_\alpha(x)\chi_{\muu}^{-1}(x) f(x)\psi(x) dx\\
				     &=& \mu_{\alpha,\muu}(f).
\end{eqnarray*}\\[-2.5ex]

(b) For given $f$, set $W_f(g):=\lambda(g\cdot\delta_{\alpha,\muu}(f))$. Then $W_f\in\WW_{a,\muu}$, and for all $a\in F^*$ we have:
\begin{eqnarray*}
W_f\begin{pmatrix}a&0\\0&1\end{pmatrix}&=&\lambda\left(\begin{pmatrix}a&0\\0&1\end{pmatrix}\delta_{\alpha,\muu%
}(f)\right)\\
		&=&\lambda (\delta_{\alpha,\muu%
}(af))			\qquad\qquad\qquad\mbox{(by the $T^1(F)$-invariance of }\delta_{\alpha,\muu}\mbox{)} \\
		&=&\int_{F^*} (af)(x) \mu_{\alpha,\muu}(dx).
\end{eqnarray*}\\[-2.5ex]

(c) Without loss of generality we can assume $f=1_{aH}$ for some $a\in F^*$. %

We have
\begin{eqnarray*}
\int_{F^*}1_{aH}(x)\mu_{\alpha,\muu}(dx)&=&\int_{F^*}1_H(a^{-1}x)\mu_{\alpha,\muu}(dx)\\
				     &=&\int_{F^*}(a\cdot 1_H)(x)\mu_{\alpha,\muu}(dx)\\
				     &=&W_H\begin{pmatrix}a&0\\0&1\end{pmatrix}  \mbox{   by (b),}
\end{eqnarray*}
and since the left-hand side is invariant under replacing $a$ by $ah$ (for $h\in H$), the right-hand side also is, so we can integrate this constant function over $H$:
\begin{eqnarray*}
                     	&=&[U:H]\int_H W_H\begin{pmatrix}ax&0\\0&1\end{pmatrix}d^\times x\\
		     	&=&[U:H]\int_{F^*}1_H(x)W_H\begin{pmatrix}ax&0\\0&1\end{pmatrix}d^\times x\\
			&=&[U:H]\int_{F^*}1_H(a^{-1}x)W_H\begin{pmatrix}x&0\\0&1\end{pmatrix}d^\times x\\
			&=&[U:H]\int_{F^*}1_{aH}(x)W_H\begin{pmatrix}x&0\\0&1\end{pmatrix}d^\times x.
\end{eqnarray*}\\[-5ex] 

\end{proof}

\subsection{Semi-local theory}\label{semilocal}
We can generalize many of the previous constructions to the semi-local case, considering all primes $\p|p$ at once.

So let $F_1, \ldots, F_m$ be finite extensions of $\QQ_p$, and for each $i$, let $q_i$ be the number of elements of the residue field of $F_i$. We put $\Fsemilocal:=F_1\times \cdots \times F_m$.

Let $R$ again be a ring, %
and  $a_i\in R, \muu_i\in R^*$ for each $i\in\{1,\ldots,m\}$. Put ${\underline{a}:=(a_1,\ldots, a_m)}$, $\underline{\muu}:=(\muu_1, \ldots,\muu_m)$. We define $\B_{\underline{a},\underline{\muu}}(\Fsemilocal, R)$ as the tensor product

\[\B_{\underline{a},\underline{\muu}}(\Fsemilocal, R):=\bigotimes_{i=1}^m \B_{a_i,\muu_i}(F_i, R).\]

For an $R$-module $M$, we define $\B^{\underline{a\muu},\underline{\muu}^{-1}}(\Fsemilocal, M):= \Hom_R(\B_{\underline{a},\underline{\muu}}(\Fsemilocal, R), M)$; let 

\begin{equation}\label{sl. pairing}
 \langle\cdot,\cdot\rangle:  \B_{\underline{a},\underline{\muu}}(\Fsemilocal, R) \times \B^{\underline{a\muu},\underline{\muu}^{-1}}(\Fsemilocal, M)\to M
 \end{equation}

denote the evaluation pairing. \\

We have an obvious isomorphism 
\begin{equation}\label{tensor}
 \bigotimes_{i=1}^m C_c^0(F_i^*, R)\to C_c^0(\Fsemilocal^*, R), \quad \bigotimes_i f_i\mapsto \left((x_i)_{i=1,\ldots, m} \mapsto \prod_{i=1}^m f_i(x_i)\right) .
\end{equation}

Now when we have $\alpha_{i,1}, \alpha_{i,2}\in R^*$ such that $a_i=\alpha_{i,1}+ \alpha_{i,2}$ and $\muu_i=\alpha_{i,1}\alpha_{i,2}q_i^{-1}$, we can define the $T^1(\Fsemilocal)$-equivariant map 
\[ \delta_{\underline{\alpha}_{1,2}}:=\delta_{\underline{\alpha_1},\underline{\muu}}: C_c^0(\Fsemilocal, R)\to \B_{\underline{a},\underline{\muu}}(\Fsemilocal, R)\]
as the inverse of \eqref{tensor} composed with $\bigotimes_{i=1}^m \delta_{\alpha_{i,1}, \muu_i}$.%

Again, we will often write $\B_{\alphaeinszwei}(F,R):=\B_{\underline{a\muu},\underline{\muu}^{-1}}(F,R)$ and $\B^{\alphaeinszwei}(F,M):=\B^{\underline{a\muu},\underline{\muu}^{-1}}(F,M)$.\\

If $H\subeq G(F)$ is a subgroup, and $M$ an $R[H]$-module, we define an $H$-action on $\B^{\underline{a\muu},\underline{\muu}^{-1}}(F,M)$ by requiring $\langle\phi,h\lambda\rangle=h\cdot \langle h^{-1}\phi,\lambda\rangle$ for all $\phi\in\B_{\underline{a},\underline{\muu}}(F,M)$, \linebreak$\lambda\in\B^{\underline{a\muu},\underline{\muu}^{-1}}(F,M)$, $h\in H$, 
and get a $T^1(\Fsemilocal)\cap H$-equivariant mapping
\[\delta^{\alphaeinszwei %
}: \B^{\underline{a\muu},\underline{\muu}^{-1}}(F,M)\to \Dist (\Fsemilocal^*,M), \quad \delta^{\alphaeinszwei}(\lambda):=\langle\delta_{\alphaeinszwei}(\cdot),\lambda\rangle .\]

Finally, we have a homomorphism 
\begin{equation}
  \label{Btensor}
 \begin{split}
\bigotimes_{i=1}^m \B^{a_i\muu_i,\muu_i^{-1}}(F_i, R) & \xrightarrow{\iso} \bigotimes_{i=1}^m \Hom_R(\B_{a_i\muu_i,\muu_i^{-1}}(F_i, R),R) \\
& \to \Hom(\B_{a_1,\muu_1}(F_1,R), \Hom(\B_{a_2,\muu_2}(F_2,R), \Hom( \ldots ,R ))...) \\
& \xrightarrow{\iso}  %
\B^{\underline{a\muu},\underline{\muu}^{-1}}(F,R).
 \end{split}
\end{equation}
where the second map is given by $\otimes_i f_i\mapsto \left( x_1\mapsto (x_2\mapsto (\ldots  \mapsto \prod_i f_i(x_i))...\right)$, and the last map by iterating the adjunction formula of the tensor product.

\clearpage

\section{Cohomology classes and global measures} \label{cohom}
\subsection{Definitions}
From now on, let $F$ denote a %
number field, with ring of integers $\OO_F$. For each finite prime $v$, let $U_v:=\OO_v^*$.
Let  $\AA=\AA_F$ denote the ring of adeles of $F$, and $\II=\II_F$ the group of ideles of $F$.
For a finite subset $S$ of the set of places of $F$, we denote by $\AA^S:=\{x\in\AA_F\,|\, x_v=0\;\forall v\in S\}$ the $S$-adeles %
and by $\II^S$ the $S$-ideles, and
put $F_S:=\prod_{v\in S}F_v$, $U_S:=\prod_{v\in S}U_v$, $U^S:=\prod_{v\notin S}U_v$ (if $S$ contains all infinite places of $F$), and similarly for other global groups. 

For $\ell$ a prime number or $\infty$, we write $S_\ell$ for the set of places of $F$ above $\ell$, and %
abbreviate the above notations to $\AA^\ell:=\AA^{S_\ell}$, $\AA^{p,\infty}:=\AA^{S_p\cup S_\infty}$, and similarly write $\II^p$, $\II^\infty$, $F_p$, $F_\infty$, $U^\infty$, $U_p$, $U^{p,\infty}$, $\II_\infty$ etc. \\ %

Let  $F$ have %
$r$ real embeddings and $s$ pairs of complex embeddings. Set $d:=r+s-1$. Let $\{ \sigma_0,\ldots, \sigma_{r-1},\sigma_r,\ldots,\sigma_{d}\}$ be a set of representatives of these embeddings (i.e. for $i\ge r$, choose one from each pair of complex embeddings), and denote by $\infty_0,\ldots,\infty_d$ the corresponding archimedian primes of $F$. We let $S^0_\infty:=\{\infty_1,\ldots,\infty_d\}\subeq S_\infty$. \\
%


For each place $v$, let $dx_v$ denote the associated self-dual Haar measure on $F_v$, and $dx:=\prod_v dx_v$ the associated Haar measure on $\AA_F$.
We define Haar measures $d^\times x_v$ on $F_v^*$ by $d^\times x_v:=c_v\frac{dx_v}{|x_v|_v}$, where $c_v=(1-\frac{1}{q_v})^{-1}$ for $v$ finite, $c_v=1$ for $v|\infty$.

For $v|\infty$ complex, we use the decomposition $\CC^*=\RR_+^*\times S^1$ (with $S^1={\{x\in\CC^*; |x|=1\}}$) to write $d^\times x_v=d^\times r_v\; d\vartheta_v$ for variables %
$r_v$, $\vartheta_v$ with $r_v\in\RR_+^*$, $\vartheta_v\in S^1$.\\

Let $S_1\subeq S_p$ be a set of primes of $F$ lying above $p$, $S_2:=S_p- S_1$. Let $R$ be a topological Hausdorff ring. 
\begin{defi}
We define the module of continuous functions
\[\C(S_1,R):=C(F_{S_1}\times F_{S_2}^*\times \II^{p,\infty}/U^{p,\infty}, R);\]
and let $\C_c(S_1,R)$ be the submodule of all compactly supported $f\in \C(S_1, R)$. We write $\C^0(S_1,R)$, $\C_c^0(S_1,R)$ %
for the submodules of locally constant maps (or of continuous maps where $R$ is assumed to have the discrete topology).\\ %
We further define 
\[\C_c^\bee(S_1,R):=\C_c(\emptyset, R)+\C_c^\bee(S_1, R)\subeq \C_c^\bee(S_1,R)\]
to be the module of continuous compactly supported maps that are ``constant near $(0_\p, x^\p)$'' for each $\p\in S_1$.\\
\end{defi}

\begin{defi}
For an $R$-module $M$, let
$\mathcal{D}_f(S_1,M)$  %
denote the $R$-module of maps \[\phi:\CCoo(F_{S_1}\times F_{S_2}^*)\times \II_F^{p,\infty}%
\to M\] that are $U^{p,\infty}$-invariant and such that $\phi(\cdot,x^{p,\infty})$ is a distribution for each $x^{p,\infty}\in\II_F^{p,\infty}$. 
\end{defi}

Since $\II_F^{p,\infty}/U^{p,\infty}$ is a discrete topological group, $\mathcal{D}_f(S_1,M)$ naturally identifies with the space of $M$-valued distributions on $F_{S_1}\times F_{S_2}^*\times \II_F^{p,\infty}/U^{p,\infty}$. 
So there exists a canonical $R$-bilinear map
\begin{equation}
\label{''40''}\D_f(S_1, M)\times\C_c^0(S_1,R)\to M, \quad (\Ophi,f)\mapsto\int f\; d\Ophi, 
\end{equation}
which is easily seen to induce an isomorphism $\D_f(S_1, M)\iso \Hom_R(\C_c^0(S_1,R), M)$.\\

For a subgroup $E\subeq F^*$ and an $R[E]$-module $M$, we let $E$ operate on $\D_f(S_1,M)$ and $\C_c^0(S_1,R)$  by $(a \Ophi) (U,x^{p,\infty}):=a \Ophi(a^{-1}U,a^{-1}x^{p,\infty})$ and $(af)(x^\infty):=f(a^{-1}x^\infty)$ for $a\in E$, $U\in \CCoo(F_{S_1}\times F_{S_2}^*)$, $x^{\cdot}\in \II_F^{\cdot}$; thus we have $\int(af)\;d(a\Ophi)=a\int f\; d\Ophi$ for all $a,f,\Ophi$.  

When $M=V$ is a finite-dimensional vector space %
over a $p$-adic field, we write $\D_f^b(S_1,V)$ for the subset of $\phi\in \mathcal{D}_f(S_1,V)$ such that $\phi$ is even a measure on $F_{S_1}\times F_{S_2}\times\II_F^{p,\infty}/U^{p,\infty}$.\\

\begin{defi}
 For a $\CC$-vector space $V$, define $\mathcal{D}(S_1, V)$ to be the set of all maps $\phi: \CCoo(F_{S_1}\times F_{S_2}^*)\times \II^{p}\to V$ such that:

\begin{enumerate}
\item[(i)] $\phi$ is invariant under $F^\times$ and $U^{p,\infty}$. %
\item[(ii)] For $x^p\in \II^p$, $\phi(\cdot,x^p)$ is a distribution of $F_{S_1}\times F_{S_2}$.
\item[(iii)] For all $U\in\CCoo(F_{S_1}\times F_{S_2}^*)$, the map $\phi_U:\II=F_p^\times\times \II^p\to V, (x_p,x^p)\mapsto \phi(x_p U, x^p)$ is smooth, and rapidly decreasing as $|x|\to\infty$ and $|x|\to 0$. %
\end{enumerate}
\end{defi}

 We will need a variant of this last set: Let $\mathcal{D}'(S_1,V)$
be the set of all maps $\phi\in\mathcal{D}(S_1,V)$ that are ''$(S^1)^s$-invariant'', i.e. such that for all complex primes $\infty_j$ of $F$ and all $\zeta\in S^1=\{x\in\CC^*; |x|=1\}$, we have 
\[\phi(U,x^{p,\infty_j}, \zeta x_{\infty_j})= \phi(U,x^{p,\infty_j}, x_{\infty_j})\text{ for all }x^p=(x^{p,\infty_j},x_{\infty_j})\in\II^p.\]

There is an obvious surjective map 
\[\mathcal{D}(S_1,V)\to\mathcal{D}'(S_1,V), \quad \phi\mapsto \left((U, x)\mapsto \int_{(S^1)^s} \phi(U,x)%
   d \vartheta_{r}\cdots d\vartheta_{r+s-1}  \right)
 \]
given by integrating over $(S^1)^s\subeq (\CC^*)^s\into \II_\infty$.\\

Let ${F^*}'\subeq F^*$ %
be a maximal torsion-free subgroup %
(so that $F/{F^*}'\iso\mu_F$, the roots of unity of $F$). If $F$ has at least one real embedding, we specifically choose ${F^*}'$ to be the set $F^*_+$ of all totally positive elements of $F$ (i.e. positive with respect to every real embedding of $F$). For totally complex $F$, there is no such natural subgroup available, so we just choose ${F^*}'$ freely.
We  set \[E':={F^*}' \cap  O_F^\times \subeq %
O_F^\times,\] so $E'$ is a torsion-free $\ZZ$-module of rank $d$. $E'$ operates freely and discretely on the space
\[\RR_0^{d+1}:=\biggl\lbrace (x_0,\ldots,x_d)\in\RR^{d+1}|\sum_{i=0}^d x_i=0\biggr\rbrace \]  %
via the embedding
\begin{eqnarray*}
E'&\into&\RR_0^{d+1}\\
a&\mapsto&(\log|\sigma_i(a)|)_{i\in S_\infty}
\end{eqnarray*}
(cf. proof of Dirichlet's unit theorem, e.g. in \cite{Neu}, Ch. 1), and the quotient $\RR_0^{d+1}/E'$ is compact. %
We choose the orientation on $\RR_0^{d+1}$ induced by the natural orientation on $\RR^d$ via the isomorphism $\RR^d\iso \RR_0^{d+1}$, $(x_1,\ldots,x_d)\mapsto (-\sum_{i=1}^d x_i, x_1,\ldots, x_d)$. So $\RR_0^{d+1}/E'$ becomes an oriented compact $d$-dimensional manifold.\\

Let $\GG_p$ be the Galois group of the maximal abelian extension of $F$ which is unramified outside $p$ and $\infty$; for a $\CC$-vector space $V$, let $\Dist(\GG_p,V)$ be the set of $V$-valued distributions of $\GG_p$. Denote by $\rho:\II_F/F^*\to\GG_p$ the projection given by global reciprocity.\\ %

\subsection{Global measures}\label{triangle} %

Now let $V=\CC$, equipped with the trivial ${F^*}'$-%
action. We want to construct a commutative diagram

\begin{equation}\label{kommDreieck}
 \xymatrix{ \D(S_1, \CC)\ar[dr]^{\phi\mapsto\mu_\phi} \ar[rr]^{\phi\mapsto\kappa_\phi} & &  H^d\bigl({F^*}', \D_f(S_1, \CC)\bigr) \ar[dl]^{\kappa\mapsto\mu_\kappa=\kappa\cap\dd(\cdot)} \\  %
            & \Dist(\GG_p,\CC) }
\end{equation}\\

First, let $R$ be any topological Hausdorff ring.
Let $\bar{E'}$ denote the closure of $E'$ in $U_p$. The projection map $\pr:\II^\infty/U^{p,\infty}\to \II^\infty/ (\bar{E'}\times U^{p,\infty})$ induces an isomorphism

\[\pr^*: C_c(\II^\infty/(\bar{E'}\times U^{p,\infty}), R)\to H^0(E', C_c(\II^\infty/ U^{p,\infty}, R)),\]

and the reciprocity map induces a surjective map $\bar{\rho}:\II^\infty/(\bar{E'}\times U^{p,\infty})\to \GG_p$.

Now we can define a map
\begin{eqnarray*}
\rho^\sharp: H_0({F^*}'/E', C_c(\II^\infty/(\bar{E'}\times U^{p,\infty}), R))\to C({\GG_p},R)\mspace{260.0mu}\\
\mspace{260.0mu}{[}f]  \mapsto \Bigl(\bar{\rho}(x)\mapsto\mspace{-13.0mu}\sum_{\zeta\in {F^*}'/E'}\mspace{-13.0mu}f(\zeta x)\text{ for }x\in \II^\infty/(\bar{E'}\times U^{p,\infty})\Bigr).
\end{eqnarray*}
This is an isomorphism, with inverse map $f\mapsto [(f\circ\bar{\rho})\cdot 1_\mathcal{F}]$, where $1_\mathcal{F}$ is the characteristic function of a fundamental domain $\mathcal{F}$ of the action of ${F^*}'/E'$ on $\II^\infty/U^\infty$.

We get a composite map 
\begin{equation}\label{(*)}
 \begin{split}
   C({\GG_p},R) & \xrightarrow{(\rho^\sharp)^{-1}}  H_0\bigl({F^*}'/E', C_c(\II^\infty/(\bar{E'}\times U^{p,\infty}), R)\bigr)\\
 &\xrightarrow{\pr^*}   H_0\bigl({F^*}'/E', H^0(E', C_c(\II^\infty/ U^{p,\infty}, R))\bigr) \\
 & \longrightarrow  H_0\bigl({F^*}'/E', H^0(E', \C_c(S_1, R))\bigr), %
 \end{split}
\end{equation}
 where the last arrow is induced by the ``extension by zero'' from $C_c(\II^\infty/ U^{p,\infty}, R)$ to $\C_c(S_1, R)$.

Now let $\eta\in H_d(E',\ZZ)\iso\ZZ$ be the generator that corresponds to the given orientation of $\RR_0^{d+1}$. This gives us, for every $R$%
-module $A$, a homomorphism
\[\xymatrix{H_0\bigl({F^*}'/E',H^0(E',A)\bigr)\ar[r]^{\cap\eta} & H_0\bigl({F^*}'/E', H_d(E', A)\bigr)}\]
Composing this with the edge morphism 
\begin{equation}\label{edge1} 
H_0\bigl({F^*}'/E', H_d(E', A)\bigr)\to H_d({F^*}', A)
\end{equation}
(and setting $A=\C_c(S_1,R)$) gives a map
\begin{equation}\label{eps}
 H_0\bigl({F^*}'/E', H^0(E',\C_c(S_1,R))\bigr)\to H_d\bigl({F^*}',\C_c(S_1,R)\bigr)   %
\end{equation}

We define
\[\dd: C(\GG_p,R)\to H_d\bigl({F^*}', \mathcal{C}_c(S_1, R)\bigr)\] %
as the composition of \eqref{(*)} with this map.

Now, letting $M$ be an $R$-module equipped with the trivial ${F^*}'$-%
action, the bilinear form \eqref{''40''}
\begin{eqnarray*}
 \D_f(S_1, M)\times \C_c%
(S_1, R) & \to & M \\
(\phi, f)  		     & \mapsto & \int f \; d\phi
\end{eqnarray*}
induces  a cap product
\begin{equation}\label{cap}
 \cap: H^d\bigl({F^*}',\D_f(S_1, M)\bigr)\times H_d\bigl({F^*}', \C_c%
(S_1, R)\bigr)\to H_0({F^*}', M) = M. 
\end{equation}

Thus for each $\kappa\in H^d({F^*}',\D_f(S_1, M))$, we get a distribution $\mu_\kappa$ on $\GG_p$ by defining 
\begin{equation}
 \int_{\GG_p} f(\gamma)\;\mu_\kappa(d\gamma):=\kappa\cap\dd(f)
\end{equation}
 for all continuous maps $f:\GG_p\to R$.\\

Now let $M=V$ be a finite-dimensional vector %
space over a $p$-adic field $K$, and let $\kappa\in H^d({F^*}',\D_f^b(S_1, V))$. We identify $\kappa$ with its image in $H^d({F^*}',\D_f(S_1, V))$; then it is easily seen that $\mu_\kappa$ is also a measure, i.e. we have a map
\begin{equation}\label{mass}
 H^d({F^*}',\D_f^b(S_1, V))\to \Dist^b(\GG_p, V),\quad \kappa\mapsto \mu_\kappa.\\[+3ex]
\end{equation}

Let $L|F$ be a $\ZZ_p$-extension of $F$. Since it is unramified outside $p$, it gives rise to a continuous homomorphism $\GG_p\to\Gal(L|F)$ via $\sigma\mapsto\sigma|_L$. Fixing an isomorphism $\Gal(L|F)\iso p^{\epsilon_p} \ZZ_p$ (where $\epsilon_p=2$ for $p=2$, $\epsilon_p=1$ for $p$ odd), we obtain a surjective homomorphism $\ell:\GG_p\to%
p^{\epsilon_p} \ZZ_p$. Here we have chosen the target space such that the $p$-adic exponential function $\exp_p(s \ell(\gamma))$ is defined for all $s\in\ZZ_p$, $\gamma\in\GG_p$. 

\begin{ex}
Let $L$ be the cyclotomic $\ZZ_p$-extension of $F$. Then we can take $\ell=\log_p\circ \N$, where $\N:\GG_p\to\ZZ_p^*$ is the  $p$-adic cyclotomic character, defined by requiring $\gamma \zeta=\zeta^{\N(\gamma)}$ for all $\gamma\in\GG_p$ and all $p$-power roots of unity $\zeta$. It is well-known (cf. \cite{Wa}, par. 5) that $\log_p(\ZZ_p^*)= p^{\epsilon_p}\ZZ_p$.\\
\end{ex}

It is well-known %
that $F$ has $t$ independent $\ZZ_p$-extensions, where $s+1\le t \le [F:\QQ]$; the Leopoldt conjecture implies $t=s+1$. We get a $t$-variable $p$-adic L-functions as follows:

\begin{defi} \label{L-fmulti}\label{L-f}
Let $K$ be a $p$-adic field, $V$ a finite-dimensional $K$-vector space, $\kappa\in H^d({F^*}',\D_f^b(S_1, V))$. Let $\ell_1, \ldots, \ell_t: \GG_p\to p^{\epsilon_p}\ZZ_p$ be continuous homomorphisms.%
The  {\it $p$-adic L-function} of $\kappa$ is given by  
\[L_p(\ul{s},\kappa):= L_p(s_1,\ldots,s_t,\kappa):= \int_{\GG_p}\left(\prod_{i=1}^t \exp_p(s_i \ell_i(\gamma))\right) \mu_\kappa (d\gamma)\]
for all $s_1,\ldots,s_t\in\ZZ_p$.\\ %
\end{defi}

\begin{remark}\label{Sigma}
 Let $\Sigma:=\{\pm1\}^r$, where $r$ is the number of real embeddings of $F$. The group isomorphism $ \ZZ/2\ZZ\iso\{\pm1\}, \epsilon\mapsto (-1)^\epsilon,$ induces a pairing 
\[\langle\cdot,\cdot\rangle:\Sigma\to \{\pm1\}, \quad\langle((-1)^{\epsilon_i})_i,  ((-1)^{\epsilon'_i})_i\rangle:= (-1)^{\sum_i \epsilon_i\epsilon'_i}.\]
 For a field $k$ of characteristic zero, a $k[\Sigma]$-module $V$ and $\underline{\mu}=(\mu_0,\ldots,\mu_{r-1})\in\Sigma$, we put $V_{\underline{\mu}}:=\{v\in V\,|\, \langle\underline{\mu},\underline{\nu}\rangle v = \underline{\nu} v\; \forall \underline{\nu}\in\Sigma\}$, so that we have $V=\bigoplus_{\underline{\mu}\in\Sigma} V_{\underline{\mu}}$. We write $v_{\underline{\mu}}$ for the projection of $v\in V$ to $V_{\underline{\mu}}$, and $v_+:=v_{(1, \ldots, 1)}$.

We identify $\Sigma$ with $F^*/{F^*}'$ via the isomorphism $\Sigma\iso \prod_{i=0}^{r-1}\RR^*/\RR^*_+ \iso F^*/{F^*}'$. Then for each $F^*$-module $M$, $\Sigma$ acts on $H^d({F^*}',\D_f(S_1, M))$ and on $H^d({F^*}',\D_f^b(S_1, M))$. The exact sequence  $\Sigma\iso \prod_{i=0}^{r-1}\RR^*/\RR_+^*=\II_\infty/\II_\infty^0 \to \GG_p \to \GG_p^+\to 0$ of class field theory (where $\II_\infty^0$ is the maximal connected subgroup of $\II_\infty$) yields an action of $\Sigma$ on $\GG_p$. We easily check that \eqref{mass} is $\Sigma$-equivariant, and that the maps 
$\gamma\mapsto\exp_p(s \ell_i(\gamma))$ factor over $\GG_p\to \GG_p^+$ (since $\ZZ_p$-extensions are unramified at $\infty$). Therefore we have 
$L_p(\ul{s},\kappa)=L_p(\ul{s},\kappa_+)$.\\
\end{remark}

For $\Ophi\in\D(S_1, V)$ and $f\in C^0(\II/F^*,\CC)$, let
\[ \int_{\II/F^*}f(x)\Ophi(d^\times x_p,x^p)\;d^\times x^p:=[U_p:U]\int_{\II/F^*}f(x)\Ophi_U(x)\;d^\times x,\]
where we choose an open set $U\subeq U_p$ such that $f(x_p u,x^p)=f(x_p,x^p)$ for all $(x_p,x^p)\in \II$ and $u\in U$; such a $U$ exists by lemma \ref{open} below.

Since this integral is additive in $f$, there exists a unique $V$-valued distribution $\mu_\Ophi$ on $\GG_p$ such that
\begin{equation}\label{47b}
 \int_{\GG_p} f \;d\mu_\Ophi = \int_{\II/{F^*}} f(\rho(x)) \Ophi(d^\times x_p, x^p)\;d^\times x^p
\end{equation}

for all functions $f\in C^0(\GG_p,V)$.\\

\begin{lemma}\label{open}
Let $F:\II/F^*\to X$ be a locally constant map to a set $X$. Then there exists an open subgroup $U\subeq \II$ such that $f$ factors over $\II/F^* U$.
 \end{lemma}
\begin{proof}(cf. \cite{Sp}, lemma 4.20)\\
 $\II_\infty=\prod_{v|\infty}F_v$ is connected, thus $f$ factors over $\bar{f}:\II/F^*\II_\infty\to X$. Since $\II/F^*\II_\infty$ is profinite, $\bar{f}$ further factors over a subgroup $U'\subeq\II^\infty$ of finite index, which is open.
\end{proof}
Let $U_\infty^0:=\prod_{v\in S^0_\infty} \RR^*_+$; the isomorphisms $U^0_\infty \iso \RR^d$, $(r_v)_v\mapsto(\log r_v)_v$, and $\RR^d\iso\RR_0^{d+1}$ give it the structure of %
a $d$-dimensional oriented manifold %
(with the natural orientation). It has the $d$-form $d^\times r_1\cdot\ldots\cdot d^\times r_d$, where (by slight abuse of notation) we choose $d^\times r_i$ on $F_{\infty_i}$ corresponding to the Haar measure $d^\times x_i$ resp. $d^\times r_i$ on $\RR^*_+\subeq F_{\infty_i}^*$. %

$E'$ operates on $U_\infty^0$ via $a\mapsto (|\sigma_i(a)|)_{i\in S_\infty^0}$, making the isomorphism $U^0_\infty\iso\RR_0^{d+1}$ $E'$-equivariant.\\

For $\Ophi\in\D'(S_1, V)$, set
\begin{eqnarray*}
 \int_0^\infty\Ophi \;d^\times r_0\colon \CCoo(F_{S_1}\times F_{S_2}^*)\times\II^{p,\infty_0}&\to&\CC\\ %
 (U,x^{p,\infty_0}) &\mapsto& \int_0^\infty\Ophi(U,r_0,x^{p,\infty_0})\; d^\times r_0,
\end{eqnarray*}
where we let $r_0\in F_{\infty_0}$ run through the positive real line $\RR_+^*$ in $F_{\infty_0}$.
Composing this with the projection $\D(S_1,V)\to \D'(S_1, V)$  gives us a map
\begin{equation}\label{''48''}
\begin{split}
  \D(S_1, V)&\to H^0\bigl({F^*}',\D_f(S_1, C^\infty(U^0_\infty, V))\bigr), \\
 \Ophi &\mapsto \int_{(S^1)^s}\left(\int_0^\infty\Ophi \;d^\times r_0\right) d\vartheta_{r}\; d\vartheta_{r+1}\ldots d\vartheta_{r+s-1}
\end{split}
\end{equation}
(where $C^\infty(U^0_\infty, V)$ denotes the space of smooth $V$-valued functions on $U^0_\infty$), since one easily checks that $\int_0^\infty\Ophi \;d^\times r_0$ is ${F^*}'$-invariant.\\ %

Define the complex $C^\bullet:=\D_f(S_1, \Omega^\bullet(U^0_\infty, V))$. By the Poincare lemma, this is a resolution of $\D_f(S_1, V)$. We now define the map $\phi\mapsto\kappa_\phi$ as the composition of \eqref{''48''} with the composition
\begin{equation}\label{(50)}
 H^0\bigl({F^*}',\D_f(S_1, C^\infty(U^0_\infty, V))\bigr)\to H^0({F^*}', C^d)\to H^d({F^*}', \D_f(S_1, V)),
\end{equation}
where the first map is induced by 
\begin{equation}\label{(51)}
 C^\infty(U^0_\infty, V)\to \Omega^d(U^0_\infty, V),\quad f\mapsto f(r_1,\ldots, r_d)d^\times r_1\cdot\ldots\cdot d^\times r_d ,
\end{equation}
 and the second is an edge morphism in the spectral sequence
\begin{equation}\label{spect}
 H^q({F^*}', C^p)\thus H^{p+q}({F^*}',\D_f(S_1, V)).
\end{equation}

Specializing to $V=\CC$, we now have:

\begin{proposition}\label{dreieck}
The diagram \eqref{kommDreieck} commutes, i.e., for each $\Ophi\in\D(S_1, \CC)$, we have
\[\mu_\phi=\mu_{\kappa_\phi}.\]
\end{proposition}

\begin{proof} 
As in \cite{Sp}, we define a pairing 
\[ \langle\;,\;\rangle\colon  \D(S_1, \CC)\times C^0(\GG_p, \CC) \to \CC \]

as the composite of $\eqref{''48''}\times \eqref{(*)}$ with 
\begin{multline}\label{(52)}
H^0\bigl({F^*}',\D_f(S_1, C^\infty(U^0_\infty,\CC))\bigr)\times H_0\bigl({F^*}'/E', H^0(E',\C_c^0(S_1,\CC))\bigr) \\ 
\xrightarrow{\cap} H_0\bigl({F^*}'/E', H^0(E',C^\infty(U^0_\infty,\CC))\bigr)\to H_0({F^*}'/E',\CC)\iso \CC,
\end{multline}

where $\cap$ is the cap product induced by \eqref{''40''}, and the second map is induced by 
\begin{equation}\label{(53)}
 H^0\bigl(E',\C^\infty(U^0_\infty,\CC)\bigr)\to\CC, \quad f\mapsto \int_{U^0_\infty/E'} f(r_1,\ldots,r_d)\; d^\times r_1\ldots d^\times r_d. 
\end{equation}

Then we have 
\[\langle\Ophi,f\rangle=\int_{\GG_p}f(\gamma)\;\mu_\phi(d\gamma)\quad\text{ for all }f\in C^0(\GG_p,\CC),\]
and we can show that $\kappa_\Ophi\cap\dd(f)=\langle \Ophi,f\rangle$ by copying the proof for the totally real case (\cite{Sp}, prop. 4.21, replacing $F_+^*$ by ${F^*}'$, $E_+$ by $E'$), 
using the fact that for a $d$-form on the $d$-dimensional oriented manifold $M:=\RR_0^{d+1}/E'\iso U^0_\infty/E'$, 
integration over $M$ corresponds to taking the cap product with the fundamental class $\eta$ of $M$ under the canonical isomorphism $H^d_{dR}(M)\iso H^d_{sing}(M)= H^d(E', \CC)$.
\end{proof}

\subsection{Exceptional zeros}
Now let $\ell_1,\ldots, \ell_t:\GG_p\to \ZZ_p$ be homomorphisms.
Let $S_1\subeq S_p$ be a set of primes, $\rr:=\#S_1$.

\begin{proposition}\label{Ableitungen}
 For each $\underline{x}=(x_1,\ldots,x_t)\in \NN_0^t$ set $|\underline{x}|:=\sum_{i=1}^t x_i$. Then
\[\dd(\prod_{i=1}^t \ell_i^{x_i})=0 \text{ for all } \ul{x} \text{ with } |\ul{x}|\le \rr-1.\]
\end{proposition}

\begin{proof} (cf. \cite{Sp}, Prop. 4.6)\\
 For each $i\in\{1,\ldots, t\}$ let $\tilde{\ell}_i:\II^\infty\to\QQ_p$ be the composition
\[\tilde{\ell}_i:\II^\infty\xrightarrow{\rho}\GG_p\xrightarrow{\ell_i}\ZZ_p\into \QQ_p.\]

Let $\p_1,\ldots, \p_m$ be the primes of $F$ above $p$. Since ${F^*}'/E'= F^*/\OO_F^\times$ is a free $\ZZ$-module (it embeds into the group of fractional ideals of $\OO_F$), we can choose a subgroup $\T\subeq {F^*}'$ such that ${F^*}'=E'\times \T$. By the finiteness of the class number, we can also find a subgroup $\T'\subeq\T$ of finite index such that 
\[\T'=\T_p\times \T^p =\langle t_1,\ldots t_m\rangle \times \T^p,\]
where $t_i$ generates some power $\p_i^{n_i}$ of $\p_i$ for all $i$, and $\ord_{\p_i}(t)=0$ for all $t\in \T^p$, $i=1,\ldots m$.\\

Let $\F\subeq\II^\infty/U^{p,\infty}$ be a fundamental domain for the action of $\T$ such that $\bar{E'}\F=\F$. %
\eqref{(*)} maps $\tilde{\ell}^\ul{x}:= \prod_{i=1}^t \tilde{\ell}_i^{x_i}$ to the class $[\tilde{\ell}^{\ul{x}} 1_\F ]\in H_0(\T,H^0(E',\C_c(S_1,\CC_p)))\subeq H_0(\T,H^0(E',\C_c^\flat(S_1,\CC_p)))$. Thus, by the definition of $\dd$, we have to show that $[\tilde{\ell}^{\ul{x}} 1_\F ]$ is mapped to zero under the map \eqref{eps}. Now we have a commutative diagram %

\[
\xymatrix{H_0(\T,H^0(E',\C_c^\flat(S_1,\CC_p)))\ar[r]^{\iota_*}\ar[d]^{\eqref{eps}} & H^0(E',H_0(\T,\C_c^\flat(S_1,\CC_p))) \ar[d]^{\cap\eta} \\
	  H_d({F^*}',\C_c^\flat(S_1,\CC_p))\ar[r]^{\coinf}_\iso & H_d(E',H_0(\T,\C_c^\bee(S_1,\CC_p))) }
 \]
where the upper horizontal map is induced by the inclusion $H^0(\cdot,X)\into X$ (or equivalently, the projection $X\onto H_0(X)$) and the lower horizontal map is the coinflation%
. 

By prop. 3.1 of \cite{Sp}, $\C_c^\bee(S_1,\CC_p)$ is a free $\CC_p[\T]$-module (the proof given in \cite{Sp} works verbatim for the case of an arbitrary number field $F$).
So it is an induced $\T$-module and therefore homologically trivial. Thus the short exact sequence for group homology (or Shapiro's lemma) %
shows that $\coinf: H_d({F^*}',\C_c^\flat(S_1,\CC_p))\to H_d(E',H_0(\T,\C_c^\bee(S_1,\CC_p)))$ is an isomorphism. So it suffices to prove that $\iota_*$ maps $[\tilde{\ell}^\ul{x} 1_\F ]$ to zero, i.e. that

\begin{equation}\label{augmentation}
  \tilde{\ell}^\ul{x} 1_\F \in I(\T)\C_c^\bee(S_1,\CC_p),
\end{equation}

where $I(\T)=\langle 1-t\rangle_{t\in\T}$ is the augmentation ideal in the group ring $\CC_p[\T]$. \\

Again by prop. 3.1 of \cite{Sp}, the restriction map 
\[\res: H_0(\T,\C_c^\bee(S_1, \CC_p))\to H_0(\T', \C_c^\bee(S_1,\CC_p)), \quad[f]\mapsto \sum_{[t]\in\T/\T'} tf,\]
is injective, and maps $[\tilde{\ell}^\ul{x} 1_\F]$ to $[\tilde{\ell}^\ul{x} 1_{\F'}]$, where $\F'=\bigcup_{[t]\in\T/\T'}t\F\subeq\II^\infty/U^{p,\infty}$ is a fundamental domain for the action of $\T'$.
Thus we may replace $\T,\F$ by $\T',\F'$ in \eqref{augmentation}. 

We can specifically choose the fundamental domain $\F:=\prod_{i=1}^m\F_i\times\F^p$, where $\F^p \subeq\II^{p,\infty}/U^{p,\infty}$ is a fundamental domain for $\T^p$ and $\F_i:=\OO_{\p_i}\setminus t_i\OO_{p_i}$.

Since the pro-$q$-part of $\GG_p$ is finite for every prime $q\ne p$ and $\QQ_p$ is torsion-free,\linebreak $\tilde{\ell}_j$ equals $\tilde{\ell}_{p,j}:\II^\infty\xrightarrow{\pr} F_p^*\into\II^\infty \xrightarrow{\tilde{\ell}_j}\QQ_p$.\\
Similarly, we let $\tilde{\ell}_{i,j}$ be the restriction of $\elll_j$ to %
${F_{\p_i}^*}$ (considered as a function on $\II^\infty$ or $F_{\p_i}^*$ as needed) for all $i, j$.

For each subset $\Xi\subeq\{1,\ldots,r\}$ let 
\[ \F_\Xi:=\prod_{i\in\Xi}\OO_{\p_i}\times \prod_{i\in S_p\setminus\Xi}\F_i\times\F^p.\]
For $\ul{n}:=(n_{i,j})_{i=1,\ldots, m; j=1,\ldots t}\in \NN_0^{mt}$ with $n_{i,j}=0$ for all $i\in\Xi$ and all $j$, we define $\lambda(\Xi,\ul{n}):=\prod_{i,j}\tilde{\ell}_{i,j}^{n_{i,j}} \cdot 1_{\F_\Xi}\in\C_c^\bee(S_1,\CC_p)$. %
Then by the multinomial formula, 
\[\elll^{\ul{x}}=\sum_{|\ul{n}|=|\ul{x}|} N_{i,j} \lambda(\emptyset, \ul{n})\]
for some $N_{i,j}\in\ZZ$, and it suffices to show that $\lambda(\emptyset, \ul{n})\in I(\T_p)\C_c^\bee(S_1,\CC_p)\subeq I(\T')\C_c^\bee(S_1,\CC_p)$ for all $\ul{n}$ with $|\ul{n}|=|\ul{x}|$. This follows from: %

\begin{lemma}
 If $\#(\Xi)+|\ul{n}| < r$, then $\lambda(\Xi, \ul{n})\in I(\T_p)\C_c^\bee(S_1,\CC_p)$.
\end{lemma}
\begin{proof}
(cf. \cite{Sp}, lemma 4.7) \\
For $t\in F^*$ and $f,g\in\C_c^\bee(S_1,\CC_p)$, we have 
\[(1-t)(f\cdot g)=((1-t)f)\cdot g +f\cdot((1-t)g)-((1-t)f)\cdot((1-t)g),\]
where $1-t\in\CC_p[\T']$.
Since $(1-t)\elll_i(x)=\elll_i(t)$, using this equation recursively shows that
\begin{equation}\label{ells}
 (1-t)\prod_{i,j}\elll_j^{n_{i,j}}=\sum_{\ul{n}'<\ul{n}} a_{\ul{n}'}\prod_{i,j}\elll_j^{n_{i,j}}
\end{equation}
for some $a_{\ul{n}'}\in\CC_p$.\\

We prove the lemma by induction on $|\ul{n}|$. Let $\Xi^C:=\{1,\ldots, r\}\setminus \Xi$. 
For $\ul{n}=\ul{0}= (0,\ldots,0)$ choose any $i\in\Xi^c$ (which is nonempty since $\#(\Xi)<r$). Then we have

\[ \lambda(\Xi,\ul{0})=1_{\F_\Xi}=(1-t_i)1_{\F_{\Xi\cup\{ i\}}}=(1-t_i)\lambda(\Xi\cup\{i\},\ul{0})\in I(\T_p)\C_c^\bee(S_1,\CC_p).\]

For $|\ul{n}|>0$, choose $i'\in\Xi^c$ such that $n_{i,j}=0$ for all $j$ (such an $i'$ exists because $\#\Xi+|\ul{n}|<\rr$)%
. Put $\Xi':=\Xi\cup\{i\}$.
Then we have 

\begin{align*}
 \lambda(\Xi, \ul{n})&=\prod_{i,j}\elll_j^{n_{i,j}}\cdot (1-t_{i'})1_{\F_{\Xi'}}\\
		      &=(1-t_{i'})\lambda(\Xi',\ul{n})-((1-t_{i'})\prod_{i,j}\elll_j^{n_{i,j}})\cdot 1_{\F_{\Xi'}}+((1-t_{i'})\prod_{i,j}\elll_j^{n_{i,j}})\cdot 1_{\F_{\Xi}}\\
		      &\equiv -((1-t_{i'})\prod_{i,j}\elll_j^{n_{i,j}})\cdot 1_{\F_{\Xi'}}+((1-t_{i'})\prod_{i,j}\elll_j^{n_{i,j}})\cdot 1_{\F_{\Xi}} \quad\mod I(\T_p)\C_c^\bee(S_1,\CC_p).
\end{align*}
But by \eqref{ells} and the induction hypothesis, we have 
\[((1-t_{i'})\prod_{i,j}\elll_j^{n_{i,j}})\cdot 1_{\F_{\Xi'}}\in \sum_{\ul{n'}<\ul{n}}\CC_p\lambda(\Xi',\ul{n}')\subeq  I(\T_p)\C_c^\bee(S_1,\CC_p)\]
and
\[((1-t_{i})\prod_{i,j}\elll_j^{n_{i,j}})\cdot 1_{\F_{\Xi}}\in \sum_{\ul{n}<\ul{n}}\CC_p\lambda(\Xi,\ul{n})\subeq  I(\T_p)\C_c^\bee(S_1,\CC_p),\]
and thus the assertion for $\lambda(\Xi,\ul{n})$.
\end{proof}
\end{proof}

\begin{remark}\label{mixed}
It would have been enough to show the proposition only for $x$-th powers of a single homomorphism $\ell:\GG_p\to\ZZ_p$ (i.e. $\dd( \ell^x)=0$ for all homomorphisms $\ell$) for all $x\le r-1$, since each product $\prod_{i=1}^t \ell_i^{x_i}$ of degree $x=|\ul{x}|$ can be written as a linear combination of $x$-th powers of some other homomorphisms ${\ell}:\GG_p\to\ZZ_p$ by a simple algebraic argument (for a ring $R\supeq \QQ$, each monomial $\prod_{i=1}^t X_i^{n_i}\in R[X_1,\ldots,X_t]$ of degree $n=\sum_i n_i$ can be written as a linear combination of $n$-th powers $(X_i+r_{i,j}X_j)^n$; let $\ell$ run through the $\ell_i+r_{i,j}\ell_j$).\\
\end{remark}

\begin{defi}
 A $t$-variable $p$-adic analytic function $f(\ul{s})=f(s_1,\ldots, s_t)$ ($s_i\in\ZZ_p$) {\it has vanishing order} $\ge \rr$ at the point $\ul{0}=(0,\ldots,0)$  if all its partial derivatives of total order $\le\rr-1$ vanish, i.e. if 
\[\frac{\dd^k}{\dd^{\ul{\nn}}\ul{s}} f(\ul{0}):=\frac{\dd^k}{\dd^{\nn_1}s_1 \cdots \dd^{\nn_t}s_t}f(\ul{0})=0\]
for all $\ul{\nn}=(\nn_1,\ldots,\nn_t)\in\NN_0^t$ { with } $k:=|\ul{\nn}|\le \rr-1$. We write $\ord_{\ul{s}=\ul{0}} f(\ul{s})\ge\rr$ in this case.\\
\end{defi}

The proposition implies the following result for the $p$-adic L-function:%
\begin{theorem}\label{exceptional}
 Let $\rr:=\#(S_1)$, $\kappa\in H^d({F^*}',\D_f^b(S_1, V))$, $V$ a finite-dimensional vector space over  a $p$-adic field. %
Then
$L_p(\ul{s},\kappa)$ is a locally analytic %
function%
, and we have \[\ord_{\ul{s}=\ul{0}} L_p(\ul{s},\kappa)\ge \rr.\]
 \end{theorem}

\begin{proof}
We have \[\frac{\dd^k}{\dd^{\ul{\nn}}\ul{s}}L_p(\ul{0},\kappa)=\int_{\GG_p}\left(\prod_{i=1}^t\ell_i(\gamma)^{\nn_i}\right)\mu_\kappa(d\gamma)=\kappa\cap \dd \left(\prod_{i=1}^t\ell_i(\gamma)^{\nn_i}\right)\]
for all $\ul{\nn}=(\nn_1,\ldots,\nn_t)\in\NN_0^t$. Thus the theorem follows from proposition \ref{Ableitungen}.%
\end{proof}

\subsection{Integral cohomology classes}%
\begin{defi}\label{int}
For $\kappa\in H^d({F^*}', \D_f(S_1, \CC))$ and a subring $R$ of $\CC$, we define $L_{\kappa, R}$ as the image of
\[ H_d({F^*}', \C_c^0(S_1, R))\to H_0({F^*}',\CC)=\CC, \quad x\mapsto \kappa\cap x .\] %
 \end{defi}

\begin{lemma}\label{Ded}
 Let $R\subeq\bar{\QQ}$ be a Dedekind ring.\\
(a)For a subring $R'\supeq R$ of $\CC$, we have $L_{\kappa,R'}=R' L_{\kappa,R}$. \\
(b) If $\kappa\ne 0$, then $L_{\kappa, R}\ne 0$.
\end{lemma}
\begin{proof}
\cite{Sp}, lemma 4.15.
%
\end{proof}

\begin{defi}\label{integral}
 A nonzero cohomology class $\kappa\in H^d({F^*}', \D_f(S_1, \CC))$ is called {\it integral} if $\kappa$ lies in the image of $H^d({F^*}',\D_f(S_1, R))\otimes_R \CC\to H^d({F^*}', \D_f(S_1, \CC))$ for some Dedekind ring $R\subeq\bar{\OO}$.
If, in addition, there exists a torsion-free $R$-submodule $M\subeq H^d({F^*}',\D_f(S_1, R))$ of rank $\le 1$ (i.e. $M$ can be embedded into $R$, by the classification of finitely generated $R$-modules) such that $\kappa$ lies in the image of $M\otimes_R\CC\to H^d({F^*}', \D_f(S_1, \CC))$, then $\kappa$ is {\it integral of rank $\le 1$}.
\end{defi}

The following results are straightforward generalizations  of the corresponding results of Spie{\ss} for totally real $F$:

\begin{proposition}
 Let $\kappa\in H^d({F^*}', \D_f(S_1, \CC))$%
. The following conditions are equivalent:  %
\begin{enumerate}
 \item[(i)] $\kappa$ is integral (resp. integral of rank $\le 1$).
\item[(ii)] There exists a Dedekind ring $R\subeq\bar{\OO}$ such that $L_{\kappa,R}$ is a finitely generated $R$-module (resp. a torsion-free $R$-module of rank $\le 1$).
\item[(iii)] There exists a Dedekind ring $R\subeq\bar{\OO}$, a finitely generated $R$-module $M$ (resp. a torsion-free $R$-module of rank $\le 1$) and an $R$-linear map $f:M\to \CC$ such that $\kappa$ lies in the image of the induced map $f_*: H^d({F^*}',\D_f(S_1, M))\to H^d({F^*}', \D_f(S_1, \CC))$.
\end{enumerate}
\end{proposition}
\begin{proof}
 As in \cite{Sp}, prop. 4.17.
\end{proof}

\begin{corollary}\label{p-adic measure}
 Let $\kappa\in H^d({F^*}', \D_f(S_1, \CC))$ be integral and $R\subeq \bar{\OO}$ be as in proposition \ref{int}. Then \\
(a) $\mu_\kappa$ is a $p$-adic measure, and \\
(b) the map $H^d({F^*}', \D_f(S_1, L_{\kappa, R}))\otimes\bar{\QQ} \to \H^d({F^*}', \D_f(S_1, \CC))$ is injective and $\kappa$ lies in its image.
\end{corollary}
\begin{proof} As in \cite{Sp}, cor. 4.18.
\end{proof}

\begin{remark}\label{abuse}
 Let $\kappa$ be integral with Dedekind ring $R$ as above. By (b) of the corollary, we can view $\kappa$ as an element of $H^d({F^*}', \D_f(S_1, L_{\kappa, R}))\otimes\bar{\QQ}$. Put $V_{\kappa}:=L_{\kappa,R}\otimes_R\CC_p$; let $\bar{\kappa}$ be the image of $\kappa$ under the composition 
\[H^d({F^*}', \D_f(S_1, L_{\kappa, R}))\otimes_R\bar{\QQ}\to H^d({F^*}', \D_f(S_1, L_{\kappa, R}))\otimes_R\CC_p\to H^d({F^*}', \D_f^b(S_1, V_{\kappa})),\]
where the second map is induced by $\D_f(S_1, L_{\kappa,R})\otimes_R\CC_p\to\D_f^b(S_1, V_\kappa)$. By lemma \ref{Ded} (a), $\bar{\kappa}$ does not depend on the choice of $R$.

Since $\mu_\kappa$ is a $p$-adic measure, $\mu_{\bar{\kappa}}$ allows integration of all continuous functions $f\in C(\GG_p,\CC_p)$, and by abuse of notation, we write $L_p(s,\kappa):=\int_{\GG_p}\N(\gamma)^s \mu_\kappa(d\gamma):=L_p(s, \bar{\kappa})$ (cf. remark \ref{Sigma}). %
So $L_p(s,\kappa)$ has values in the finite-dimensional $\CC_p$-vector space $V_\kappa$.\\
\end{remark}

\clearpage
\section{$\boldsymbol{p}$-adic L-functions of automorphic forms}%
\label{L-functions}
We keep the notations from chapter \ref{cohom}; so $F$ is again a number field with $r$ real embeddings and $s$ pairs of complex embeddings.

For an ideal $0\ne \m\subeq \OO_F$, we let $K_0(\m)_v \subeq G(\OO_{F_v})$ be the subgroup of matrices congruent to an upper triangular matrix modulo $\m$, and we set $K_0(\m):=\prod_{v\nmid \infty}K_0(\m)_v$, $K_0(\m)^S:=\prod_{v\nmid\infty, v\notin S}K_0(\m)_v$ for a finite set of primes $S$.  For each $\p|p$, let $q_\p=N(\p)$ denote the number of elements of the residue class field of $F_\p$.

We denote by $|\cdot|_\CC$ the square of the usual absolute value on $\CC$, i.e. $|z|_\CC=z\bar{z}$ for all $z\in\CC$, and write $|\cdot|_\RR$ for the usual absolute value on $\RR$ in context.

\begin{defi}
Let $\mathfrak{A}_0(G,\underline{2}, \ZC)$ denote the set of all %
{\it cuspidal automorphic representations} $\pi=\tensor_v \pi_v$ of $G(\AA_F)$ with central character $\ZC$ such that $\pi_v\iso {\sigma(|\cdot|_{F_v}^{1/2},|\cdot |_{F_v}^{-1/2})}$ at all archimedian primes $v$. Here we follow the notation of \cite{JL}; so %
$\sigma(|\cdot|_{F_v}^{1/2},|\cdot |_{F_v}^{-1/2})$ is the discrete series of weight 2, $\D(2)$, if $v$ is real, and is isomorphic to the principal series representation $\pi(\mu_1,\mu_2)$ with $\mu_1(z)=z^{1/2}\bar{z}^{\; -1/2}$, $\mu_2(z)=z^{-1/2}\bar{z}^{1/2}$ if $v$ is complex (cf. section 4.5 below). %

\end{defi}
We will only consider automorphic representations that are {\it $p$-ordinary} , i.e $\pi_\p$ is ordinary (in the sense of chapter \ref{local}) for every $\p|p$.

Therefore, for each $\p|p$ we fix two non-zero elements $\alpha_{\p,1},\alpha_{\p,2}\in \bar{\OO}\subeq\CC$ such that $\pi_{\alpha_{\p,1},\alpha_{\p,2}}$ is an ordinary, unitary %
representation. By the classification of unitary representations (see e.g. \cite{Gelbart}, Thm. 4.27)%
, a spherical representation $\pi_{\alpha_{\p,1},\alpha_{\p,2}}=\pi(\chi_1,\chi_2)$ is unitary if and only if either $\chi_1,\chi_2$ are both unitary characters (i.e. $|\alpha_{\p,1}|=|\alpha_{\p,2}|=\sqrt{q_\p}$)\footnote{To avoid confusion: By $|\alpha_{\p,i}|$ we always mean the archimedian absolute value of $\alpha_{\p,i}\in\CC$; whereas in the context of the $p$-adic characters $\chi_i$,  $|\cdot|$ always means the $p$-adic absolute value, unless otherwise noted.}, or $\chi_{1,2}=\chi_0|\cdot |^{\pm s}$ with $\chi_0$ unitary and $-\einhalb<s<\einhalb$. %
A special representation $\pi_{\alpha_{\p,1},\alpha_{\p,2}}=\pi(\chi_1,\chi_2)$ is unitary if and only if the central character $\chi_1\chi_2$ is unitary. %
In all three cases, we have thus $\max\{|\alpha_{\p,1}|,|\alpha_{\p,2}|\}\ge \sqrt{q_\p}$. Without loss of generality, we will assume the $\alpha_{\p,i}$ to be ordered such that $|\alpha_{\p,1}|\le |\alpha_{\p,2}|$  for all $\p|p$. 

As in chapter \ref{local}, we define $a_\p:=\alpha_{\p,1}+\alpha_{\p,2}$, $\muu_\p:=\alpha_{\p,1}\alpha_{\p,2}/q_\p$.

Let $\underline{\alpha_i}:=%
(\alpha_{\p,i},  \p|p)$, for $i=1,2$. We denote by $\mathfrak{A}_0(G,\underline{2},\ZC, \underline{\alpha_1},\underline{\alpha_2})$ the subset of all $\pi\in \mathfrak{A}_0(G,\underline{2},\ZC)$ such that $\pi_\p=\pi_{\alpha_{\p,1},\alpha_{\p,2}}$ for all $\p|p$.\\

Let $S_1\subeq S_p$ be the set of places such that $\pi_\p$ is the Steinberg representation (i.e. $\alpha_{\p,1}=\muu_\p= 1$, $\alpha_{\p,2}=q$).\footnote{Note that all $\p|p$ with $\alpha_{\p,1}=\muu_\p\in \bar{\OO}^*$, i.e. $\alpha_{\p,2}=q$, already lie in $S_1$, since $|\alpha_{\p,2}|<q$ in the spherical case. $L_p(s,\pi)$ should have an exceptional zero for each $\p\in S_1$, according to the exceptional zero conjecture.}\\ %

For later use we note that $\pi^\infty=\otimes_{v\nmid \infty}\pi_v$ is known to be defined over a finite extension of $\QQ$, the smallest such field being the {\it field of definition} of $\pi$ (cf. \cite{Sp}).\\ %

\subsection{Upper half-space}\label{upper}
Let $\H_2:=\{z\in \CC| \Im(z)>0 \}\iso \RR\times\RR^*_+$ be the complex upper half-plane, and let $\H_3:=\CC\times\RR^*_+$ be the 3-dimensional upper half-space. Each $\H_m$ is a differentiable manifold of dimension $i$. If we write $x=(u,t)\in\H_m$ with $t\in\RR^*_+$, $u$ in $\RR$ or $\CC$, respectively, it has a Riemannian metric $ds^2=\frac{dt^2+du\; d\overline{u}}{t}$, which induces a hyperbolic geometry on $\H_m$, i.e. the geodesic lines on $\H_m$ are given by ``vertical'' lines $\{u\}\times \RR^*_+$ and half-circles with center in the line or plane $t=0$.\\

We have the %
decomposition $\GL_2(\CC)=B'_\CC\cdot Z(\CC)\cdot K_\CC$, where $B'_\CC$ is the %
subgroup of matrices $\bigl( \begin{smallmatrix}   \RR^*_+ &\CC \\ 0 & 1
                     \end{smallmatrix} \bigr) $,
$Z$ is the center, and $K_\CC=\SU(2)$ (cf. \cite{By}, Cor. 43); and analogously $\GL_2(\RR)^+ =B'_\RR\. Z(\RR)\. K_\RR$ with $B'_\RR=\{\left(\begin{smallmatrix} y&x\\ 0&1 \end{smallmatrix}\right)|x\in\RR, y\in\RR^*_+\}$ and $K_\RR=\SO(2)$.

We can identify $B'_\CC$ with $\H_3$ via $\left(\begin{smallmatrix}t&z\\ 0&1\end{smallmatrix}\right)\mapsto (z,t)$, and $B'_\RR$ with $\H_2$ via 
$\left(\begin{smallmatrix} y&x\\ 0&1 \end{smallmatrix}\right)\mapsto x+iy$.
This gives us natural projections
\[\pi_\RR: \GL_2(\RR)^+\onto \GL_2(\RR)^+/\RR^*\SO(2)\iso\H_2\]
and
\[\pi_\CC: \GL_2(\CC)\onto\GL_2(\CC)/\CC^*\SU(2)\iso\H_3.\]

The corresponding left actions on cosets are invariant under the Riemannian metrics on $\H_m$, %
and can be given explicitly as follows:

$\GL_2(\RR)^+$ operates on $\H_2\subeq\CC$ via M\"obius transformations, \[\begin{pmatrix} a&b\\c&d  \end{pmatrix} (z) :=\frac{az+b}{cz+d},\]

and $\GL_2(\CC)$ operates on $\H_3$ by
\[\begin{pmatrix} a&b\\c&d  \end{pmatrix} (z,t):=\left( \frac{(az+b)(\overline{cz+d})+a\overline{c}t^2}{|cz+d|^2+|ct|^2}, \frac{|ad-bc|t}{|cz+d|^2+|ct|^2}\right) \]
(\cite{By}, (3.12)); specifically, we have 
\[ \begin{pmatrix} t&z\\0&1   \end{pmatrix} (0,1)=(z,t)\quad\text{ for } (z,t)\in\H_3 .\]\\

A differential form $\omega$ on $\H_m$ is called {\it left-invariant} if it is invariant under the pullback $L_g^*$ of left multiplication $L_g:x\mapsto g x$ on $\H_m$, for all $g\in G$. Following \cite{By}, eqs. (4.20), (4.24), we choose the following basis of left invariant differential 1-forms on $\H_3$:
\[\beta_0:=-\frac{dz}{t}, \quad \beta_1:=\frac{dt}{t}, \quad \beta_2:=\frac{d\bar{z}}{t}, \]
and on $\H_2$ (writing $z=x+iy\in\H_2$):
\[\beta_1:=\frac{dz}{y}, \quad \beta_2:=-\frac{d\overline{z}}{y}.\]
We note that a form $f_1\beta_1+f_2\beta_2$ is harmonic on $\H_2$ if and only if $f_1/y$ and $f_2/y$ are holomorphic functions in $z$ (\cite{By}, lemma 60).\\

Let $\KK\in\{\RR,\CC\}$. The Jacobian $J(g, (0,1))$ of left multiplication by $g$ in $(0,1)\in\H_m$ with respect to the basis $(\beta_i)_i$ gives rise to a representation
\[\rho=\rho_{\KK}
: Z(\KK)\. K_\KK\to \SL_m(\CC)\]
with $\rho\vert_{Z(\KK)}$ trivial, which on $K_\KK$ is explicitly given by

\[\rho_\CC \begin{pmatrix}
          u&v\\ -\bar{v}&\bar{u}
         \end{pmatrix}
	  = \begin{pmatrix}   u^2&2uv&v^2\\-u\bar{v}&u\bar{u}-v\bar{v}&v\bar{u}\\\bar{v}^2&-2\bar{uv}&\bar{u}^2
                                   \end{pmatrix},\]
resp.
\[\rho_\RR \begin{pmatrix}  \cos(\vartheta)&\sin(\vartheta)\\ -\sin(\vartheta) & \cos(\vartheta)
\end{pmatrix} = \begin{pmatrix} e^{ 2i \vartheta}&0\\0& e^{-2i\vartheta}\end{pmatrix} \]
(\cite{By}, (4.27), (4.21)). In the real case, we will only consider harmonic forms on $\H_2$ that are multiples of $\beta_1$, thus we sometimes identify $\rho_\RR$ with its restriction $\rho_\RR^{(1)}$ to the first basis vector $\beta_1$,

\[\rho_\RR^{(1)}:\SO(2)\to S^1\subeq \CC^*, \quad \kappa_\vartheta=\begin{pmatrix}  \cos(\vartheta)&\sin(\vartheta)\\ -\sin(\vartheta) & \cos(\vartheta)
\end{pmatrix}  \mapsto e^{ 2i \vartheta}.\]  %

For each $i$, let $\omega_i$ be the left-invariant differential 1-form on $\GL_2(\KK)$ which coincides with the pullback $(\pi_\CC)^* \beta_i$ at the identity.
Write $\underline{\omega}$ (resp. $\underline{\beta}$) for the column vector of the $\omega_i$ (resp. $\beta_i$). Then we have the following lemma from \cite{By}:

\begin{lemma}
For each $i$, the differential $\omega_i$ on $G$ induces $\beta_i$ on $\H_m$, by restriction to the subgroup $B'_{k}\iso\H_m$. For a function $\phi:G\to \CC^m$, the form $\phi\cdot\underline{\omega}$ (with $\CC^m$ considered as a row vector, so $\.$ is the scalar product of vectors) induces $f\.\underline{\beta}$, where $f:\H_m\to\CC^m$ is given by 
\[f(z,t):=\phi\left(\begin{pmatrix} t&z\\0&1 \end{pmatrix} \right).\]
\end{lemma}
\noindent (See \cite{By}, Lemma 57.)\\

To consider the infinite primes of $F$ all at once, we define
\[ \H_\infty:=\prod_{i=0}^d \H_{m_i}=\prod_{i=0}^{r-1}\H_2\times \prod_{i=r}^{d}\H_3\]

(where $m_i=2$ if $\sigma_i$ is a real embedding, and $=3$ if $\sigma_i$ is complex),
and let $\H^0_\infty:=\prod_{i=1}^d \H_{m_i}$ be the product with the zeroth factor %
 removed.\footnote{The choice of the $0$-th factor is for convenience; we could also choose any other infinite place, whether real or complex.}\\

For each embedding $\sigma_i$, the elements of $\PP^1(F)$ are cusps of $\H_{m_i}$: for a given complex embedding $F\into\CC$, we can identify $F$ with $F\times \{0\}\into \CC\times\RR_{\ge0}$ and define the ''extended upper half-space`` as $\overline{\H_3}:= \H_3\cup F \cup \{\infty\}\subeq \CC\times\RR_{\ge0}\cup \{\infty\}$; similarly for a given real embedding $F\into\RR$, we get the extended upper half-plane $\overline{\H_2}:=\H_2\cup F\cup\{\infty\}$ . A basis of neighbourhoods of the cusp $\infty$ is given by the sets $\{(u,t)\in\H_{m}|t>N \}$, $N\gg 0$, and of $x\in F$ by the open half-balls in $\H_{m}$ with center $(x,0)$.\\

Let $G(F)^+\subeq G(F)$ denote the subgroup of matrices with totally positive determinant %
. It acts on $\H^0_\infty$ by composing the embedding \[G(F)^+\into \prod_{v|\infty, v\ne v_0}G(F_v)^+,  \qquad g\mapsto (\sigma_1(g),\ldots,\sigma_d(g)),\] 
with the actions of $G(\CC)^+= G(\CC)$ on $\H_3$ and $G(\RR)^+$ on $\H_2$ as defined above, and on $\Omega^d_{\harm}(\H^0_\infty)$ by the inverse of the corresponding pullback, $\gamma\cdot \underline{\omega} := (\gamma^{-1})^*\underline{\omega}$. Both are left actions.\\

Denote by $S_\CC$ (resp. $S_\RR$) the set of complex (resp. real) archimedian primes of $F$. %

For each complex $v$, we write the codomain of %
$\rho_{F_v}$ as
\[\rho_{F_v}: Z(F_v)\. K_{F_v}\to \SL_3(\CC)=:\SL(V_v),\]
for a three-dimensional $\CC$-vector space $V_v$. We denote the harmonic forms on $\GL_2(F_v)$, $\H_{F_v}$ defined above by $\underline{\omega_v}$%
, $\underline{\beta_v}%
$ etc.\\

Let $V=\bigotimes_{v\in S_\CC} V_v\iso (\CC^3)^{\tensor s}$, $Z_\infty=\prod_{v|\infty}%
Z(F_v)$, $K_\infty=\prod_{v|\infty}K_{F_v}$. We can merge the representations $\rho_{F_v}$ %
for each $v|\infty$  into a representation
\[\rho=\rho_\infty:=\bigotimes_{v\in S_\CC}\rho_{\CC} \otimes \bigotimes_{v\in S_\RR}\rho_\RR^{(1)}: Z_\infty\. K_\infty \to \SL(V),\] 
and define $V$-valued  vectors of differential forms $\underline{\omega}:=\bigotimes_{v\in S_\CC}\underline{\omega_v}\otimes \bigotimes_{v\in S_\RR}\omega_v^1$, $\underline{\beta}:=\bigotimes_{v\in S_\CC} \underline{\beta_v}\otimes \bigotimes_{v\in S_\RR}(\beta_v)_1$ on $\GL_2(F_\infty)$ and $\H_\infty$, respectively.\\

\subsection{Automorphic forms}
Let $\ZC: \AA_F^*/F^*\to \CC^*$ be a Hecke character that is trivial at the archimedian places. %
We also denote by $\ZC$ the corresponding character on $Z(\AA_F)$ under the isomorphism $\AA_F^*\to Z(\AA_F)$, $a\mapsto \left(\begin{smallmatrix}a&0\\0&a \end{smallmatrix}\right)$.

\begin{defi}\label{aut}
An {\it automorphic cusp form of parallel weight $\underline{2}$ %
with central character $\ZC$} is a map $\Ophi:G(\AA_F)\to V$ such that

\begin{enumerate}
 \item[(i)] $\Ophi(z \gamma g)=\ZC(z)\Ophi(g)$ for all $g\in G(\AA)$, $z\in Z(\AA)$, $\gamma\in G(F)$.
 \item[(ii)] $\Ophi(g k_\infty )=\Ophi(g)\rho(k_\infty)$ for all $k_\infty\in K_\infty$, $g\in G(\AA)$ (considering $V$ as a row vector).
 \item[(iii)] $\Ophi$ has ``moderate growth`` on $B'_{\AA}:=\lbrace \left(\begin{matrix}
                                                                y & x \\ 0 & 1
                                                               \end{matrix}\right) \in G(\AA)\rbrace$, %
             i.e.  $\exists C,\lambda\;\forall A\in B'_{\AA}: \lVert \Ophi(A)\rVert \leq C\. \sup(|y|^\lambda, |y|^{-\lambda})$ (for any fixed norm $\lVert \cdot\rVert$ on $V$);
      
      and $%
	  \Ophi|_{G(\AA_\infty)}\. \underline{\omega}$ is the pullback of a harmonic form $\omega_\Ophi=f_\Ophi\. \underline{\beta}$ on $\H_\infty$.
	  
 \item[(iv)] There exists a compact open subgroup $K'\subeq G(\AA^\infty)$ such that $\Ophi(gk)=\Ophi(g)$ for all $g\in G(\AA)$ and $k\in K'$. %
 \item[(v)] For all $g\in G(\AA_F)$,
\[ \int_{\AA_F/F} \Ophi\left(\begin{pmatrix}
                                    1&x\\0&1
                                   \end{pmatrix} g\right)\; dx=0.\qquad\text{\it{(``Cuspidality'')}}\]
\end{enumerate}
We denote by $\A_0(G, \harm, \underline{2},\ZC)$ the space of all such maps $\Ophi$.\\
\end{defi}

For each $g^\infty\in\AA_F^\infty$, let $\omega_\Ophi(g^\infty)$ be the restriction %
of $\Ophi(g^\infty,\cdot)\cdot\underline{\omega}$ from $G(\AA_F^\infty)$ to $\H_\infty$; it is a $(d+1)$-form on $\H_\infty$.\\ %

We want to integrate $\omega_\Ophi(g^\infty)$ between two cusps of the space $\H_{m_0}$%
. (We will identify each $x\in\PP^1(F)$ with its corresponding cusp in $\overline{\H_{m_0}}$ in the following.) The geodesic between the cusps $x\in F$ and $\infty$ in $\overline{\H_{m_0}}$ is the line $\{x\}\times \RR^*_+\subeq\H_{m_0}$    %
and the integral of $\omega_\Ophi$ along it is finite since $\Ophi$ is uniformly rapidly decreasing: 

\begin{theorem} (Gelfand, Piatetski-Shapiro)
 An automorphic cusp form $\Ophi$ is rapidly decreasing modulo the center on a fundamental domain $\F$ of\; $\GL_2(F)\backslash \GL_2(\AA_F)$; \linebreak i.e. there exists an integer $r$ such that for all $N\in \NN$ there exists a $C>0$ such that
\[\Ophi(zg)\le C|z|^r\|g\|^{-N}\]
for all $z\in Z(\AA_F)$, $g\in \F\cap\SL_2(\AA_F)$. Here $\|g\|:=\max\{|g_{i,j}|, |(g^{-1})_{i,j}|\}_{i,j\in\{1,2\}}$.
\end{theorem}
\noindent(See \cite{CKM}, Thm. 2.2; or \cite{Kur78}, (6) for quadratic imaginary $F$.) \\

In fact, the integral of $\omega_\Ophi(g^\infty)$ along $\{x\}\times \RR^*_+\subeq\H_{m_0}$ equals the integral of $\Ophi(g^\infty,\cdot)\cdot\underline{\omega}$ along a path $g_t\in \GL_2(F_{\infty_0})$, $t\in\RR^*_+$, where we can choose \[g_t=\frac{1}{\sqrt{t}}\left( \begin{matrix}t&x\\0&1 \end{matrix}\right)= \left(\begin{matrix}\frac{1}{\sqrt{t}}& \frac{x}{\sqrt{t}} \\ 0& \sqrt{t}  \end{matrix}\right),\]
 and thus have $\|g_t\|=\sqrt{t}$ for all $t\gg 0$, $\|g_t\|=C\frac{1}{\sqrt{t}}$ for $t\ll 1$, so the integral $\int_{x}^{\infty} \omega_\Ophi(g^\infty) \in \Omega^d_{\harm}(\H^0_\infty)$ is well-defined by the theorem.\\

For any two cusps $a, b\in \PP^1(F)$, we now define \[\int_{a}^{b} \omega_\Ophi(g^\infty):=\int_{a}^{\infty} \omega_\Ophi(g^\infty) -\int_{b}^{\infty} \omega_\Ophi(g^\infty) \in \Omega^d_{\harm}(\H^0_\infty).\] %

Since $\Ophi$ is {\it uniformly} rapidly decreasing ($\|g_t\|$ does not depend on $x$, for $t\gg 0$), this integral along the path $(a,0)\to(a,\infty)=(b,\infty)\to(b,0)$ in $\bar{\H}_{m_0}$ is the same as the limit (for $t\to\infty$) of the integral along $(a,0)\to(a,t)\to(b,t)\to(b,0)$; and since $\omega_\Ophi$ is harmonic (and thus integration is path-independent within $\H_{m_0}$) the latter is in fact independent of $t$, so equality holds for each $t>0$, or along any path from $(a,0)$ to $(b,0)$ in $\H_{m_0}$. Thus $\int_{a}^{b} \omega_\Ophi(g^\infty)$ equals the integral of $ \omega_\Ophi(g^\infty)$ along the geodesic from $a$ to $b$, and we have 

\[\int_{a}^{b} \omega_\Ophi(g^\infty)+\int_{b}^{c} \omega_\Ophi(g^\infty)=\int_{a}^{c} \omega_\Ophi(g^\infty)\]
for any three cusps $a,b,c\in\PP^1(F)$. 
Let $\Div(\PP^1(F))$ denote the free abelian group of divisors of $\PP^1(F)$, and let $\M:=\Div_0(\PP^1(F))$ be the subgroup of divisors of degree 0. \\
 
We can extend the definition of the integral linearly to get a homomorphism 

\[\M\to  \Omega^d_{\harm}(\H^0_\infty),\quad   m\mapsto \int_{m} \omega_\Ophi(g^\infty).\\ \]

For $\gamma\in G(F)^+$, $g\in G(\AA^\infty)$, $m\in\M$ and $x_\infty^0\in G(F_{S^0_\infty})$, we have
\begin{eqnarray*}
\gamma^* \left(\int_{\gamma m}\omega_\Ophi(\gamma g)\right)(x_\infty^0) & = & \int_{\gamma m}\omega_\Ophi(\gamma g) (\gamma x_\infty^0)\\
							    & = & \int_{\gamma m}\Ophi(\gamma g, \gamma x_\infty^0, *)\cdot \omega \\
					  & = & \int_{\gamma m}\Ophi( g,  x_\infty^0, \gamma^{-1} *) \cdot\underline{\omega} \qquad\quad\text{ (by (i) of definition \ref{aut})}\\
					  & = & \int_{m}\Ophi( g,  x_\infty^0, *) \cdot\underline{\omega} \quad\text{ (since $\underline{\omega}$ is $G(F_\infty)$-left invariant)} \\
					    &=& \int_m \omega_\Ophi(g) (x_\infty^0),
\end{eqnarray*}
i.e. 
\begin{equation}\label{G(F)-Invarianz}\gamma^* \left(\int_{\gamma m}\omega_\Ophi(\gamma g)\right)= \int_m \omega_\Ophi(g). \end{equation}\\

Now let $\m$ be an ideal of $F$ prime to $p$, let $\chi_Z$ be a Hecke character of conductor dividing $\m$, and $\underline{\alpha_1},\underline{\alpha_2}$ as above.

\begin{defi}\label{S_2}
We define $S_2(G,\m,{\underline{\alpha_1},\underline{\alpha_2}})$ %
to be the $\CC$-{vector space} of all maps \[\Phi:G(\AA^p)\to \B^{\alphaeinszwei}(F_p, V)=\Hom (\B_{\alphaeinszwei}(F_p,\CC),V)\] such that:

\begin{enumerate}
 \item[(a)] $\Ophi$ is ``almost'' $K_0(\m)$-invariant (in the notation of \cite{Gelbart}), i.e. $\Ophi(gk)=\Ophi(g)$ for all $g\in G(\AA^p)$ and $k\in\prod_{v\nmid \m p} G(\OO_v)$, and  $\Ophi(gk)=\ZC(a)\Ophi(g)$ for all $v|\m$,  
$k=\begin{pmatrix} a&b\\c&d\end{pmatrix}\in K_0(\m)_v$ and $g\in G(\AA^p)$.

 \item[(b)] For each $\psi\in \B_{\alphaeinszwei}(F_p,\CC)$, the map 
    \[ \langle \Phi,\psi \rangle: G(\AA)=G(F_p)\times G(\AA^p)\to V, \; (g_p,g^p)\mapsto \Phi(g^p)(g_p\psi)\]
    lies in  $\A_0(G, \harm, \underline{2},\ZC)$.
\end{enumerate}
\end{defi}

Note that (a) implies that $\phi$ is $K'$-invariant for some open subgroup $K'\subeq K_0(\m)^p$ of finite index (\cite{By}/\cite{Weil}).

\subsection{Cohomology of $\GL_2(F)$}

Let $M$ be a left $G(F)$-module and $N$ an $R[H]$-module, for a ring $R$ and a subgroup $H\subeq G(F)$. Let $S\subeq S_p$ be a set of primes of $F$ dividing $p$; as above, let $\chi=\chi_Z$ be a Hecke character of conductor $\m$ prime to $p$. 

\begin{defi}
For a compact open subgroup $K\subeq K_0(\m)^S\subeq G(\AA^{S,\infty})$, we denote by $\A_f(K, S, M;N)$ %
 the $R$-module of all maps $\Phi:G(\AA^{S,\infty})\times M \to N$ such that \begin{enumerate}
                                                                              \item $\Phi (gk,m)=%
\Phi(g,m)$ for all $g\in G(\AA^{S,\infty})$, $m\in M$,  $k\in\prod_{v\nmid \m p} G(\OO_v)$; 
\item  $\Phi(gk)=\ZC(a)\Phi(g)$ for all $v|\m$,  $k=\begin{pmatrix} a&b\\c&d\end{pmatrix}\in K_0(\m)_v$ and $g\in G(\AA^{S,\infty})$, $m\in M$.                                                                        
\end{enumerate}

We denote by $\A_f(S, M;N)$ the union of the $\A_f(K, S, M;N)$ over all compact open subgroups $K$.
\end{defi}

$\A_f(S, M;N)$ is a left $G(\AA^{S,\infty})$-module via $(\gamma\cdot\Phi)(g,m):=\Phi(\gamma^{-1}g,m)$ %
and has a left $H$-operation given by $(\gamma\cdot\Phi) (g,m):=\gamma\Phi(\gamma^{-1}g,\gamma^{-1}m)$, commuting with the $G(\AA^{S,\infty})$-operation.\\

In contrast to our previous notation, we consider two subsets $S_1\subeq S_2\subeq S_p$ in this section. We put $(\alphaeinszwei)_{S_1}:=\{(\alpha_{\p,1}, \alpha_{\p,2}) |\p\in S_1 \}$%
, we set
\[\A_f((\alphaeinszwei)_{S_1}, S_2, M;N)=\A_f(S_2, M; \B^{(\alphaeinszwei)_{S_1}}(F_{S_1},N));\]

we write $\A_f(\m, (\alphaeinszwei)_{S_1}, S_2, M;N) %
:=\A_f(K_0(\m), (\alphaeinszwei)_{S_1}, S_2, M;N)$. 
If $S_1=S_2$, we will usually drop $S_2$ from all these notations.

We have a natural identification of $\A_f(\m, (\alphaeinszwei)_S, M;N)$ with the space of maps $G(\AA^{S,\infty})\times M\times  \B_{(\alphaeinszwei)_S}(F_S, R) \to N$ that are ``almost'' $K$-invariant.\\

Let $S_0\subeq S_1\subeq S_2\subeq S_p$ be subsets. %
The pairing \eqref{sl. pairing} induces a pairing
\begin{equation}
 \langle\cdot,\cdot\rangle: \A_f((\alphaeinszwei)_{S_1}, S_2, M;N)\times \B_{(\alphaeinszwei)_{S_0}}(F_{S_0}, R) \to \A_f((\alphaeinszwei)_{S_0}, S_2, M;N),
\end{equation}
which, when restricting to $K$-invariant elements, induces an isomorphism 
\begin{equation}
 \A_f(K, (\alphaeinszwei)_{S_1}, S_2, M;N)\iso \B^{(\alphaeinszwei)_{S_1\minus S_0}}(F_{S_1\minus S_0}, \A_f(\alphaeinszwei)_{S_0}, S_2, M;N).
\end{equation}
Putting $S_0:=S_1\minus \{\p\}$ for a prime $\p\in S_1$, we specifically get an isomorphism
\[\A_f(K, (\alphaeinszwei)_{S_1}, S_2, M;N)\iso \B^{\alpha_{\p,1}, \alpha_{\p,2}}(F_\p, \A_f(\alphaeinszwei)_{S_0}, S_2, M;N).\]

Lemmas \ref{sphericalResolution} and \ref{specialResolution} now immediately imply the following:%

\begin{lemma}\label{resolutions}
Let $S\subeq S_p$, $\p \in S$, $S_0:=S\minus \{\p\}$. Let $K\subeq G(\AA^{S,\infty})$ be a compact open subgroup. \\
(a) If $\pi_{\alpha_{\p,1},\alpha_{\p,2}}$ is spherical,  we have exact sequences
\[ 0\to \A_f(K,(\alphaeinszwei)_S,M; N) \to Z\xrightarrow{\R-\muu_\p} Z\to 0\] and \[ 0\to Z\to \A_f(K_0,(\alphaeinszwei)_{S_0}, M;N) \xrightarrow{T-a_\p} \A_f(K_0,(\alphaeinszwei)_{S_0}, M;N) \to 0\]
for a $G(\AA^{S_0,\infty})$-module $Z$ and a compact open subgroup $K_0=K\times K_\p$ of $G(\AA^{S_0,\infty})$.\\

(b) If $\pi_{\alpha_{\p,1},\alpha_{\p,2}}$ is special (with central character $\chi_\p$)
,  we have exact sequences
\[ 0\to \A_f(K,(\alphaeinszwei)_S,M; N) \to Z'\to Z\to 0\]
and 
\begin{align*}
0\to &Z\to \A_f(K_0,(\alphaeinszwei)_{S_0}, M;N)^2 \to \A_f(K_0,(\alphaeinszwei)_{S_0}, M;N)^2 \to 0,\\
0\to &Z'\to \A_f(K'_0,(\alphaeinszwei)_{S_0}, M;N)^2 \to \A_f(K'_0,(\alphaeinszwei)_{S_0}, M;N)^2 \to 0,
\end{align*}%
with $Z:= \A_f(K_0, (\alphaeinszwei)_{S_0}, S, M; N(\chi_{\p})
)$ and $Z':=\A_f(K'_0, (\alphaeinszwei)_{S_0}, S, M; N(\chi_{\p})
)$, where %
$K_0=K\times K_\p$ and $K_0'=K\times K'_\p$ are compact open subgroups of $G(\AA^{S_0,\infty})$.\\
\end{lemma}

\begin{proposition}\label{flattensor}
 Let $S\subeq S_p$ and let $K$ be a compact open subgroup of $G(\AA^{S,\infty})$. \\

(a) For each flat $R$-module $N$ (with trivial $G(F)$-action), the canonical map 
\[H^q(G(F)^+,\A_f(K, (\alphaeinszwei)_S, \M; R))\otimes_R N \to H^q(G(F)^+,\A_f(K,(\alphaeinszwei)_S,\M; N))\]
is an isomorphism for each $q\ge 0$.\\

(b) If $R$ is finitely generated as a $\ZZ$-module, then $H^q(G(F)^+,\A_f(K, (\alphaeinszwei)_S, \M; R)$ is finitely generated over $R$.
\end{proposition}

\begin{proof} (cf. \cite{Sp}, Prop. 5.6)\\
(a) The exact sequence of abelian groups $0\to\M\to\Div(\PP^1(F))\iso \Ind^{G(F)}_{B(F)}\ZZ\to\ZZ\to 0$ induces a short exact sequence of $G(\AA^{S,\infty})$-modules
\begin{equation}\label{flatexact}\begin{split}
 0\to \A_f(K,(\alphaeinszwei)_S,\ZZ; N)\to \Coind^{G(F)^+}_{B(F)^+} \A_f(K,(\alphaeinszwei)_S,\ZZ; N)\qquad\qquad \\  \qquad\qquad\to\A_f(K,(\alphaeinszwei)_S,\M;N)\to 0.
            \end{split}
\end{equation}
Using the five-lemma on the associated diagram of long exact cohomology sequences $H^q(\cdot,R)\otimes_R N$ (which is exact due to flatness) and $H^q(\cdot,N)$ , it is enough to show that \eqref{flatexact} holds for $\A_f(K, (\alphaeinszwei)_S,\ZZ;\cdot)$ and $\Coind^{G(F)^+}_{B(F)^+}\A_f(K,(\alphaeinszwei)_S,\ZZ;\cdot)$ instead of $\A_f(K,(\alphaeinszwei)_S,\M;\cdot)$. By lemma \ref{resolutions}, it is furthermore enough to consider the case $S=\emptyset$. Since $\A_f(K,\ZZ;N)\iso\Coind^{G(\AA^\infty)}_{K} N$, we thus have to show that
\begin{align*}
 H^q(G(F)^+,\Coind^{G(\AA^\infty)}_{K} R)\otimes_R N &\to H^q(G(F)^+,\Coind^{G(\AA^\infty)}_K N),\\
H^q(B(F)^+,\Coind^{G(\AA^\infty)}_{K} R)\otimes_R N &\to H^q(B(F)^+,\Coind^{G(\AA^\infty)}_K N)
\end{align*}
are isomorphisms for all $q\ge0$ and all flat $R$-modules $N$.\\

Since every flat module is the direct limit of free modules of finite rank, it suffices to show that $N\mapsto H^q(G(F)^+,\Coind^{G(\AA^\infty)}_K N)$ and $N\mapsto H^q(B(F)^+,\Coind^{G(\AA^\infty)}_K N)$ commute with direct limits.\\

For $g\in G(\AA^\infty)$, put $\Gamma_g:=G(F)^+\cap gKg^{-1}$, By the strong approximation theorem, $G(F)^+\backslash G(\AA^\infty)/K$ is finite. Choosing a system of representatives $g_1,\ldots,g_n$, we have 
\[H^q(G(F)^+,\Coind^{G(\AA^\infty)}_K N)= \bigoplus_{i=1}^n H^q(\Gamma_{g_i}, N). \]
Since the groups $\Gamma_g$ are arithmetic, they are of type (VFL), and thus the functors $N\mapsto H^q(\Gamma_g, N)$ commute with direct limits by \cite{Serre2}, remarque on p. 101. \\

Similarly, the Iwasawa decomposition $G(\AA^\infty)=B(\AA^\infty)\prod_{v\nmid \infty} G(\OO_v)$ implies that $B(F)^+\backslash G(\AA^\infty)/K$ is finite. %
Therefore, the same arguments show that \linebreak$N\mapsto H^q(B(F)^+,\Coind^{G(\AA^\infty)}_K N)$ commutes with direct limits.\\

(b) This follows along the same line of reasoning as (a), since $H^q(\Gamma_g, R)$ is finitely generated over $\ZZ$ by \cite{Serre2}, remarque on p. 101.
\end{proof}

With the notation as above, we define
\[H^q_*(G(F)^+, \A_f((\alphaeinszwei)_S, M; R)):=\varinjlim H^q(G(F)^+,\A_f(K, (\alphaeinszwei)_S, M; R))\]
where the limit runs over all compact open subgroups $K\subeq G(\AA^{S,\infty})$; and similarly define $H^q_*(B(F)^+,\A_f((\alphaeinszwei)_S, \M; R)$. The proposition immediately implies

\begin{corollary}
 Let $R\to R'$ be a flat ring homomorphism. Then the canonical map
\[H^q_*(G(F)^+,\A_f((\alphaeinszwei)_S, \M; R))\otimes_R R' \to H^q_*(G(F)^+,\A_f( (\alphaeinszwei)_S, \M; R')\]
is an isomorphism, for all $q\ge 0$.
\end{corollary}

If $R=k$ is a field of characteristic zero, $H^q_*(G(F)^+,\A_f( (\alphaeinszwei)_S, M; R)$ is a smooth %
$G(\AA^{S,\infty})$-module, and we have 
\[H^q_*(G(F)^+,\A_f( (\alphaeinszwei)_S, M; k)^K=H^q(G(F)^+,\A_f(K, (\alphaeinszwei)_S, M; k).\]

We identify $G(F)/G(F)^+$ with the group $\Sigma=\{\pm1\}^r$ via the isomorphism \[G(F)/G(F^+)\xrightarrow{\det} F^*/F^*_+\iso \Sigma\] (with all groups being trivial for $r=0$). Then $\Sigma$ acts on $H^q_*(G(F)^+,\A_f( (\alphaeinszwei)_S, M; k)$ and $H^q(G(F)^+,\A_f(K, (\alphaeinszwei)_S, M; k)$ by conjugation.\\
For $\pi\in\mathfrak{A}_0(G,\underline{2})$   
and $\underline{\mu}\in\Sigma$, we write $H^q_*(G(F)^+,\cdot)_{\pi,\underline{\mu}}:=\Hom_{G(\AA^{S,\infty})}(\pi^S, H^q_*(G(F)^+,\cdot))_{\underline{\mu}}$.\\

Now we can show that $\pi$ occurs with multiplicity $2^r$ in $H^q_ *(G(F)^+,\A_f((\alphaeinszwei)_S, \M; k)$:

\begin{proposition}\label{harder}
 Let $\pi\in \mathfrak{A}_0(G,\underline{2},\ZC, \underline{\alpha_1},\underline{\alpha_2})$, $S\subeq S_p$. Let $k$ be a field which contains the field of definition of $\pi$. %
Then for every $\underline{\mu}\in\Sigma$, we have
\begin{equation}
 H^q_*(G(F)^+,\A_f((\alphaeinszwei)_S, \M; k)_{\pi,\underline{\mu}} = \begin{cases}
                                                                       k, \quad \text{if $q=d$;}\\
									 0, \quad \text{if $q\in\{0,\ldots, d-1\}$}
                                                                      \end{cases}
\end{equation}
\end{proposition}
\begin{proof} (cf. \cite{Sp}, prop. 5.8)\\
 First, assume $S=\emptyset$. The sequence \eqref{flatexact} induces a cohomology sequence

\[\begin{split}
   \ldots\to H^q_*(G(F)^+,\A_f(\ZZ, k))\to H^q_*(B(F)^+,\A_f(\ZZ, k)) \to H^q_*(G(F)^+,\A_f(\M, k))\\ 
\to H^{q+1}_*(G(F)^+,\A_f(\ZZ, k))\to\ldots 
  \end{split} \]

Harder (\cite{Ha}) has determined the action of $G(\AA^\infty)$ on $H^q_*(G(F)^+,\A_f(\ZZ, k))$ and $H^q_*(G(F)^+,\A_f(\ZZ, k))$: For $q<d$, $H^q_*(G(F)^+,\A_f(\ZZ, k))$ is a direct sum of one-dimensional representations; for $q=d$ there is a $G(\AA^\infty)$-stable decomposition 
\[H^{d+1}_*(G(F)^+,\A_f(\ZZ, k))=H^{d+1}_{\cusp}\oplus H^{d+1}_{\res}\oplus H^{d+1}_{\Eis},\]
 with the last two summands again being direct sums of one-dimensional representations, and 
\[H^{d+1}_{\cusp}(G(F)^+,\A_f(\ZZ, k))_{\pi,\underline{\mu}}\iso k\]
 (\cite{Ha}, 3.6.2.2); $H^q_*(B(F)^+,\A_f(\ZZ, k))$ always decomposes into one-dimensional \linebreak%
{$G(\AA^\infty)$-representations}. Since $\pi^S$ does not map to one-dimensional representations, this proves the claim for $S=\emptyset$. 

Now for $S=S_0\cup\{\p\}$ and $\pi_\p$ spherical, lemma \ref{resolutions}(a) and the statement for $S_0$ give an isomorphism 
\[H^q_*(G(F)^+,\A_f((\alphaeinszwei)_{S_0}, \M; k))_{\pi,\underline{\mu}}\iso H^q_*(G(F)^+,\A_f((\alphaeinszwei)_S, \M; k))_{\pi,\underline{\mu}}\]
 since the Hecke operators $T_\p$, $\R_\p$ act on the left-hand side by multiplication with $a_\p$ or $\muu_\p$, respectively. If $\pi_\p$ is special, we can similarly deduce the statement for $S$ from that for $S_0$, using the first exact sequence of lemma \ref{resolutions}(b) (cf. \cite{Sp}), since the results of \cite{Ha} also hold when twisting $k$ by a (central) character.
\end{proof}

\subsection{Eichler-Shimura map}%

Given a subgroup $K_0(\m)^p\subeq G(\AA^{p,\infty})$ as above, there is a map 

\[I_0: S_2(G,\m,\alphaeinszwei) \to H^0(G(F)^+,\A_f(\m, \alphaeinszwei, \M; \Omega^d_{\harm}(\H^0_\infty)))\]
given by
\[I_0(\Phi):(\psi, (g,m))\mapsto \int_m\omega_{\langle \Phi,\psi\rangle} (1_p,g), \]
for $\psi\in \B_\alphaeinszwei(F_p,\CC),  g\in G(\AA^{p,\infty}), m\in \M$, where $1_p$ denotes the unity element in $G(F_p)$.

This is well-defined since both sides are ``almost'' $K_0(\m)$-invariant, and the $G(F)^+$-invariance of $I_0(\Phi)$ follows from the similar invariance for differential forms, and the definition of the $G(F)^+$-operations on $\A_f(M,N)$, $\B^\alphaeinszwei(F_p,N)$ and $\Omega^d_{\harm}(\H^0_\infty)$: For each $\psi\in \B_\alphaeinszwei(F_p,\CC), g\in G(\AA^{p,\infty}), m\in \M$, we have

\begin{eqnarray*}
 (\gamma I_0(\Phi)) (\psi, (g, m))&&= \gamma I_0(\Phi)(\gamma^{-1}\psi, (\gamma^{-1} g, \gamma^{-1} m))\\
				&&= \gamma\cdot\int_{\gamma^{-1} m}\omega_{\langle\Phi, \gamma^{-1}\psi\rangle} (1_p, \gamma^{-1}g)\\
				&&= (\gamma^{-1})^*\int_{\gamma^{-1} m}\omega_{\langle\Phi, \gamma^{-1}\psi\rangle} (1_p, \gamma^{-1}g)\\
				&&=\int_m\omega_{\langle\Phi, \gamma^{-1}\psi\rangle}(\gamma 1_p, g) \qquad\qquad\text{(by \eqref{G(F)-Invarianz})}\\
				&&= I_0(\Phi)  (\psi, (g, m)).
\end{eqnarray*}

We have a complex $\A_f(m,\alphaeinszwei,\M;\CC)\to C^\bullet:=\A_f(\m,\alphaeinszwei,\M;\Omega_{\harm}^\bullet(\H^0_\infty))$. %
Therefore we get a map
\begin{equation}\label{(60)}
S_2(G,\m,\alphaeinszwei)\to H^d( G(F)^+,\A_f(\m, \alphaeinszwei, \M; \CC)) %
\end{equation}
by composing $I_0$ with the edge morphism $H^0(G(F)^+,C^d)\to %
H^d(G(F)^+, \A_f(\m, \alphaeinszwei, \M; \CC ))%
$ of the spectral sequence
\[H^q(G(F)^+,C^p)\implies H^{p+q}(G(F)^+,C^\bullet).\]

Using the map $\delta^{\alphaeinszwei}: \B^{\alphaeinszwei}(F,V)\to \Dist (F_p^*,V)$ from section \ref{semilocal}, we next define a map 

\begin{equation}
\Delta_V^{\alphaeinszwei}: S_2(G,\m, \alphaeinszwei)\to \D(S_1, V) 
\end{equation}
by
\[ \Delta_V^{\alphaeinszwei}(\Phi)(U, x^p)= \delta^{\alphaeinszwei}\left(\Phi\begin{pmatrix}
                                    x^p&0\\0&1
                                   \end{pmatrix}\right)(U)\]
 for $U\in\CCoo(F_{S_1}\times F_{S_2}), x^p\in \II^p$, and we denote by $\Delta^{\alphaeinszwei}: S_2(G,\m, \alphaeinszwei)\to \D(S_1, \CC)$ its (1,...,1)th coordinate function (i.e. corresponding to the harmonic forms $\bigotimes_{v|\infty}(\omega_v)_1$, $\bigotimes_{v|\infty}(\beta_v)_1$ in section \ref{upper}):
\[\Delta^{\alphaeinszwei}(\Phi)(U, x^p)= \delta^{\alphaeinszwei}\left(\Phi\begin{pmatrix}
                                    x^p&0\\0&1
                                   \end{pmatrix}\right)_{(1,\ldots,1)}(U).\] 
Since for each complex prime $v$,  $S^1\iso \SU(2)\cap T(\CC)$ operates via $\rho_v$  on $\Phi$, $\Delta^{\alphaeinszwei}$ is easily seen to be $S^1$-invariant, i.e. it lies in $\D'(S_1,\CC)$.\\

We also have a natural (i.e. commuting with the complex maps of each $C^\bullet$) family of maps
\begin{equation}\label{family}
 \A_f(\m,\alphaeinszwei,\M,\Omega_{\harm}^i(\H^0_\infty))\to \D_f(S_1, \Omega^i(U^0_\infty,\CC)) 
\end{equation}
{ for all $i\ge 0$}, and
\begin{equation}\label{fam2}
 \A_f(\m,\alphaeinszwei,\M,\CC)\to \D_f(S_1, \CC)
\end{equation}
(the $i=-1$-th term in the complexes), by
mapping $\Phi\in \A_f(\m,\alphaeinszwei,\M,\.)$ first to 
\[(U,x^{p,\infty})\mapsto \Phi \left(\begin{pmatrix}x^{p,\infty}&0\\0&1 \end{pmatrix},\infty - 0\right)(\delta_\alphaeinszwei(1_U))\in \Omega^i_{\harm}(\H^0_\infty) \text{ resp. }\in\CC, \] 
and then for $i\ge 0$ restricting the differential forms to $\Omega^i(U^0_\infty)$ via \[U^0_\infty=\prod_{v\in S^0_\infty} \RR_+^*\into\prod_{v\in S^0_\infty}\H_v=\H^0_\infty.\]

One easily checks %
that \eqref{family} and \eqref{fam2} are compatible with the homomorphism of ``acting groups'' ${F^*}'\into G(F)^+, x\mapsto \bigl(\begin{smallmatrix}x&0\\0&1 \end{smallmatrix}\bigr)$, so we get induced %
maps in cohomology %
\begin{equation}\label{(64)_0}
 H^0(G(F)^+,\A_f(\m,\alphaeinszwei,\M,\Omega_{\harm}^d(\H^0_\infty)))\to H^0(\D_f(S_1, \Omega^d(U^0_\infty,\CC)))
\end{equation}
and
\begin{equation}\label{(64)}
H^d(G(F)^+,\A_f(\m,\alphaeinszwei,\M, \CC)) \to H^d({F^*}', \D_f(S_1, \CC)),
\end{equation}
which are linked by edge morphisms of the respective spectral sequences to give a commutative diagram (given in the proof below).\\

\begin{proposition}\label{square}
We have a commutative diagram:

\[\xymatrix{ S_2(G,\m,\underline{\alpha_1},\underline{\alpha_2}) \ar[rr]^-{\eqref{(60)}}\ar[d]^{\Delta^{\alphaeinszwei}} & &H^d(G(F)^+, \A_f(\m,\alphaeinszwei, \M,\CC))\ar[d]^{\eqref{(64)}}\\
	    \D'(S_1,\CC) \ar[rr]^-{\phi\mapsto\kappa_\phi} & & H^d\bigl({F^*}', \mathcal{D}_f(S_1, \CC)\bigr)}\label{commSquare}\]
\end{proposition}

\begin{proof}
The given diagram factorizes as
\[\footnotesize \xymatrix{ S_2(G,\m,\underline{\alpha_1},\underline{\alpha_2}) \ar[r]^-{I_0}\ar[d]^{\Delta^{\alphaeinszwei}} &H^0(G(F)^+,\A_f(\m,\alphaeinszwei,\M,\Omega_{\harm}^d(\H^0_\infty)))\ar[r]\ar[d]^{\eqref{(64)_0}} &H^d(G(F)^+, \A_f(\m,\alphaeinszwei, \M,\CC))\ar[d]^{\eqref{(64)}}\\
	    \D'(S_1,\CC) \ar[r] %
 & H^0(\D_f(S_1, \Omega^d(U^0_\infty,\CC)))\ar[r] & H^d\bigl({F^*}', \mathcal{D}_f(S_1, \CC)\bigr)}\]

The right-hand square is the naturally commutative square mentioned above; the commutativity of the left-hand square can be checked by hand:\\	

Let $\Phi\in S_2(G,\m,\underline{\alpha_1},\underline{\alpha_2})$. Then $I_0(\Phi)$ is the map $(\psi, (g,m))\mapsto \int_m\omega_{\langle \Phi,\psi\rangle} (1_p,g)$, which is mapped under \eqref{(64)_0} to
\begin{eqnarray*}
(U, x^{p,\infty})&\mapsto& \left.\int_0^\infty {\omega_{\langle \Phi, \delta_\alphaeinszwei(1_U) \rangle}\left(1_p, \begin{pmatrix} x^{p,\infty}&0\\0&1\end{pmatrix}\right)}\middle|^{}_{{\raisebox{-7pt}{$U^0_\infty$}} } \right. \\
&=& \int_0^\infty \Phi_{(1,\ldots,1)} \begin{pmatrix}   x^p&0\\0&1 \end{pmatrix} (\delta_{\alphaeinszwei}(1_U)) \frac{dt_0}{t_0} \frac{dt_1}{t_1}\ldots \frac{dt_d}{t_d};
\end{eqnarray*}

along the other path, $\Phi$ is mapped under $\Delta^{\alphaeinszwei}$ to the map 

\[(U, x^p)\mapsto \delta^{\alphaeinszwei}\left(\Phi\begin{pmatrix}
                                    x^p&0\\0&1
                                   \end{pmatrix}\right)_{(1,\ldots,1)}(U)= \Phi_{(1,\ldots,1)} \begin{pmatrix}   x^p&0\\0&1 \end{pmatrix} (\delta_{\alphaeinszwei}(1_U))\]
and then also to 

\[ (U, x^{p,\infty})\mapsto \int_0^\infty \Phi_{(1,\ldots,1)} \begin{pmatrix}   x^p&0\\0&1 \end{pmatrix} (\delta_{\alphaeinszwei}(1_U)) \, d^\times r_0\, d^\times r_1\,\ldots \,d^\times r_d \]
(with $x^p=(x^{p,\infty}, r_0, r_1, \ldots, r_d)$).
\end{proof}

\subsection{%
Whittaker model}

We now consider an automorphic representation $\pi=\tensor_\nu \pi_\nu \in\mathfrak{A}_0(G,\underline{2},\ZC, \underline{\alpha_1},\underline{\alpha_2})$. %
Denote by $\cc(\pi):=\prod_{v\text{ finite}}\cc(\pi_v)$ the conductor of $\pi$.

Let $\chi:\II^\infty\to\CC^*$ be a unitary %
character of the finite ideles; for each finite place $v$, set $\chi_v=\chi|_{F_v^*}$. 
For each prime $v$ of $F$, let $\WW_v$ denote the Whittaker model of $\pi_v$. For each finite and each real prime, we choose $W_v\in\WW_v$ such that the local L-factor equals the local zeta function at $g=1$, i.e. such that

\begin{equation}\label{L=Zeta}
 L(s,\pi_v\tensor\chi_v)=\int_{F_v^*}W_v\begin{pmatrix}
                                    x&0\\0&1
                                   \end{pmatrix} \chi_v(x) |x|^{s-\einh} \; d^\times x
\end{equation}
for any unramified quasi-character $\chi_v:F_v^*\to\CC^*$ and $\Re(s)\gg 0$. %

This is possible by \cite{Gelbart}, Thm. 6.12 (ii); and by loc.cit., Prop. 6.17, $W_v$ can be chosen such that $\SO(2)$ operates on $W_v$ via $\rho_v$ for real archimedian $v$, and 
is  ``almost'' $K_0(\cc(\pi_v))$-invariant for finite $v$. %

For complex primes $v$ of $F$, we can also choose a $W_v$ satisfying \eqref{L=Zeta} and which behaves well with respect to the $\SU(2)$-action $\rho_v$, as follows:\\

By \cite{Kur77}, there exists a three-dimensional function \[\underline{W_v}=(W_v^0,W_v^1,W_v^2):G(F_v)\to\CC^3\] such that $W_v^i\in\WW_v$ for all $i$, and such that $\SU(2)$ operates by the right via $\rho_v$ on $\underline{W_v}$; i.e. for all $g\in G(F_v)$ and $h=\begin{pmatrix}
                                    u&v\\-\bar{v}&\bar{u}
                                   \end{pmatrix} \in \SU(2)$, we have
\[\underline{W_v}(gh)=\underline{W_v}(g) M_3(h),\]
where
\[M_3(h)=\begin{pmatrix}   u^2&2uv&v^2\\-u\bar{v}&u\bar{u}-v\bar{v}&v\bar{u}\\\bar{v}^2&-2\bar{uv}&\bar{u}^2
                                   \end{pmatrix}. \]

Note that $W_v^1$ is thus invariant under right multiplication by a diagonal matrix $\begin{pmatrix}
                                    u&0\\0&\bar{u}
                                   \end{pmatrix}$ with $u\in S^1\subeq \CC$. Since $\pi_v$ has trivial central character for archimedian $v$ by our assumption, a function in $\WW_v$ is also invariant under $Z(F_v)$. %
Thus we have \[W_v^1\left(g \begin{pmatrix}
                                    u&0\\0&1
                                   \end{pmatrix}\right)= W_v^1(g)\quad\text{ for all }g\in G(F_v), \;u\in S^1.\] %

$W_v^1$ can be described explicitly in terms of a certain Bessel function, as follows.
The modified Bessel differential equation of order $\alpha\in\CC$ is

\[x^2\frac{d^2 y}{dx^2}	+x \frac{dy}{dx}-(x^2+\alpha^2) y = 0.\]

Its solution space (on $\{\Re z> 0\}$) is two-dimensional; we are only interested in the second standard solution $K_v$, which is %
characterised by the asymptotics 
\[K_v(z) \sim \sqrt{\frac{\pi}{2 z}}\;e^{-z}\]
(as defined in \cite{Weil}; see also \cite{DLMF}, §10.25).\footnote{Note that \cite{Kur77} uses a slightly different definition of the $K_v$, which is $\tfrac{2}{\pi}$ times our $K_v$.}

By \cite{Kur77}, we have %
				  $W_v^1\begin{pmatrix}
                                    x&0\\0&1
                                   \end{pmatrix}= \tfrac{2}{\pi} x^2 K_0(4\pi x)$.\\
($W_v^0$ and $W_v^2$ can also be described in term of Bessel functions; they are linearly dependent and scalar multiples of $x^2 K_1(4\pi x)$.%
)\\ %

By \cite{JL}, Ch. 1, Thm. 6.2(vi), $\sigma(|\cdot|_\CC^{1/2},|\cdot |_\CC^{-1/2})\iso \pi(\mu_1,\mu_2)$ with \[\mu_1(z)=z^{1/2}\bar{z}^{\; -1/2}=|z|_\CC^{-1/2} z,	\qquad	\mu_2(z)=z^{-1/2}\bar{z}^{1/2}=|z|_\CC^{-1/2}\bar{z};\] and the L-series of the representation is the product of the L-factors of these two characters:

\begin{eqnarray*} 
L_v(s, \pi_v)=L(s, \mu_1) L(s,\mu_2)&=&2\,(2\pi)^{-(s+\einh)}\Gamma(s+\einhalb)\cdot 2\,(2\pi)^{-(s+\einh)}\Gamma(s+\einhalb)\\
									&=&4\,(2\pi)^{-(2s+1)}\Gamma(s+\einhalb)^2.\\
\end{eqnarray*}

On the other hand, letting $d^\times x=\frac{dx}{|x|_\CC}=\frac{dr}{r}d\vartheta$ (for $x=re^{i\vartheta}$), %
we have  for ${\Re(s)>-\einhalb}$:

\begin{eqnarray*} 
 \int_{\CC^*} W_v^1\begin{pmatrix}
                     x&0\\0&1
                    \end{pmatrix}  |x|_\CC^{s-\einh} \; d^\times x %
								      &=&\int_{S^1}\int_{\RR_+} W_v^1\begin{pmatrix}
													r e^{i\vartheta}&0\\0&1
													\end{pmatrix}  |x|_\CC^{s-\einh} \; \dfrac{dr}{r} \; d\vartheta \\
						    &=&4 \int_0^\infty x^2 K_0(4\pi x) x^{2s-1} \; \frac{dx}{x} \\
& & \mspace{-64mu}\text{\small (invariance under $\SU(2)\cdot Z(F_v)$ gives a constant integral w.r.t. $\vartheta$)}\\
						    &=&4\, (4\pi)^{-2s+1}\int_0^\infty\! K_0(x) x^{2s} \;dx  \\
						    &=&4\, (4\pi)^{-2s+1} \; 2^{2s-1}\;{\Gamma(s+\einhalb)}^2 \qquad\mspace{11mu}\text{\small(by \cite{DLMF} 10.43.19)}\\ 		%
						    &=&4\, (2\pi)^{-2s+1}\;\Gamma (s+\einhalb)^2\\
\end{eqnarray*}

Thus we have %
\[\int_{\CC^*} W_v^1\begin{pmatrix}
                     x&0\\0&1
                    \end{pmatrix}  |x|_\CC^{s-\einh} \; d^\times x = (2\pi)^2\; L_v(s, \pi_v) \]  %

for all $\Re(s)>-\einhalb$.\\ %

We set $W_v:=(2\pi)^{-2}\;W_v^1$; thus \eqref{L=Zeta} holds also for complex primes.\\

Now that we have defined $W_v$ for all primes $v$, put $W^p(g):=\prod_{v\nmid p}W_v(g_v)$ for all $g=(g_v)_v\in G(\AA^p)$.

We will also need the vector-valued function $\underline{W}^p:G(\AA_F)\to V$ given by \[\underline{W}^p(g):=\prod_{v\nmid p \text{ finite or }v \text{ real}}W_v(g_v)\cdot\bigotimes_{v\text{ complex}}(2\pi)^{-2}\underline{W_v}(g_v).\]

\subsection{$p$-adic measures of automorphic forms}%
\label{interpol}

Now return to our $\pi\in\mathfrak{A}_0(G,\underline{2}, \ZC, \underline{\alpha_1},\underline{\alpha_2})$. 
We fix an additive character $\psi:\AA\to\CC^*$ which is trivial on $F$, and let $\psi_v$ denote the restriction of $\psi$ to $F_v\into\AA$, for all primes $v$. We further require that $\ker(\psi_\p)\supeq\OO_\p$ and $\p^{-1}\not\subeq \ker\psi_\p$ for all $\p|p$, so that we can apply the results of chapter \ref{local}.

As in chapter \ref{local}, let $\mu_{\pi_\p}:=\mu_{\alpha_{\p,1}/\muu_\p} = \mu_{q_\p/\alpha_{\p,2}}$ denote the distribution $\chi_{q_\p/\alpha_{\p,2}}(x)\psi_\p(x) dx$ on $F_\p$, and let $\mu_{\pi_p}:=\prod_{\p|p}\mu_{\pi_\p}$ be the product distribution on $F_p:=\prod_{\p|p}F_\p$.\\

Define  $\Ophi=\Ophi_\pi:\Co(F_{S_1}\times F_{S_2}^*)\times \II^p\to\CC$ by

\[\Ophi(U,x^p):=\sum_{\zeta\in F^*}\mu_{\pi_p}(\zeta U)W^p\begin{pmatrix}\zeta x^p&0\\0&1\end{pmatrix}. \]

By proposition \ref{Prop2.9}(a), we have for each $U\in\CCoo(F_{S_1}\times F_{S_2}^*)$:

\begin{eqnarray*}
 \Ophi(x_pU,x^p)&=&\sum_{\zeta\in F^*}\mu_{\pi_p}(\zeta x_pU)W^p\begin{pmatrix}\zeta x^p&0\\0&1\end{pmatrix}\\
	       &=&\sum_{\zeta\in F^*}W_U\begin{pmatrix}\zeta x_p&0\\0&1\end{pmatrix}W^p\begin{pmatrix}\zeta x^p&0\\0&1\end{pmatrix}\\
	       &=&\sum_{\zeta\in F^*}W\begin{pmatrix}\zeta x&0\\0&1\end{pmatrix},\\
\end{eqnarray*}
where $W(g):=W_U(g_p)W^p(g^p)$ lies in the global Whittaker model $\WW=\WW(\pi)$ for all $g=(g_p,g^p)\in G(\AA)$, putting $W_U:=W_{1_U}$; so $\Ophi$ is well-defined and lies in $\mathcal{D}(S_1, \CC)$ (since $W$ is smooth and rapidly decreasing; distribution property, $F^*$- and $U^{p,\infty}$-invariance being clear by the definitions of $\Ophi$ and $W^p$).\\

Let $\mu_\pi:=\mu_{\Ophi_\pi}$ be the distribution on $\GG_p$ corresponding to $\Ophi_\pi$, as defined in \eqref{47b}, %
and let $\kappa_\pi:=\kappa_{\Ophi_\pi}\in H^d({F^*}', \D_f(S_1, \CC))$ be the cohomology class defined by \eqref{''48''} and \eqref{(50)}.\\ %

\begin{theorem}\label{interpolation} %
 Let $\pi\in \mathfrak{A}_0(G,\underline{2},\ZC, \underline{\alpha_1},\underline{\alpha_2})$; %
we assume the $\alpha_{\p,i}$ to be ordered such that $|\alpha_{\p,1}|\le |\alpha_{\p,2}|$  for all $\p|p$.\footnote{So we have %
$\chi_{\p,1}=|\cdot|\chi_{\p,2}$ for all special $\pi_\p$.}  

(a) Let $\chi:\mathcal{G}_p\to\CC^*$ be a character of finite order with conductor $\f(\chi)$. Then we have the {\it interpolation property}
\[\int_{\mathcal{G}_p}\chi(\gamma)\mu_{\pi}(d\gamma)=\tau(\chi)\prod_{\p\in S_p} e(\pi_\p,\chi_\p)\cdot  L(\einhalb,\pi\tensor\chi)  ,\]
where 
\[\footnotesize e(\pi_\p,\chi_\p)= \begin{cases}
  \dfrac{(1-\alpha_{\p,1}x_\p q_\p^{-1})(1-\alpha_{\p,2}x_\p^{-1}q_\p^{-1})(1-\alpha_{\p,2} x_\p q_\p^{-1})}{(1-x_\p \alpha_{\p,2}^{-1})},  & \ord_\p(\f(\chi))=0\text{ and }\pi\text{ spherical,}\phantom{\Biggl(\Biggr)}\\   %

  \dfrac{(1-\alpha_{\p,1}x_\p q_\p^{-1})(1-\alpha_{\p,2}x_\p^{-1}q_\p^{-1})}{(1-x_\p\alpha_{\p,2}^{-1})},  & \ord_\p(\f(\chi))=0\text{ and }\pi\text{ special,}\phantom{\Biggl(\Biggr)}\\

  (\alpha_{\p,2}/q_\p)^{\ord_\p(\f(\chi))}%
,                      &\ord_\p(\f(\chi))>0
\end{cases} \]
and $x_\p:=\chi_\p(\omeg_\p)$.\\

(b) Let $U_p:=\prod_{\p|p}U_\p$, put $\phi_0:=(\phi_\pi)_{U_p}$. Then
\[ \int_{\II/F^*} \phi_0(x) d^\times x = \prod_{\p|p} e(\pi_\p,1)\cdot  L(\einhalb,\pi).	\]

(c) $\kappa_\pi$ is integral (cf. definition \ref{integral}). For $\underline{\mu}\in\Sigma$, let $\kappa_{\pi,\underline{\mu}}$ be the projection of $\kappa_\pi$ to $H^d({F^*}', \D_f(S_1, \CC))_{\pi,\underline{\mu}}$. Then $\kappa_{\pi,\underline{\mu}}$ is integral of rank $\le 1$.\\
\end{theorem}

\begin{proof}%
(a) We consider $\chi$ as a character on $\II_F/F^*$ (which is unitary and trivial on $\II_\infty$), and choose a subgroup $V\subeq U_p$ such that $\chi_p|_V=1$ (where $\chi_p:=\chi|_{F_p}%
$) and $V$ is a product of subgroups $V_{\p} \subeq U_\p$.

Let $W_V\in\WW_p$ be the product of the $W_{V_\p}$, as defined in prop. \ref{Prop2.9}, set $W(g):=W^p(g^p)W_V(g_p)\in\WW$, and
let \[\Ophi_V(x):=\Ophi(x_pV,x^p)=\sum_{\zeta\in F^*}W\begin{pmatrix}\zeta x^p&0\\0&1\end{pmatrix}.\]

Since $\pi$ is unitary, we have $|\alpha_{\p,2}|\ge \sqrt{q_\p}>1=|\chi_\p(\omeg_\p)|$ for all $\p$, thus ${e(\pi_\p,\chi_\p|\cdot|_\p^s)}$ is always %
non-singular, and we will be able to apply proposition \ref{Prop2.7} locally below.

We want to show that the equality
\[   [U_p:V]\int_{\II_F/F^*}\chi(x)|x|^s\Ophi_V(x)d^\times x=N(\f(\chi))^s\tau(\chi)\prod_{\p|p}e(\pi_\p,\chi_\p|\cdot|_\p^s)\cdot L(s+\einhalb,\pi\tensor\chi)\]
holds for $s=0$. %
Since both the left-hand side and $L(s+\einhalb,\pi\tensor\chi)$ are holomorphic in $s$ (see \cite{Gelbart}, Thm. 6.18 and its proof), %
it suffices to show this equality for $\Re(s)\gg 0$. 

For such $s$, we have
\begin{eqnarray*}
[U_p:V]\int_{\II_F/F^*}\chi(x)|x|^s\Ophi_V(x)d^\times x  &=& \int_{\II_F}\chi(x)|x|^sW\begin{pmatrix}x&0\\0&1\end{pmatrix}d^\times x\hspace*{11mm}\mbox{ (def. of  $\Ophi_V$)}
\end{eqnarray*}\\[-6ex]
\begin{eqnarray*}
\hspace*{10mm}&=& [U_p:V]\int_{F_p^*}\chi_p(x)|x|^sW_U\begin{pmatrix}x&0\\0&1\end{pmatrix}d^\times x\cdot
			\int_{\II_F^p}\chi^p(y)|y|^sW^p\begin{pmatrix}y&0\\0&1\end{pmatrix}d^\times y\\
\hspace*{10mm}&=&\prod_{\p|p}\int_{F_\p^*}\chi_\p(x)|x|_\p^s\mu_{\pi_\p}(dx)\cdot L_{S_p}(s+\einhalb,\pi\tensor\chi)\hspace*{12mm}\mbox{ (by prop. \ref{Prop2.9} and \eqref{L=Zeta})}\\
\hspace*{10mm}&=&\prod_{\p|p}\left(e(\pi_\p,\chi_\p|\cdot|_\p^s)\tau(\chi_\p|\cdot|_\p^s)\right)\cdot L(s+\einhalb,\pi\tensor\chi)\hspace{10mm}\mbox{ (by prop. \ref{Prop2.7})}\\
\hspace*{10mm}&=& N(\f(\chi))^s\tau(\chi)\prod_{\p|p}e(\pi_\p,\chi_\p|\cdot|_\p^s)\cdot L(s+\einhalb,\pi\tensor\chi).
\end{eqnarray*}
For $s=0$, we get the claimed statement, since by \eqref{47b} we have
\[\int_{\mathcal{G}_p}\chi(\gamma)\mu_{\pi}(d\gamma)=\int_{\II_F/F^*}\chi(x)\Ophi(dx_p,x^p)d^\times x^p=[U_p:V]\int_{\II_F/F^*}\chi(x)\Ophi_V(x)d^\times x.\]

(b) This follows immediately from (a), setting $\chi=1$, since $\tau(1)=1$.\\

(c) Let $\lambda_\alphaeinszwei\in \B^\alphaeinszwei(F_p,\CC)$ be the image of $\otimes_{v|p}\lambda_{a_v,\muu_v}$ under the map \eqref{Btensor}.  For each $\psi\in \B_\alphaeinszwei(F_p,\CC)$, define
\begin{eqnarray*}
 \langle\Phi_\pi,\psi \rangle (g^p,g_p)&:=&\sum_{\zeta\in F^*}\lambda_\alphaeinszwei \left(\begin{pmatrix}\zeta&0\\0&1 \end{pmatrix} g_p\cdot\psi\right)\underline{W}^p\left(\begin{pmatrix}\zeta&0\\0&1 \end{pmatrix} g^p\right) \\   
		      &=:& \sum_{\zeta\in F^*}\underline{W_\psi} \left(\begin{pmatrix}\zeta&0\\0&1 \end{pmatrix} g\right)
\end{eqnarray*}
for a $V$-valued function $\underline{W_\psi}$ whose every coordinate function is in $\WW(\pi)$.

This defines a map $\Phi_\pi:G(\AA^p)\to \B^{\underline{\alpha_1},\underline{\alpha_2}}(F_p, V)$. In fact, $\Phi_\pi$ lies in $S_2(G,\m, \alphaeinszwei)$, where $\m$ is the  prime-to-$p$ part of $\f(\pi)$: 

Condition (a) of definition \ref{S_2} %
follows from the fact that the $W_v$ are almost $K_0(\cc(\pi_v))$-invariant, for $v\nmid p,\infty$.

For condition (b), we check that $\langle\Phi_\pi,\psi\rangle$ satisfies the conditions (i)-(v) in the definition of $\A_0(G, \harm, \underline{2},\chi)$:

Each coordinate function of $\langle\Phi_\pi,\psi\rangle$ lies in (the underlying space of) $\pi$ by \cite{Bump}, Thm. 3.5.5, thus $\langle\Phi,\psi\rangle$ fulfills (i) and (v), and has moderate growth.
(ii) and (iv) follow from the choice of the $W_v$ and $\underline{W_v}$. %
Now since $\pi_v\iso \sigma(|\cdot|_v^{1/2},|\cdot |_v^{-1/2})$ for $v|\infty$, it follows from those conditions that $\langle\Phi,\psi\rangle|_{B'_{F_v}}\cdot \underline{\beta_v}=C \sum_{\zeta\in F^*} \underline{W_v}\left(\begin{smallmatrix}\zeta t&0\\0&1 \end{smallmatrix}\right)\cdot \underline{\beta_v}$ is harmonic for each archimedian place $v$ of $F$: for real $v$, it is well-known that $f(z)/y$ is holomorphic for $f\in \D(2)$, and thus $f\cdot (\beta_v)_1$ is harmonic; for complex $v$, this is also true, see e.g. \cite{Kur78}, p. 546 or \cite{Weil}.\\

Now we have 
\begin{eqnarray*}
\Delta^\alphaeinszwei (\Phi_\pi) (U,x^p)&=& \delta^\alphaeinszwei\left(\Phi\begin{pmatrix} x^p&0\\0&1 \end{pmatrix}\right)_{(1,\ldots,1)}(U) \\
	 &=& \sum_{\zeta\in F^*}\lambda_\alphaeinszwei\left(\begin{pmatrix} \zeta&0\\0&1 \end{pmatrix} \delta_\alphaeinszwei(1_U) \right) W^p\begin{pmatrix} \zeta x^p&0\\0&1 \end{pmatrix} \\
	 &\text{(*)}\atop=& \sum_{\zeta\in F^*}\mu_{\pi_p}(\zeta U) W^p\begin{pmatrix} \zeta x^p&0\\0&1 \end{pmatrix} = \Ophi_\pi (U, x^p),
\end{eqnarray*}
where (*) follows from the calculation (with $w_0$ as defined in Ch. 2)
\begin{eqnarray*}
 \lambda_\alphaeinszwei\left(\begin{pmatrix} \zeta&0\\0&1 \end{pmatrix} \delta_\alphaeinszwei(1_U) \right) 
	  &=& \prod_{\p|p}\int_{F_\p}\begin{pmatrix} \zeta&0\\0&1 \end{pmatrix} \delta_{\alpha_{\p,1},\alpha_{\p,2}}(1_U) \left(w_0{\small \begin{pmatrix} 1&x\\0&1 \end{pmatrix}}\!\right)\psi_\p(-x) dx\\
      &=& \prod_{\p|p}\int_{F_\p} \delta_{\alpha_{\p,1},\alpha_{\p,2}}(1_U)\Big(\!{\footnotesize \underbrace{\begin{pmatrix} 0&1\\-1&0 \end{pmatrix}\!\begin{pmatrix}1&x\\0&1 \end{pmatrix}\!\begin{pmatrix} \zeta^{-1}&0\\0&1 \end{pmatrix}}}\!\Big) \psi_\p(-x) dx\\
      & &\qquad\qquad\qquad\qquad\qquad\quad\;{\scriptsize =\left(\begin{matrix}0&1\\-\zeta^{-1}& -x \end{matrix}\right)} \\
      &=&\prod_{\p|p}\int_{F_\p} \chi_{\alpha_{\p,2}}(-x)\,\chi_{\alpha_{\p,1}}(-1)%
\; 1_U( -x \zeta)\;\psi_\p(-x) \; dx\\
       &=&\int_{%
\zeta U}\prod_{\p|p} \chi_{\alpha_{\p,2}}(-x)\psi_\p(-x)dx =\mu_{\pi_p}(\zeta U)
\end{eqnarray*}  %
for all $\zeta\in F^*$.\\

Let $R$ be the integral closure of $\ZZ[a_\p, \muu_\p%
; \p|p]$ in its field of fractions; thus $R$ is a Dedekind ring $\subeq\bar{\OO}$ for which $\B_{\alphaeinszwei}(F,R)$ is defined. $\CC$ is flat as an $R$-module (since torsion-free modules over a Dedekind ring are flat); thus by proposition \ref{flattensor}, the natural map
\[H^d(G(F)^+, \A_f(\m,\alphaeinszwei, \M, R))\otimes\CC \to  H^d(G(F)^+, \A_f(\m,\alphaeinszwei, \M,\CC)) \]

is an isomorphism. The map \eqref{(64)} can be described as the ''$R$-valued'' map  \[H^d(G(F)^+, \A_f(\m,\alphaeinszwei, \M, R))\to H^d\bigl({F^*}', \mathcal{D}_f(R)\bigr)\] tensored with $\CC$. By proposition \ref{square}, $\kappa_\pi$ lies in the image of \eqref{(64)}, and thus in  $H^d\bigl({F^*}', \mathcal{D}_f(R)\bigr)\otimes\CC$; i.e. it is integral.  %

Similarly, it follows from propositions \ref{flattensor} and \ref{harder} that $\kappa_{\pi,\underline{\mu}}$ is integral of rank $\le 1$.
\end{proof}

\begin{corollary}
 $\mu_\pi$ is  a $p$-adic measure.
\end{corollary}
\begin{proof}
 By proposition \ref{dreieck}, $\mu_\pi=\mu_{\phi_\pi}=\mu_{\kappa_\pi}$. Since $\kappa_\pi$ is integral, $\mu_{\kappa_\pi}$ is a $p$-adic measure by corollary \ref{p-adic measure}.
\end{proof}

\subsection{Vanishing order of the $p$-adic L-function}%
\label{result}
Let $L_1,\ldots,L_t$ be independent $\ZZ_p$-extensions of $F$, and let $\ell_1,\ldots,\ell_t:\GG_p\to p^{\epsilon_p}\ZZ_p$ be the homomorphisms corresponding to them (as in section \ref{triangle}). Then we have the %
{\it $p$-adic L-function}
\[L_p(\ul{s},\pi):=L_p(\ul{s}, \kappa_\pi):= L_p(s_1,\ldots,s_t,\kappa_{\pi,+}):= \int_{\GG_p}\prod_{i=1}^t\exp_p(s_i\ell_i(\gamma)) \mu_\pi(d\gamma)\]
of definition \ref{L-f}, with $s_1,\ldots ,s_t\in\ZZ_p$. $L_p(\ul{s},\pi)$ is a locally analytic function with values in the one-dimensional $\CC_p$-vector space $V_{\kappa_{\pi,+}}=L_{\kappa,\bar{\OO}, +}\otimes_{\bar{\OO}}\CC_p$. \\

By theorem \ref{exceptional}, we have

\begin{theorem}\label{YES}
$L_p(\ul{s},\pi)$ is a locally analytic ($t$-variabled) function%
, and all partial derivatives of order $\le\rr:=\#(S_1)$ vanish; i.e.  we have \[\ord_{\ul{s}=\ul{0}}L_p(\ul{s},\pi) \ge \rr.\\ \]
\end{theorem}

Now let $E$ be a modular elliptic curve over $F$, corresponding to an automorphic representation $\pi$; by this we mean that %
the local L-factors of the Hasse-Weil L-function $L(E,s)$ and of the automorphic L-function $L(s-\einhalb,\pi)$ coincide at %
all places $v$ of $F$. From the definition of the respective L-factors (cf. \cite{Si} for the Hasse-Weil L-function, \cite{Gelbart} for the automorphic L-function) we know that $\pi$ has trivial central character. Moreover, for $\p|p$, $\pi_\p$ is a principal series representation iff $E$ has good reduction at $\p$, and in this case $\pi_\p$ is ordinary iff $E$ is ordinary (i.e. not supersingular) at $\p$; $\pi_\p$ is a special (resp. Steinberg) representation iff $E$ has multiplicative (resp. split multiplicative) reduction at $\p$. For $v|\infty$, $\pi_v$ is ``of weight 2'' as assumed before. \\

We say that $E$ is {\it $p$-ordinary} if it has good ordinary or multiplicative reduction at all places $\p|p$ of $F$. So $E$ is $p$-ordinary iff $\pi$ is ordinary at all $\p|p$. %
In this case, we define the %
{$p$-\nobreak adic} L-function of $E$ by $L_p(E,\ul{s}):=L_p(\ul{s},\pi)$.\\ %

For each $i\in\{1,\ldots,t\}$ and each prime $\p|p$ of $F$, we write $\ell_{\p,i}$ for the restriction of $\ell_i$ to $F_{\p}\into\II\onto \GG_p$. Let $q_{\p}$ be the Tate period of ${E|F_{\p}}$ and $\ord_{\p}$ the normalized valuation on $F_{\p}^*$. %
We define the {\it L-invariants} of $E|F_\p$ with respect to $L_i$ by %
\[\LL_{\p,i}(E):=\ell_{\p,i} (q_\p)/\ord_\p(q_\p)\]
Then we can generalize Hida's exceptional zero conjecture to general number fields:
\begin{conj}
Let $S_1$ be the set of $\p|p$ at which $E$ has split multiplicative reduction, $\rr:=\#S_1$, $S_2:=S_p\setminus S_1$. Then 
 \begin{equation}\label{1'}
 \ord_{\ul{s}=\ul{0}}L_p(E,\ul{s})\ge \rr, 
\end{equation}
and we have
\begin{equation}\label{2}
\frac{\dd^\rr}{\dd s_i^\rr}L_p(E,\ul{s})|_{\ul{s}=0}=\rr! \prod_{\p\in S_1}\LL_{\p,i}(E)  \prod_{\p\in S_2} e(\pi_\p, 1) \cdot L(E,1),
\end{equation}
for all $i=1,\ldots, t$, where $e(\pi_\p,1)=(1-{\alpha_{\p,1}}^{-1})^2$ if $E$ has good ordinary reduction at $\p$, and $e(\pi_\p,1)=2$ if $E$ has (non-split) multiplicative reduction at $\p$.
\end{conj}

Note that the conjecture (considered for all sets of independent $\ZZ_p$-extensions of $F$) also determines the ``mixed'' partial derivatives $\frac{\dd^k}{\dd^{\ul{n}}\ul{s}} L_p(E, \ul{0})$ of order $n$, since they can be written as $\QQ$-linear combinations of $n$-th ``pure'' partial derivatives  $\frac{\dd^n}{\dd {s'}_i^n}L_p(E,\ul{0})$ with respect to other choices of independent $\ZZ_p$-extensions of $F$ by remark \ref{mixed}.\\

Theorem \ref{YES} immediately implies %
the first part \eqref{1'} of the conjecture: %

\begin{corollary} 
Let $E$ be a $p$-ordinary modular elliptic curve over $F$.  Let $\rr$ be the number of places $\p|p$ at which $E$ has split multiplicative reduction. Then %
we have \[ \ord_{\ul{s}=\ul{0}}L_p(E,\ul{s})\ge \rr.\] 
\end{corollary}
\qed

\clearpage
\addcontentsline{toc}{section}{References}
\bibliographystyle{alpha}

\pagestyle{plain}
\printindex

\end{document}